\documentclass[11pt,letterpaper]{amsart}

\usepackage{fullpage} 
\usepackage{amsmath}  
\usepackage{amssymb}  
\usepackage{amscd}    
\usepackage{bm}       
\usepackage{bbm}      
\usepackage{enumerate}
\usepackage{mathrsfs} 
\usepackage{comment}  
\usepackage{cite}     
\usepackage{graphicx} 
\usepackage{xfrac}    
\usepackage{cancel}   
\usepackage{calc}     
\usepackage{multirow}            
\usepackage{caption,subcaption}  
\usepackage[utf8]{inputenc}     
\usepackage[english]{babel}		
\usepackage[dvipsnames]{xcolor} 
\usepackage[colorlinks=true,allbordercolors=white, citecolor=blue]{hyperref}
\usepackage[cmtip,color,all]{xy}
\xyoption{line}

\usepackage{pgfplots}
\usepgfplotslibrary{fillbetween}
\pgfplotsset{compat=newest}

\usepackage{tikz}
\usetikzlibrary{calc, shadows.blur}
\usepackage{tikz-3dplot}   
\usetikzlibrary{cd}        
\usetikzlibrary{plotmarks}
\usetikzlibrary{calc}
\usetikzlibrary{matrix}
\usetikzlibrary{positioning}
\usetikzlibrary{backgrounds}
\usetikzlibrary{intersections}
\usetikzlibrary{fillbetween}
\usetikzlibrary{knots}
\usetikzlibrary{arrows, arrows.meta}
\usetikzlibrary{patterns, patterns.meta}
\usetikzlibrary{shapes, shapes.misc}
\usetikzlibrary{decorations.pathreplacing, decorations.markings, decorations.pathmorphing}

\definecolor{ringblue}{HTML}{0088FF}
\definecolor{diskblue}{HTML}{005588}
\definecolor{ringorange}{HTML}{FF8800}
\definecolor{diskorange}{HTML}{AA4400}
\definecolor{connector}{HTML}{BB00BB}

\makeatletter
\pgfarrowsdeclare{ang 90}{ang 90}
{
  \pgfutil@tempdima=.3pt%
  \advance\pgfutil@tempdima by .25\pgflinewidth%
  \pgfutil@tempdimb=7.29\pgfutil@tempdima\advance\pgfutil@tempdimb by.5\pgflinewidth%
  \pgfarrowsleftextend{+-\pgfutil@tempdimb}
  \pgfutil@tempdimb=.5\pgfutil@tempdima\advance\pgfutil@tempdimb by1.6\pgflinewidth%
  \pgfarrowsrightextend{+\pgfutil@tempdimb}
}
{
  \pgfutil@tempdima=0.005pt%
  \advance\pgfutil@tempdima by.25\pgflinewidth%
  \pgfsetdash{}{+0pt}
  \pgfsetroundcap
  \pgfsetmiterjoin
  \pgfpathmoveto{\pgfpointadd{\pgfqpoint{0.5\pgfutil@tempdima}{0pt}}{\pgfqpointpolar{135}{18\pgfutil@tempdima}}}
  \pgfpathlineto{\pgfqpoint{0.5\pgfutil@tempdima}{0pt}}
  \pgfpathlineto{\pgfpointadd{\pgfqpoint{0.5\pgfutil@tempdima}{0pt}}{\pgfqpointpolar{-135}{18\pgfutil@tempdima}}}
  \pgfusepathqstroke
}
\pgfarrowsdeclarereversed{ang 90 reversed}{ang 90 reversed}{ang 90}{ang 90}

\tikzset{
    dot diameter/.store in=\dot@diameter,
    dot diameter=3pt,
    dot spacing/.store in=\dot@spacing,
    dot spacing=10pt,
    dots/.style={
        line width=\dot@diameter,
        line cap=rect,
        dash pattern=on 4.2pt off \dot@spacing
    }
}
\makeatother

\usepackage{amsthm}   
\theoremstyle{plain}
\newtheorem{theorem}{Theorem}[section]
\newtheorem{lemma}[theorem]{Lemma}
\newtheorem{proposition}[theorem]{Proposition}
\newtheorem{corollary}[theorem]{Corollary}

\theoremstyle{definition}
\newtheorem{definition}[theorem]{Definition}
\newtheorem{example}[theorem]{Example}

\newtheorem{convention}[theorem]{Convention}
\newtheorem{question}[theorem]{Question}

\theoremstyle{remark}
\newtheorem{remark}[theorem]{Remark}

\newcommand{\Z}{\mathbb{Z}}
\newcommand{\C}{\mathbb{C}}

\newcommand{\R}{\mathbb{R}}
\newcommand{\Q}{\mathbb{Q}}

\newcommand{\E}{\mathbb{E}}

\renewcommand{\L}{\mathbb{L}}   

\newcommand{\Hom}{\text{Hom}}
\newcommand{\id}{\mathrm{id}}

\newcommand{\ad}{\mathrm{ad}}

\newcommand{\tensor}{\otimes}

\newcommand{\br}{\mathrm{BrSym}}

\newcommand{\into}{\hookrightarrow}
\newcommand{\onto}{\twoheadrightarrow}

\renewcommand{\H}{\mathcal{H}} 

\newcommand{\B}{\mathcal{B}}

\newcommand{\desus}{s^{-1}}
\newcommand{\eil}{\mathcal{E}\!{\mathit il\/}}  
\newcommand{\dual}[1]{({#1})^{\#}}
\newcommand{\contdual}[1]{({#1})^{\#^{c}}}
\newcommand{\compgring}[1]{\widehat{\Q}[{#1}]}
\newcommand{\Free}[1]{\mathbb{F}_{#1}}
\newcommand{\BCH}{\mathrm{BCH}}
\newcommand{\Mal}[2][]{\mathrm{Mal}#1({#2}#1)}  

\newcommand{\mE}{\mathcal{E}}
\newcommand{\mB}{\mathcal{B}}
\newcommand{\HHarr}[3][]{H^{#2}_{\mE} #1( {#3} #1)} 
\newcommand{\holiebar}[2]{\HHarr{#1}{#2}}
\newcommand{\liebar}[2]{\mE^{#2}({#1})}
\newcommand{\dHar}{d_\mE}

\newcommand{\lcs}[1]{\mathrm{lcs}({#1})}
\newcommand{\apl}{\operatorname{A}_{PL}}
\newcommand{\gr}{\mathrm{gr}~}

\newcommand{\Gr}{\mathcal{G}\!{\mathit r\/}}  
\newcommand{\Tr}{\mathcal{T}\!{\mathit r\/}}  
\newcommand{\Eil}{\eil}  
\newcommand{\Lie}{\mathcal{L}\!{\mathit ie\/}}    
\newcommand{\cofreeE}{{{\mathbb E}}}
\newcommand{\vscofreeE}{\cofreeE}
\newcommand{\inv}{ {{\text -}1}} 
\newcommand{\sym}{\mathcal S}

\usepackage{mathtools}
\DeclarePairedDelimiter\tp{\lfloor}{\rfloor\,} 
\DeclarePairedDelimiter\tpc{\lfloor}{\rfloor}  
\DeclarePairedDelimiterXPP\lpw[1]{}{\lfloor}{\rfloor}{{_{\!_w}}}%
{\hspace{-.27em}\delimsize\lfloor\mathopen{}#1\delimsize\rfloor\hspace{-.27em}}
\DeclarePairedDelimiterX\lp[1]{\lfloor}{\rfloor}%
{\hspace{-.27em}\delimsize\lfloor\mathopen{}#1\delimsize\rfloor\hspace{-.27em}}
\DeclarePairedDelimiter\cobr{]}{[}%

\newcommand{\linep}[2]{ \ensuremath{  %
 \tpc*{#1}{#2}
} }
\newcommand{\longlinep}[3]{ \ensuremath{  %
 \tp*{#1}{#2}
} }
\newcommand{\graphpp}[3]{ \ensuremath{
  \tp[\big]{\tpc{#1}{#2}}{#3}
} }

\newcommand{\maps}{\colon}
\newcommand{\st}{~\big|~}

\newcommand{\lew}{\mathbin{
 \raisebox{1pt}{$<$}\hspace{-7pt}\raisebox{-2.5pt}{$\scriptscriptstyle w$}\,
}}
\newcommand{\lewsm}{\mathbin{
  \,\raisebox{.5pt}{$\scriptstyle <$}\hspace{-6pt}\raisebox{-2.5pt}{$\scriptscriptstyle w$} \hspace{2pt}
}}

\makeatletter
\def\@tocline#1#2#3#4#5#6#7{\relax  
  \ifnum #1>\c@tocdepth             
  \else
    \par \addpenalty\@secpenalty\addvspace{#2} 
    \begingroup \hyphenpenalty\@M
    \@ifempty{#4}{%
      \@tempdima\csname r@tocindent\number#1\endcsname\relax
    }{%
      \@tempdima#4\relax
    }%
    \parindent\z@ \leftskip#3\relax \advance\leftskip\@tempdima\relax 
    \rightskip\@pnumwidth plus4em \parfillskip-\@pnumwidth   
    #5\leavevmode\hskip-\@tempdima
      \ifcase #1
       \or\or \hskip 1em \or \hskip 2em \else \hskip 3em \fi%
      #6\nobreak\relax
    \hfill\hbox to\@pnumwidth{\@tocpagenum{#7}}\par
    \nobreak
    \endgroup
  \fi}
\makeatother

\begin{document}

\author[N. Gadish]{Nir Gadish}
\address{Mathematics Department, 
University of Pennsylvania}
\email{ngadish@math.upenn.edu}
\author[A. Ozbek]{Aydin Ozbek}
\address{Mathematics Department, 
University of Oregon}
\email{ozbekriordan@gmail.com}
\author[D. Sinha]{Dev Sinha}
\address{Mathematics Department, 
University of Oregon}
\email{dps@uoregon.edu}
\author[B. Walter]{Ben Walter}
\address{Mathematics Department, 
University of the Virgin Islands}
\email{benjamin.walter@uvi.edu}

\keywords{Rational homotopy theory, fundamental groups, Malcev Lie algebras, Lie coalgebraic duals of groups, Harrison and bar cohomology, letter linking and braiding, Hopf invariants of loops}
\subjclass[2020]{
20F40,
55P62,
57M07,
57M05,
55Q25,
55S30}

\noindent
\title{
  Infinitesimal calculations in fundamental groups
  }

  \begin{abstract}
      We show that Hopf invariants, defined by evaluation
      in Harrison cohomology of the commutative
      cochains of a space, calculate the logarithm
      map from a fundamental group to its Malcev Lie algebra. They thus present the zeroth Harrison cohomology as a universal dual object to the Malcev Lie algebra.
      This structural theorem supports  explicit calculations
      in algebraic topology, geometric topology, and combinatorial group theory.  In particular, we give the first algorithm to determine whether a power of a word is a \(k\)-fold nested commutator while encoding commutator structure in any group presented by generators and relations.
  \end{abstract}

\maketitle

  \begin{center}
      {\em To the memory and loved ones of the second author.}
  \end{center}


\section{Introduction}\label{intro}

This work revisits a classical question: what is the capacity of rational or real cochains to ``measure'' elements of a fundamental group?
Rational homotopy theory and Chen's iterated integrals are two classical approaches to this question, but they have not been utilized for example to develop  concrete algorithms, in combinatorial group theory. 

While previous general computational approaches to rational homotopy theory afford abstract calculation \emph{of} the rational completion of a group, we build a toolkit to calculate \emph{in} this completed group, a distinction akin to that between the Isomorphism Problem \emph{of} groups and the Word Problem \emph{in} groups. In this work we revisit and recast classical tools into a form that is well tailored for algebraically presented groups and further interlaces the algebra of groups with intersection theory in manifolds.

Measurement is made through the logarithm map from a group to its Malcev Lie algebra, which occurs classically in the context of the rationalization of a space \cite{Quillen-rational} (see \cite{Ivanov-nonsimplyConnRationalization} and \cite{LieEncyclopedia}). Our resulting theory is effective: we give algorithms for calculations in general groups, which for example produce all Vassiliev invariants of pure braids.

Realizing a group $\Gamma$ as the fundamental group of a pointed topological space $X$,
we produce homotopy invariants of loops, which yield well-defined functions on $\Gamma$. Recall that (co)homology of loop spaces has been studied using the \emph{bar and cobar constructions} since the 1950's, starting with work of Adams, Hilton, Eilenberg and Moore.  These are complexes
built from tensor powers of cochains (or differential forms), computing cohomology of loop spaces in many cases. Chen's iterated integrals provide a concrete formula for these invariants of loops using differential forms on manifolds: given $1$-forms $\omega_1,\ldots,\omega_n$ on a smooth manifold $X$, and a path $s\maps [0,1]\to X$, set
\[
\int_s (\omega_1|\omega_2|\ldots| \omega_n) := \int_{0<t_1<\ldots<t_n<1} f_1(t_1)f_2(t_2)\ldots f_n(t_n) dt_1\ldots dt_n
\]
where $f_i(t)dt = s^*(\omega_i)$.
Chen \cite{chen1971pi1} shows that these integrals detect the image of a loop $[s]\in\pi_1(X,x)$ in the rational completion of the fundamental group.

Following ideas of Hain \cite{Hain-LinkGroups}, 
consider the case where each form $\omega_i$ is supported near a hypersurface $M_i\subseteq X$ and counts intersection with it, which are known as \emph{Thom forms}. In this case,
the Chen integral  counts the number of times the path $s$ meets the hypersurfaces $(M_1,\ldots,M_n)$ \emph{in this order}. 
Specializing $X$ to the complement of a link in $S^3$, Hain used this approach to give a geometric proof for Milnor's presentation of link groups $\pi_1(X,x)$ after unipotent completion. 

We build upon and combinatorialize this idea in two ways. {First}, we work with algebraic presentations of the group $\Gamma$ via generators and relations. This algebra connects with the above topological story through an explicit manifold equipped with $1$-forms that count appearances of particular group generators in a word. We thus reduce Chen integrals into what is effectively an entirely algebraic and combinatorial theory, but which still retains the tight relationship with intersection theory.

Our algebraic perspective has one consider ordered counts of generators $a_1\ldots,a_n\in \Gamma$ as they appear in words $w(a_1,\ldots,a_n)\in \Gamma$, which we call \emph{letter braiding functions}: e.g. $(A_1\to A_2)$ counts the (signed) number occurrences of $(\ldots a_1^{\pm} \ldots a_2^{\pm}\ldots )$ in words such as 
$$ \bigl[a_1^{-1}a_2a_1,\;a_2a_1a_2^{-1}\bigr] = a_1^{-1}a_2a_1a_2a_1a_2^{-1}a_1^{-1}a_2^{-1}a_1a_2a_1^{-1}a_2^{-1}$$
on which $(A_1\to A_2)=1$, see the exposition in Subsection \ref{combalg} below.
In the free group setting, this could have been worked out fifty years ago, if not earlier, but to our knowledge it was not. Our functions generalize the Magnus expansion\footnote{In fact, we naturally arrive at a rational variant of the Magnus expansion: the \textit{exponential expansion} which is defined on generators by $a_i^{\pm}\mapsto exp(\pm A_i) = 1\pm A_i+\frac{1}{2}A_i^2\pm \frac{1}{3!}A_i^3+\ldots$ and extended multiplicatively.} and Fox calculus to all groups, though they only
``see'' a group through
the commutator filtration
of its Malcev Lie algebra. 
These counts serve as obstructions to some power of a word lying in a given $k$-fold
commutator subgroup, as stated in Theorem~\ref{thm:into-combinatorial-alg} and proven at the end of Section~\ref{structural}. 
These functions have been implemented in Python and SageMath by the second and fourth authors. 
A fundamental question we then answer is the following.
\begin{question}
    Which of these ordered letter counts are well-defined in the presence of relations?
\end{question}
One can set up linear equations in search of linear combinations of counts 
that vanish on all products of conjugates of relations, but that would not lead to a terminating algorithm, even when word-length is restricted. Instead, we give a recursive, finite procedure in Section \ref{Algorithms}, 
which is best understood in the
context of rational homotopy theory.

To make computations manageable, we employ the Lie coalgebraic approach to rational homotopy theory \cite{Sinha-Walter1,sinha-walter2}.  From the perspective of Chen integrals, 
this eliminates tautological relations such as $\int (w_1|w_2)+\int (w_2|w_1) = (\int w_1)(\int w_2)$.
It has long been known as a basic tenet of Koszul--Moore duality that the indecomposable elements in the bar complex of a commutative algebra form a Lie coalgebra. In a relatively recent development, the last two authors used Lie coalgebras to model rational homotopy theory in \cite{Sinha-Walter1}, facilitating calculations as well as a geometric characterization of Hopf invariants
for higher homotopy groups. Restricting to indecomposables retains all information about the space of study, while also substantially reducing the dimensions one needs to consider in calculations.

Centering this coalgebraic model in our group theoretic pursuit, we describe a Lie cobracket structure on the collection of letter braiding functions. Our Theorem~\ref{LiftingCriterion} then characterizes these functions as follows (stated imprecisely here -- see \ref{sec:lifting} for details).
\begin{theorem}
    The letter braiding functions that
are well-defined on a presented group $\Gamma = \bigl\langle S\,\big|\,R\bigr\rangle$ form the \emph{maximal Lie coalgebra of functions on words in $S$ that vanish on $R$}, and they pair perfectly with the Malcev Lie algebra of $\Gamma$ when finitely generated.
\end{theorem}
Thus, we exchange the need to check vanishing on all products of conjugates of relations for a constraint on the Lie cobracket, which is a finite-dimensional linear algebraic condition and thus readily computable.

While the application to combinatorial group theory is our broadest and most-refined, in this paper we also circle back to geometric topology and show that the ideas developed here can yield further insight through ad-hoc analysis.  For an illustrative example, we characterize classes of ``linear and quadratic'' functions on pure braid groups:
\begin{itemize}
    \item linear functions count crossings of pairs of strands, while 
    \item quadratic functions are sensitive not only to the internal ordering within pairs of crossings but also to crossings that occur ``to the right'' of a third strand. See Example \ref{purebraid} below.
\end{itemize}

One can abstractly formalize the approach of our work in the following simple and very general terms, generalizing the process of evaluating a cohomology class on the fundamental class of a manifold.
\begin{definition}\label{HopfFormalism}
Let $C^*$ be a type of cochains (a functor from topological spaces to chain complexes with some additional structure), such as singular or de\,Rham, suppose $F$ is a homotopy invariant functor on such cochain objects, and let $H_F^*(X)$ denote the cohomology of $F(C^*(X))$.
Suppose one fixes a map $\int_{S^1}$ from 
$H_F^*(S^1)$ to some abelian group.  
Now, for any $\tau \in H_F^*(X)$ we define the {\bf $C$-$F$-valued
Hopf invariant} associated
to $\tau$ to be the function on
$\pi_1(X)$
sending $f: S^1 \to X$ to $\int_{S^1} f^* \tau$.  
\end{definition}
Our main theorems below imply that the
geometric intersection counts and algebraic letter linking counts mentioned above are
de\,Rham-Harrison-valued Hopf invariants -- the cochains $C^*$ are the de\,Rham complex and $F$ is the Harrison complex of a commutative dg-algebra (see Subsection \ref{categorical}).
Moreover, de\,Rham-Harrison-valued Hopf invariants pair universally, and thus perfectly, with the Malcev Lie algebra of a group.

To help weave together perspectives from algebraic topology, geometric topology, and combinatorial algebra for readers with different backgrounds,  
we have written further introductory sections, which can be read independently and in any order, with
readers encouraged to explore their favorite ``side of the elephant''\footnote{\url{https://tinyurl.com/t7xdssnu}}. In Section~\ref{categorical}, we give the formal framework for results, which comes from algebraic topology. In Section~\ref{combalg}, we elaborate our first applications in combinatorial group theory. In Section~\ref{geomtop}, we share basic motivating examples which come from geometric topology.

To serve readers of different backgrounds, we offer the following guide for a first pass through the paper.  We hope this is in particular helpful for readers not experienced with topics such as 
the bar complex, Thom forms, or the Malcev Lie algebra, the first two of which are
elaborated in the Appendix. 
We suggest that such a reader first take their time with Sections~\ref{combalg}~and~\ref{geomtop}. After that, they could go to the first part of  appendix~\ref{basicbar} to learn the basics of the
bar construction. From there, a reader who is familiar with de\,Rham theory \cite{BottTu} could review properties of Thom forms in Theorem~\ref{Thom}. (If a reader does not have experience with de\,Rham cohomology, they could still follow if they are willing to believe that submanifolds can define cochains with `pullback' being preimage, `product' being intersection, and `coboundary' being boundary.)
With the bar complex and Thom forms on hand, a reader
is prepared for the critical central example of Section~\ref{FreeCase}, which shows explicitly how the formal framework of ``pull-back, 
weight-reduce, evaluate'' leads to
counting intersections of a loop with submanifolds,
accounting for order and orientation.  
For handlebodies, this in turn leads to letter braiding
as introduced in Section~\ref{combalg}, establishing our theory for free groups. From there, a reader could go to the Examples of Section~\ref{examples}.
Finally, a first pass could end with
a view of the Lifting Criterion in 
Section~\ref{sec:lifting}, as
motivation for the Lie coalgebraic formulation we employ.

\medskip

    {We would like to thank Greg Friedman, who provided valuable feedback, as well as the family of the second
author, who shared some of his private work with us
after his untimely passing. We also thank Richard Hain for offering his perspective and helping us understand the context in which our work resides.}

\tableofcontents

\subsection{Infinitesimal group theory and rational homotopy}\label{categorical}

We first recall basic facts about  Malcev Lie algebras.
The {\bf Baker-Campbell-Hausdorff expansion} 
formally substitutes the product of power
series for $e^x$ and $e^y$ into the series for $\log(1+x)$ to 
express the logarithm of a product of exponentials in a noncommutative algebra as
\begin{align*}
\BCH(x,y) &:= \log(e^x e^y) = \log(1 + (e^x e^y - 1)) \\ 
          &= x+y+\frac{1}{2}[x,y]+\frac{1}{12}\bigl([x,[x,y]] + [y, [y,x]]\bigr) -\frac{1}{24}[x,[y,[x,y]]] + \cdots,
\end{align*}
where $[a,b] = ab - ba$.

The Malcev Lie algebra of a group $G$ is a complete Lie algebra $\Mal G$ equipped with a function $\log\maps G\to \Mal G$ satisfying the {\bf BCH equality}
\begin{align*}
\log(\gamma \cdot \gamma') &= \BCH\bigl(\log(\gamma),\log(\gamma')\bigr) \\ 
&= \log(\gamma)+\log(\gamma') + \frac{1}{2}[\log(\gamma),\log(\gamma')]+\cdots,
\end{align*}
and having the universal property that any function $G\to L$ into a {nilpotent} Lie algebra $L$ satisfying the BCH equality factors uniquely through a Lie algebra homomorphism $\Mal G\to L$. The Malcev Lie algebra
is thus the completed Lie algebra topologically generated by the elements $\log(\gamma)$, subject only to the relations coming from the BCH equality.
Quillen in \cite{Quillen-rational} presented $\Mal G$ as the set of primitive elements in the rational group ring 
completed at the ideal generated by $(\gamma-1)$ for all $\gamma\in G$, with $\log(\gamma) = \sum \frac{(-1)^{n+1}}{n}(\gamma-1)^n$.

In this paper we carefully develop and apply a new
dual notion to $\Mal{G}$.
In the finitely generated setting, functions on a Lie algebra inherit the dual structure of a Lie coalgebra, which can be briefly understood as a vector space $E$ equipped with a cobracket $\cobr{\cdot} :  E\to E\otimes E$ such that the dual space $E^\#$ naturally forms a Lie algebra. We develop cofree Lie coalgebras more fully in Section \ref{sec:prelims}.

\begin{definition}\label{LiePairing}
A $\Q$-valued pairing $\langle -, - \rangle_E$ of a Lie coalgebra $(E,\,\cobr{\cdot}\;)$ 
with the Malcev Lie algebra $\Mal G$ is
a {\bf Lie pairing} if
\begin{equation*}
\bigl\langle \,\cobr{\tau}\,, a\otimes b \bigr\rangle_E = \bigl\langle \tau, [a,b] \bigr\rangle_E,
\end{equation*}
where on the left  $E^{\otimes 2}$ pairs with 
$\Mal{G}^{\otimes 2}$ by $\langle \alpha\otimes \beta , a\otimes b\rangle_E = \langle \alpha,a\rangle_E \cdot \langle \beta,b\rangle_E$.

Such a pairing is {\bf universal} if, for any other Lie pairing $\langle -, - \rangle_{E'}$, there exists a unique homomorphism $\varphi : E' \to E$ of Lie coalgebras 
such that  $$\bigl\langle \,{\tau}, a \,\bigr\rangle_{
E'} = \bigl\langle \, \varphi(\tau),  a\, \bigr\rangle_E.$$

A Lie coalgebra equipped with a universal Lie pairing with $\Mal G$ is called the {\bf Lie dual} of $G$. It is unique up to unique isomorphism.
\end{definition}

\color{black}



Lie pairings extend the notion of 
homomorphisms from $G$ to the rational numbers. Indeed, if $\tau\in E$ satisfies $\cobr{\tau}=0$ then the pairing with $\tau$ automatically vanishes on all commutators in $\Mal{G}$. Thus 
\begin{equation}\label{wtzeropairing}
\bigl\langle \,{\tau}, \log(ab) \,\bigr\rangle_{E}
 = \bigl\langle \,{\tau}, \log(a) + \log(b) \,\bigr\rangle_{E}
 = \bigl\langle \,{\tau}, \log(a) \,\bigr\rangle_{E}
 + \bigl\langle \,{\tau}, \log(b) \,\bigr\rangle_{E},
\end{equation}
so the Lie pairing with $\tau$ defines a rational functional on the abelianization of $G$. The kernel of the cobracket of a Lie coalgebra is called the {weight one subspace} 
(see Proposition~\ref{cofree conilpotent}).
By Equation~(\ref{wtzeropairing}), a Lie pairing restricts to a bilinear pairing between the weight one subspace of $E$ and the abelianization
of $G$. In other words, a
Lie pairing naturally extends homomorphisms $G\to \Q$, 
and as we show below, the Lie dual of $G$ has weight one subspace canonically isomorphic to $G^\# := \Hom_{Grp}(G, \Q)$.

When $G^\#$ is finite dimensional, a universal Lie pairing on $E$ defines an isomorphism of Lie algebras $E^\# \cong \Mal G$. But when $G^\#$ is infinite dimensional, 
the Lie dual is distinct from the standard linear dual of $\Mal{G}$, which is no longer in general a Lie coalgebra, since the putative
cobracket could have image involving 
infinite sums.

Lie coalgebras arise  in studying commutative differential graded algebras, and in particular
commutative cochain complexes.   Given a commutative augmented differential graded algebra $A$, Harrison constructed a cochain complex $\liebar{A}{*}$ computing its derived indecomposables (see Definition~\ref{D:E} and Theorem~\ref{HarrisonDerivedIndec}). 
These derived indecomposables
naturally form a Lie coalgebra, the central example for operadic Koszul-Moore duality \cite{Koszul-Moore-survey}. 

Rational homotopy theory relates the Harrison cohomology of a commutative cochain model for a topological space $X$ with the linear dual of rational homotopy groups of $X$, matching the Lie coalgebraic structures on both sides. 
The third and fourth authors made this isomorphism explicit
and geometric for higher homotopy groups in the simply connected setting in \cite{sinha-walter2}. 
The present work extends this to fundamental groups.  

Let $X$ be a pointed simplicial set or  differentiable manifold  and $A^*(X)$ the de\,Rham (PL or smooth, respectively) cochains  on $X$.
Following Definition~\ref{HopfFormalism}, we first show in Proposition~\ref{Ground} that there exists a canonical isomorphism $\int_{S^1}\maps \HHarr[\big]{0}{A^*(S^1)} \xrightarrow{\sim} \Q$. 
This canonical isomorphism gives rise to the Hopf-invariant pairing between $\HHarr[\big]{0}{A^*(X)}$ and $\pi_1(X)$ via pull-back and evaluation: 
with the value of a Harrison cocycle $\tau$ on some $f: S^1 \to X$  given by $\int_{S^1} f^* \tau$.

Recall that because $\log$-series of group elements topologically generate the Malcev Lie algebra, a pairing defined on those extends uniquely to the entire algebra.  Our main structural result is the following.   
    
\begin{theorem}\label{thm-intro:universal property}
    The Hopf invariant pairing defined by
    \[
    \bigl\langle \tau,\, \log(\gamma) \bigr\rangle = \int_{S^1} \gamma^*\tau
    \]
    determines a universal Lie pairing of  $\HHarr[\big]{0}{A^*(X)}$ with $\pi_1(X)$.
\end{theorem}

Thus the zeroth Harrison cohomology of $X$ is the natural
Lie extension of the evaluation of cohomology on homotopy,
where     $\pi_1(X)^\# \cong H^1(X;\Q)$.
While conceptually satisfying, we see the true value of this theorem in that the right hand side can be explicitly analyzed.  Whenever a fundamental group is tied to a space of interest, or such a space can otherwise be constructed,  Harrison cohomology provides  a means for calculations in its Malcev Lie algebra.  We provide 
 applications in algebra and in geometric topology, as explained next.

\begin{remark}
    The Harrison complex is central in the theory of $C_\infty$-algebras,
and our theory extends there.  
In Remark~\ref{C-infty} we  outline  progress the second author
made along these lines before his untimely passing.  But we leave it to experts to deduce the straightforward generalization to the $C_\infty$ 
setting.  
\end{remark}
 
\subsection{Combinatorial group theory}\label{combalg}

Our framework leads to a remarkably elementary combinatorial approach to understanding the radical of the lower central series of presented groups in terms of combinatorial  ``letter braiding'' counts 
on words.  
We encourage readers
to regularly refer ahead to Example~\ref{BraidingProductExample} to see 
definitions illustrated as we make them.

\begin{definition}
    An {\bf $[n]$-tuple} is a function from $[n] = \{1, 2, \ldots, n\}$ to an abelian group, here exclusively the additive group of the rational numbers.
If $f$ is an $[n]$-tuple,  define $\int f$ to be the sum $\Sigma_{i=1}^n \ f(i)$. 
\end{definition}

Our slight embellishment of notation is a reminder that
these tuples should be viewed
as functions on $[n]$, in
which case integration is a natural operation.  Informally, these will be functions on the letters
of a word.

\begin{definition}
Let $F_S=\langle S\rangle$ be the free group generated by a set $S$,
and consider a length \(n\) word $w = s_1^{\epsilon_1} \ldots s_n^{\epsilon_n}$ where $s_i \in S$ 
and $\epsilon_i = \pm 1$.
  Given $h\maps F_S \to \Q$ a homomorphism, 
  define its {\bf induced $[n]$-tuple} $w^*h \maps [n]\to \Q$ 
  to be the function $w^*h(i) =  h(s_i^{\epsilon_i}) = \epsilon_i\,h(s_i)$.
\end{definition}

For induced $[n]$-tuples, $\int w^*h \;  = h(w)$, recovering the value of the homomorphism on the word.
In this case, the integral is a linear combination of 
signed counts of occurrences of generators.



\begin{definition}\label{BraidingProduct}
Given a word $w$ of length $n$ and \(i,j\) with $1 \leq i, j \leq n$, define 
    $i \lew j$ to mean  $i<j$ if $\epsilon_i=1$ or $i\le j$ if $\epsilon_i=-1$
    (i.e. \(i \lew i\) only if the \(i^\mathrm{th}\) letter of \(w\) is an inverse).

    Given $w$ and $f$, a word of length $n$ and an $[n]$-tuple  as above, 
    define the \textbf{cobounding} of $f$ in $w$ to be the 
    $[n]$-tuple  given by incremental sums 
    $\lpw f(j) = \sum_{i \lewsm j} f(i)$.

    If $g$ is another $[n]$-tuple, define the 
    \textbf{braiding product} of $f$ and $g$ over $w$
    to be the point-wise product 
    $\lpw f g$. As this is again an $[n]$-tuple, the process can be iterated further.
\end{definition}


Note that \(\lpw f\) is a ``discrete anti-derivative'' of \(f\) 
``relative to \(w\).''
Integrating braiding products of induced 
$[n]$-tuples formalizes a process of counting
ordered configurations of letters.

\begin{example}\label{ex:first braiding example}
In $F_S = \langle a,b,c\rangle$
consider the {\bf indicator homomorphism} $A$, sending $a \mapsto 1$
and the other generators to zero, with $B$ and $C$ defined
similarly. In this case $\int w^*A$ counts signed occurrences of the letter $a$ in $w$.
A short exercise shows that integrating the braiding product  
 $\int \lpw{w^* A}{w^* B}$ counts  configurations of letters 
 ``$a$ then $b$''  
 in $w$, with a sign which is switched if either letter occurs as an inverse. 
 
For example if $w = a^{\inv} b b a b$, then
 $\int\lpw{w^* A}{w^* B} = -2$.  
 This also illustrates the general fact that 
 if $\int w^* A$ vanishes, then 
 $\int \lpw{w^* A}{w^* B}$ may be  computed by counting 
 configurations of ``$b$ between canceling $a$-$a^\inv$ pairs'', using {any} 
 choice 
 pairing off $a$ and $a^\inv$ occurrences in $w$.  
 This is because any $b$'s coming \emph{after} such pairs contribute cancelling values.
 We call counting  configurations of one letter between 
 cancelling pairs of another (and its generalizations) ``letter linking,'' discussed further below.

 This behavior continues. For example
 $\int \lpw[\big] {\lpw{w^* A}{w^* B}} w^*C$ counts 
 signed configurations 
 of the form ``$(\ldots a^{\pm} \ldots b^{\pm} \ldots c^{\pm}\ldots)$''
 in $w$.
 See Section~\ref{sec:configuration braiding} for 
a precise formulation of such letter counting in general.
\end{example}

 The  subtlety is in the definition of $\lew$. 
Such a convention is necessary for words like $w = a a^{-1}$ to give
$\int \lpw{w^* A} w^*A = -1 + 1 = 0$, rather than
being $-1$ if one had naively counted only the configuration ``$a$-then-$a$".
Instead of being an ad-hoc solution making our counts well-defined on the free group, this convention is dictated by topology, as we will show in Section~\ref{topcombo}.

Extending these ideas, 
we use iterated braiding products of induced $[n]$-tuples to define 
functions on the free group.
We first introduce more compact notation.

\begin{definition}\label{def:symbol}
    A {\bf formal braiding symbol}  in a set of formal variables, say $X_1, X_2, \cdots$ is, inductively,
    a formal linear combination of expressions, each of
    which is either a variable $X_i$ or has the form $\lp \alpha \beta$ or $\beta \lp \alpha$, where $\alpha$ and
    $\beta$ are formal braiding symbols.
    The {\bf weight} of a symbol is 
    the largest number of variables involved 
    in a homogeneous term (applying multilinearity).

    A {\bf braiding symbol} is a pair $\sigma_{\vec h}$ 
    where $\sigma$ is a formal braiding symbol and $\vec{h} = h_1, h_2, \cdots$
    are homomorphisms $F_S\to \Q$.
\end{definition}

In examples, we shorten notation by incorporating  homomorphisms in the symbol --
for example $\lp {h_1} {h_2}$ stands for the braiding symbol $(\lp {X_1} {X_2})_{h_1,h_2}$.
Braiding symbols give recipes for specifying $[n]$-tuples associated to any word.

\begin{definition}\label{letterlinkingdef}
    If $\sigma_{\vec h}$ is a braiding symbol and $w$ is a word of length $n$, the  {\bf associated $[n]$-tuple}, denoted $w^* \sigma_{\vec{h}}$, is obtained by performing the prescribed braiding products recursively.  That is,
    \[
    w^* (\lp {\alpha_{\vec{h}}} \beta_{\vec{h}}) = \lpw {w^* \alpha_{\vec{h}}} w^*\beta_{\vec{h}},
    \]
    starting with $w^* (X_i)_{\vec{h}} =w^*h_i$ as defined above, and extending linearly for linear combinations.  The {\bf value}
    of the symbol on $w$ is $\sigma_{\vec{h}} (w) = \int w^*\sigma_{\vec{h}}$.  We call the  assignment
    $w \mapsto \sigma_{\vec{h}}(w)$ the {\bf letter braiding function} associated to the symbol $\sigma_{\vec{h}}$.
\end{definition}

Thus what we previously denoted $\int \lpw{w^* A} w^*B$ becomes $\lp{A}B (w)$.  

To recap, letter braiding functions are functions on the
set of words, and they are defined through construction and integration of functions on the sequence of letters of a word.

\begin{example}\label{BraidingProductExample}
Let $w=bca b^\inv c^\inv bb$ in $F_S = \langle a, b, c \rangle$ and 
use indicator functions $A$, $B$, $C$ as in Example~\ref{ex:first braiding example}.
The table below gives a step-by-step computation of the braiding product
$\lp{B-2A} C$.  The caret marks in the row for $\lp{B-2A}$ clarify
the partial sums of the cobounding by indicating the point up to which the summation has been performed, as dictated by $\lew j$ to include $j$
or not.
\begin{center}
\begin{tabular}{c|ccccccc}
 $w$ & $b$ & $c$ & $\ a$ & $\ b^\inv$ & $\,c^\inv$ & $b$ & $b$ \\ \hline\hline \\[-2ex]
 $B-2A$ & $1$ & $0$  & $-2$ & $-1$ & $0$ & $\;1$ & $\;1$ \\ \hline \\[-2ex]
 $\lp{B-2A}\!$ & ${\hat{}}\,\,0\,\,$ 
                 & ${\hat{}}\,\,1\,\,$
                 & $\,{\hat{}}\!\phantom{-}1\,$
                 & $\,{-2}\,\,{\hat{}}$
                 & ${-2}\,\,{\hat{}}\,$
                 & ${\hat{}}\,{-2}\ $
                 & ${\hat{}}{-1}$ \\
 $C$ & $0$ & $1$ & $\phantom{-}0$ & $\phantom{-}0$ & $-1\ $  & $\;0$  & $\;0$ \\ \hline \\[-2ex]
 $\lp{B-2A} C$ & $0$ & $\,1$ & $\phantom{-}0$ & $\phantom{-}0$ & $\ 2$ & $\ 0$ & $\ 0$\\ 
\end{tabular}
\end{center}
We deduce that $\lp{B-2A} C\,(w) = 1+2=3$.
Note that this could also be computed by combining 
``$b$ then $c$'' and ``$a$ then $c$'' configuration counts as remarked in
Example~\ref{ex:first braiding example}.  Or, because $(B-2A)(w)=0$, by configurations of 
``$c$ between $b$ or $a$ pairs'' following a pairing scheme matching $b$ with either
$b^{-1}$ or $a$ (in which each $a$ counts twice).

\end{example}

We omit the elementary proof of the following fact, which critically depends on the definition of $\lew$.

\begin{proposition}
    Letter braiding functions on words descend to well-defined functions on free 
    groups.
\end{proposition}

While letter braiding functions are elementary combinatorial objects, 
we will connect them with larger frameworks from geometry and rational
homotopy theory under conditions we call the ``eigen-setting.'' In this setting, 
 multiple approaches to invariant functions on groups coincide and all have integer values. One such theory consists of the subset of letter braiding functions with all proper subsymbols 
vanishing. These are called ``letter linking functions'', as given in Example~\ref{ex:first braiding example}, and  further discussed in Definition~\ref{def:lp}
and Theorem~\ref{freecase}.

Letter linking functions are the subject of \cite{monroe-sinha}, whose main result is that letter linking functions associated to symbols decorated by indicator homomorphisms 
are sharp obstructions for the lower central series
of free groups.  
For example $\lp A B (a b a^{\inv} b^{\inv}) = 1$,
as there is one $b$ in between  an
$a$-$a^{\inv}$ pair.  Since the $a$'s and $b$'s have
a non-zero ``linking number'', the word $a b a^{\inv} b^{\inv}$ is not
a two-fold commutator.

One of our main present results is to generalize
this to the radical of the lower central series for all groups.
Let $q: F_S \to G$ present a group $G$ as a quotient of a free group $F_S$, and say a 
letter linking function of $F_S$ {descends to} $G$
if the function is constant on cosets of $q$.
Our Lifting Theorem in  Section~\ref{sec:lifting} characterizes the letter
linking functions that descend to a quotient group. Clearly, such functions must vanish on relations in $G$, but this is by itself insufficient.  Our theorem gives a complete criterion using the language of Lie coalgebras, developed in Section~\ref{sec:prelims}. Functions on $G$ in hand, the following theorem quantifies the information they capture.

\begin{theorem}\label{thm:into-combinatorial-alg}
Let $w\in F_S$ be a word
representing an element of $G$.  Then some power
of $w$ is a product of $k$-fold nested commutators in $G$ if and only if $w$ vanishes on all linking functions of weight $\leq k$ that descend to $G$.
\end{theorem}

Proving this theorem at the end of 
Section~\ref{ProofFundamental} will be
the culmination of our technical work.
When $k=1$ this restates the classical fact the subgroup of $G$ on which all homomorphisms to $\Q$ vanish is the radical of the commutator subgroup $[G,G]$.  
The $k>1$ cases are new.  Moreover, this along with the Lifting Theorem gives rise to an algorithm to determine whether some power of a word is a $k$-fold commutator, an algorithm which has been implemented -- see Section~\ref{Algorithms}.  
Previous algorithms, both
 theoretical \cite{Chen-Fox-Lyndon} and implemented \cite{Nickel-Algorithm}, employ
 inductive calculation of the subquotients
 of the rationalized lower central series.  
 Our invariants
 are defined on the entire group and have
 a cobracket structure, leading to the
 first approach to these questions which organically retains information about commutators.

\bigskip

\subsection{Geometric topology}\label{geomtop} 

The geometric genesis of our work is illustrated in an elementary analysis of the Borromean rings,
which are depicted in Figure~\ref{borromean3}, going back to works of Massey \cite{Massey_1980} and further developed by Hain \cite{Hain-LinkGroups}.

\begin{figure}
\begin{subfigure}[b]{0.40\textwidth}
\centering
\begin{tikzpicture}[scale=.45]

\coordinate (1) at (-.7,.2);
\coordinate (2) at (5.0,.1);
\coordinate (3) at (-.7,-1.0);
\coordinate (4) at (6.1,.3);
\coordinate (1x) at (.3,1.2);
\coordinate (2x) at (.94,.35);
\coordinate (3x) at (.4,-1.6);
\coordinate (4x) at (-.3,-2.2);
\coordinate (1b) at (4.75,1.5);
\coordinate (2b) at (3.98,.80);
\coordinate (3b) at (4.2,-1.3);
\coordinate (4b) at (4.8,-1.8);


\draw[draw=none,fill=red!40] (-1,-.2) circle (2);
\draw[draw=none,fill=blue!40] (5.7,-.2) circle (2);

\begin{scope}[ultra thick,decoration={markings,
    mark=at position 0.6 with {\arrow[scale=.8,yshift=0mm]{ang 90}}}]
\draw [postaction={decorate},shorten >=2mm] (1) to[out=315,in=205] (2x) to[out=25, in=160] (2b);
\draw [postaction={decorate},shorten >=1.2mm,ultra thick,-{angle 90[red]}] (2) to[out=-50, in=-40] (3x) to[out=140, in=0,] (3);
\path [ultra thick, out=-140, in=115, shorten >=2mm] (3) edge (4x);
\path [ultra thick, out=110, in=-20,shorten >=2mm] (4) edge (1b);
\path [out=220,in=90,shorten <=0.2cm,shorten >=1.2mm,ultra thick,-{angle 90[red]}] (1x) edge (1);
\path [out=-25, in=140,ultra thick,shorten <= 2mm,shorten >=1.2mm,ultra thick,-{angle 90[blue]}] (2b) edge (2);
\end{scope}

\begin{scope}[ultra thick,decoration={markings,
    mark=at position 0.35 with {\arrow[scale=.8,yshift=-.5mm]{ang 90}}}]
\path [shorten <= 1mm, out=-40, in=-70,shorten >=1.2mm,ultra thick,-{angle 90[blue]}] (4x) edge[postaction={decorate}] (4);
\end{scope}

\begin{scope}
[decoration={markings, mark=at position 0.8 with {\arrow[scale=.8,yshift=-.2mm]{ang 90}}}]
\path [shorten <=2mm,shorten >= 2mm,ultra thick,out=160,in=30] (1b) edge[postaction={decorate}] (1x);
\end{scope}

\draw[ultra thick] (.867,.517) to[out=111,in=0] (-1,1.8) to[out=180,in=90] (-3,-.2) to[out=-90,in=180] (-1,-2.2) to[out=0,in=-140] (.286,-1.732);
\draw[ultra thick] (.554,-1.459) to[out=50,in=-79] (.963,.182);
\draw[ultra thick]  (5.353,-2.170)  
 to[out=-10,in=180] (5.7,-2.2) 
 to[out=0,in=-90]   (7.7,- .2) 
 to[out=90,in=0]    (5.7, 1.8) 
 to[out=180,in=90]  (3.7,- .2) 
 to[out=270,in=125] (4.062,-1.347); 
\draw[ultra thick]  (4.414,-1.732)  
 to[out=-30,in=155] (4.855,-2.012); 

\draw  node[fill,circle,inner sep=0pt,minimum size=7pt, red] at (1) {\pgfuseplotmark{*}};
\draw  node[fill,circle,inner sep=0pt,minimum size=7pt, red] at (3) {\pgfuseplotmark{*}};
\draw  node[fill,circle,inner sep=0pt,minimum size=7pt, blue] at (2) {\pgfuseplotmark{*}};
\draw  node[fill,circle,inner sep=0pt,minimum size=7pt, blue] at (4) {\pgfuseplotmark{*}};
\node [fill=black,star,star points=5,star point height =2mm,scale=.5,draw] at (2.5,1.95){};

\end{tikzpicture}
\vspace{-.5cm} 
\caption{A presentation of the Borromean rings with Seifert surfaces for two components. }\label{borromean3}
\end{subfigure} \hspace{.5cm}    
\begin{subfigure}[b]{0.40\textwidth}
\centering
\begin{tikzpicture}[scale=.25]

\draw[black, ultra thick] (0,0) rectangle (10,10);

\draw[red, very thick] (0,8).. controls (4.5,8) and (5.5,9.3)  .. (10,9.3);
\draw  node[fill,circle,inner sep=0pt,minimum size=7pt, red] at (0,8) {\pgfuseplotmark{*}};
\draw  node[fill,circle,inner sep=0pt,minimum size=7pt, red] at (10cm,9.3) {\pgfuseplotmark{*}};

\draw[blue, very thick] (0,6).. controls (2,6) .. (10,7.8);
\draw  node[fill,circle,inner sep=0pt,minimum size=7pt, blue] at (0,6) {\pgfuseplotmark{*}};
\draw  node[fill,circle,inner sep=0pt,minimum size=7pt, blue] at (10cm,7.8) {\pgfuseplotmark{*}};

\draw[red, very thick] (0,4).. controls (3,3)  .. (10,2.2);
\draw  node[fill,circle,inner sep=0pt,minimum size=7pt, red] at (0,4) {\pgfuseplotmark{*}};
\draw  node[fill,circle,inner sep=0pt,minimum size=7pt, red] at (10cm,2.2) {\pgfuseplotmark{*}};

\draw[blue, very thick] (0,2).. controls (2,2)  .. (10,1);
\draw  node[fill,circle,inner sep=0pt,minimum size=7pt, blue] at (0,2) {\pgfuseplotmark{*}};
\draw  node[fill,circle,inner sep=0pt,minimum size=7pt, blue] at (10cm,1) {\pgfuseplotmark{*}};

\draw[red, very thick] (10,6.6).. controls (8,6.6) and (8,4.8) .. (10,4.3);
\draw  node[fill,circle,inner sep=0pt,minimum size=7pt, red] at (10,4.3) {\pgfuseplotmark{*}};
\draw  node[fill,circle,inner sep=0pt,minimum size=7pt, red] at (10,6.6) {\pgfuseplotmark{*}};

\draw[blue, very thick](5.7,4.8) circle (.7);

\draw node[red,scale=1.2,ultra thick] at (-1.2,8) {$+$};
\draw node[blue,scale=1.2] at (-1.2,6) {$+$};
\draw node[red,scale=1.2] at (-1.2,4) {$-$};
\draw node[blue,scale=1.2] at (-1.2,2) {$-$};
\node [fill=black,star,star points=5,star point height =2mm,scale=.5,draw] at (0,10){};

\end{tikzpicture}
\caption{A cobordism between intersection points given by a homotopy}\label{cobordism}
\end{subfigure}\caption{}
\end{figure}

In the figure, two link components have Seifert surfaces chosen, drawn in red and blue.  
Beginning at the basepoint marked by a star, we record intersections of the third component with those surfaces down the left edge of the square
in Figure~\ref{cobordism}, which represents the domain
of that component.  These four intersections alternate
between red and blue, with signs determined by whether the third component is directed out of or into the page.  These signs alternate, and in  fact
 with any choice of basepoint, the third component represents a simple commutator in
the fundamental group of the first two components.

The square in Figure~\ref{cobordism} represents the possible behavior of such intersections under an isotopy of the third
component.  Assuming transversality of the homotopy, the preimage
of the two Seifert surfaces would be curves, pictured 
in red and blue, which are disjoint.  Because the colors of the left ends of these curves 
alternate, and the curves are disjoint, all four curves must have 
endpoints on the right edge.
So a homotopy of this component resulting in the unlink is impossible. The Isotopy Extension Theorem can now be invoked to complete a proof that the Borromean rings are linked. 

In general, the Hopf invariants discussed in Section~\ref{categorical} measure elements of fundamental groups by tracking intersections with codimension one submanifolds such as the Seifert surfaces in the example above. Cohomology is the most basic form of this: counting signed intersection points of a curve and a hypersurface. The example illustrates the greater discernibility of counting more intricate combinatorial configurations -- the interleaving of intersection times -- in a case where the simple cohomological counts vanish.

Such counts are at the heart of letter braiding, as we introduced in Section~\ref{combalg} and will connect to topology in Section~\ref{topcombo}.  
There are two directions of generalization.  One is that the idea can
be iterated.  Imagine for example
a more complicated analog of the Borromean rings of Figure \ref{borromean3}, with its
middle black component winding more times through the other two components. Then counts of 
some red point in between canceling
pairs of blue points, which themselves are between canceling
pairs of red points will also define a homotopy invariant.  Such a count  is a higher weight Hopf invariant.

The second direction of generalization is to allow the Seifert
surfaces to intersect, which undermines the argument we used to show homotopy invariance.  Consider a Whitehead link, as pictured in Figure~\ref{whitehead}, and imagine a third component linking with it.  Since the red and blue surfaces
intersect, the order in which the third component intersects them can change.  

\begin{figure}
\begin{subfigure}[b]{0.40\textwidth}
\centering
\begin{tikzpicture}[line cap=round,line join=round,scale=2.5]

\def\bluepath{%
  (0,0.6) to[out=0, in=90]
  (0.5,0) to[out=-90, in=0]
  (0,-0.6) to[out=180, in=-90]
  (-0.5,0) to[out=90, in=180]
  cycle
}
\def\redpathA{(-1,0) to[out=120, in=-60,looseness=1.2] ( 1,0)}
\def\redpathB{( 1,0) to[out=120, in=-5,looseness=0.9] 
              ( .4, .25) to[out=175, in=90,looseness=1.5]
              (  0,  0) to[out=270, in=0, looseness=1.5]
              (-.15,-.25) to[out=180, in=-60,looseness=0.9]
              (-1,0)}
\def\redpath{%
  (-1,  0) to[out=120, in=-60,looseness=1.2] ( 1,0)
  ( 1,  0) to[out=120, in=-5,looseness=0.9] 
  (.4,.25) to[out=175, in=90,looseness=1.5]
  ( 0,  0) to[out=270, in=0, looseness=1.5]
  (-.15,-.25) to[out=180, in=-60,looseness=0.9] 
  (-1,  0)}

\coordinate (La) at (-0.545,-0.01); 
\coordinate (Lb) at ( 0.00,-0.13); 
\coordinate (Ra) at ( 0.55, 0.02); 
\coordinate (Rb) at ( 0.02, 0.16); 
\def\intarcL{(La) to (Lb)}
\def\intarcR{(Ra) to (Rb)}
\def\bigrect{(-2,-1) rectangle (2,1)}

\begin{knot}[clip width=3.0, clip radius=15pt, looseness=1.3]
  \strand[line width=1.0mm,red]  \redpathA;
  \strand[line width=1.0mm,red]  \redpathB;
  \strand[line width=1.0mm,blue,rotate=-15,xslant=0.2] \bluepath;
  \flipcrossings{2,3}
\end{knot}

\begin{scope}[looseness=1.3, fill opacity=0.5]
  \fill[blue!50,draw=none,rotate=-15,xslant=0.2] \bluepath;
  \fill[red!50, draw=none] \redpath;
\end{scope}

\begin{scope}[looseness=1.3, fill opacity=0.3]
  \clip[rotate=-15,xslant=0.2] \bluepath;
  \clip \redpath;
  \fill[red!20!blue!80] \bigrect;
\end{scope}
\begin{scope}[looseness=1.3, fill opacity=0.3]
  \clip[rotate=-15,xslant=0.2] \bluepath;
  \clip \redpath;
  \clip (La) -- (-2,-2) -- (0,-2) -- (Lb) -- (La);
  \fill[blue!20!red!80] \bigrect;
\end{scope}
\begin{scope}[looseness=1.3, fill opacity=0.3]
  \clip[rotate=-15,xslant=0.2] \bluepath;
  \clip \redpath;
  \clip (Ra) -- ( 2, 2) -- (0, 2) -- (Rb) -- (Ra);
  \fill[blue!20!red!80] \bigrect;
\end{scope}

\draw[line width=.8mm,violet] \intarcL;
\draw[line width=.8mm,violet] \intarcR;

\fill[violet] (La) circle (0.03);
\fill[violet] (Lb) circle (0.03);
\fill[violet] (Ra) circle (0.03);
\fill[violet] (Rb) circle (0.03);

\end{tikzpicture}
\vspace{-1.5cm}   
\caption{A presentation of the Whitead link, a 
Seifert surface for each component, and their
intersection. }\label{whitehead}
\end{subfigure} \hspace{.5cm} 
\begin{subfigure}[b]{0.40\textwidth}

\centering
\begin{tikzpicture}[scale=0.35]

\draw[black, ultra thick] (0,0) rectangle (10,8);

\draw[name path=line,red,ultra thick] (0,7) to[out=0,in=180] (10,1);
\draw[name path=curve,blue,ultra thick] (0,4) to[out=0,in=180] (10,4);
\draw[fill,violet,name intersections={of= line and curve}] (intersection-1) circle[radius=.2];
\draw[ultra thick, violet] (intersection-1) ..controls (6,7) and (8,6.5).. (10,6.7);

\draw  node[fill,circle,inner sep=0pt,minimum size=7pt, blue] at (0,4) {\pgfuseplotmark{*}};
\draw  node[fill,circle,inner sep=0pt,minimum size=7pt, red] at (0,7) {\pgfuseplotmark{*}};
\draw  node[fill,circle,inner sep=0pt,minimum size=7pt, violet] at (10,6.7) {\pgfuseplotmark{*}};
\draw  node[fill,circle,inner sep=0pt,minimum size=7pt, red] at (10,1) {\pgfuseplotmark{*}};
\draw  node[fill,circle,inner sep=0pt,minimum size=7pt, blue] at (10,4) {\pgfuseplotmark{*}};

\draw node[blue,scale=1.2,ultra thick] at (-1.2,4) {$+$};
\draw node[red,scale=1.2,ultra thick] at (-1.2,7) {$+$};

\end{tikzpicture}
\vspace{-.8cm} 
\caption{Tracking what happens through a homotopy to intersections with Seifert surfaces and a ``correction'' surface.}\label{whitecobordism}
\end{subfigure}\caption{}
\end{figure}

But we can consider an auxiliary surface
which bounds their intersection, drawn as purple segments in Figure~\ref{whitehead}, along with the parts of the link components.  
If we track intersections of a third link component both with the two Seifert surfaces and this auxiliary surface, say in purple,
we would see phenomena such as in Figure~\ref{whitecobordism}.  Here  red and blue preimages of intersections cross, which  happens where
the Seifert surfaces intersect, but their crossing
point is a boundary point for the preimage curve for the purple surface.  
So as the order of intersections with red and blue Seifert 
surfaces changes -- reading left to right, the ordering changes (red,\,blue) to 
(blue,\,red) -- an intersection with the purple surface
is created or destroyed.  Thus intersections with the
purple surface can be used as ``correction terms'' for 
finer counts of blue and red intersections which depend on order.

Here it is also worth considering  the Hopf link, for which there can be no such correction term. Indeed, we can recognize two Seifert 
surfaces as generators of first cohomology of
the complement of the Hopf link, and 
the non-triviality of their cup product implies there
is no surface in the complement which bounds their intersection (along with sub-curves of the components).

Making contact with Hain's work in \cite{Hain-LinkGroups}, we formalize the role of submanifolds and their intersections 
using the language of Thom forms, as reviewed in
Appendix~\ref{Thomforms}.
Suppose $M$ is a three-manifold, such as a
link complement, so surfaces have associated one-forms.  The wedge product corresponds
to intersection, and the behavior of the product $\Omega_{dR}^1(M) \otimes \Omega_{dR}^1(M) \to \Omega_{dR}^2(M)$ determines whether two surfaces
can be used for ``counting alternating preimages''.   
If the product of one-forms is zero at the cochain level, there
is immediately such an invariant.  If it is zero in cohomology, then
there is an invariant involving a ``correction term.''  And if the product
of one-forms is non-zero in cohomology, then there is no such invariant.

More generally, given a collection of
codimension one submanifolds and some loop, one can count preimages of the submanifolds in prescribed orders.    Such  counts
are generally not invariant, as the codimension one submanifolds
can intersect.  But one can find linear combinations which
are invariant through analysis of  intersections and their
relations, which
 encode cup products and Massey products. 
Remarkably, this story is  governed by  classical constructions
from  algebra and rational homotopy theory: the bar complex and the Harrison complex.  The Harrison
theory distinguishes fundamental group elements up to their images in the Malcev Lie algebra, as codified in Theorem~\ref{thm-intro:universal property}.

\subsection{Relation to other work  and future directions}\label{Introwrap-up}


The gold standard for understanding fundamental groups explicitly through cochains has been the theory
of Chen integrals \cite{chen1971pi1}. 
Recasting Chen's integrals as time-ordered intersections with hypersurfaces features prominently in Hain's analysis of link groups \cite{Hain-LinkGroups}, where he cites these ideas as `folklore', having appeared in Massey's and Porter's work from 1960s-1980s. Indeed, our  combinatorial theory of letter linking invariants could have been developed 
through analysis of Chen integrals 
in the setting of handlebodies using Thom
forms for belt disks, as we use in Section~\ref{topcombo}.

On the other hand, rational homotopy theorists have known since the 1970's that rational cochains compute the Malcev Lie algebra of fundamental groups. We found it most effective, both conceptually and technically, to restate the construction in our language of Lie pairings and the Lie coalgebraic dual of a group (see Definition \ref{LiePairing}), leading to our theory of the Harrison-valued Hopf invariants. The coalgebraic dual perspective has the benefit of avoiding difficulties of duality in infinite dimensional vector spaces, allowing us to uniformly treat non-finitely generated groups.

 In the process of working on this project, the first author realized that a distinct theory could be developed using the classical Bar complex on the ring $C^*$ of simplicial cochains, with coefficients in any PID. The theory of simplicial-Bar-valued Hopf invariants is called letter-braiding, and it agrees in the algebraic setting with the one given in Section~\ref{combalg} (when restricted to ``nested'' braiding products, which in fact span all such products).
 
 As represented in Figure~\ref{bigpicture}, the various theories produce the same set of integer-valued invariants in the so-called eigen-setting, as formulated for Hopf invariants in Definition~\ref{def:lp}. 
The resulting three theories are distinct. Harrison Hopf invariants give a perfect duality with the Malcev Lie algebra, while Chen theory more naturally pairs the real group ring filtered by powers of its augmentation ideal, and similarly for letter braiding, which is defined over any PID. Calculations outside of the eigen-setting differ as already seen for the circle, with $\pi_1 = \langle a \rangle$ and first cohomology generated by the dual class $A$: the three theories all have an invariant we call $(A|A)$, which agree in the eigen-setting, but on the word $a^2$ the Harrison Hopf invariant returns $0$ as $(A|A) = -(A|A)=0$ by definition in the Harrison complex; the letter braiding invariant is $1$, as checked from Definition~\ref{letterlinkingdef}; and 
finally, the Chen integral will be $\frac{1}{2} + 1 + \frac{1}{2}  = 2$ because of two ``self-interaction terms'' added to the ``interaction'' between the two occurrences of $a$.

\begin{figure} \label{fig:eigen-setting}
\includegraphics[scale=.35]{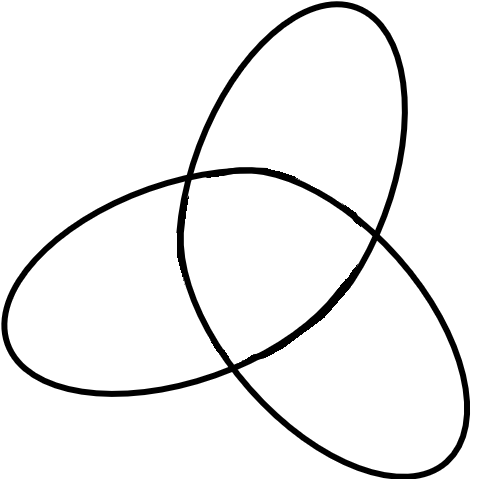}
\vspace{-4.00cm} 

\hspace{1.2cm}Letter\\
\hspace{1.0cm}Braiding

\vspace{.40cm}

\hspace{.3cm}Letter\\
\hspace{.25cm}Linking

\vspace{-.35cm}

\hspace{-2.7cm}Harrison\\
\hspace{-2.7cm}Invariants

\vspace{-.20cm}

\hspace{2cm}Chen-\\
\hspace{2.4cm}Hain

\vspace{.2cm}
\caption{Topological measures of the fundamental group. All produce letter linking invariants in the eigen-setting} \label{bigpicture}
\end{figure}

We bridge between the manifold theoretic tools and the combinatorial applications using our construction of a smooth thickening of a presentation two-complex of a group. Such thickenings are commonplace in geometric topology, but our construction here is tailored specifically to enable planar-graphical calculations on words in groups, akin to Fenn's spider diagrams \cite{Fenn-geometric-topology}. We hope to develop such diagrammatic tools further in subsequent work.

Another primary connection to highlight is work
of the last two authors.  In \cite{sinha-walter2}, they show that the same simple formalism we use here, namely pull-back and evaluation in the Harrison complex, gives exactly the rational functionals on the homotopy groups of simply connected spaces.
Moreover, they show that the resulting invariants are represented by ``higher linking numbers, with correction,'' generalizing the  classical Hopf invariant which measures a map $S^3 \to S^2$ through
the linking number of preimages of two distinct points
in $S^2$. Then in \cite{monroe-sinha}, the third author developed letter linking invariants for free groups purely combinatorially, establishing ``predictions'' of \cite{sinha-walter2} for maps $S^1 \to \bigvee S^1$.  In the current work, we return to topology for our development of letter linking for all groups.

Ultimately, by providing ``coordinates'' for the Malcev Lie algebra of any group, 
de Rham-Harrison Hopf invariants generalize the classical tools of Magnus
expansions and Fox derivatives.  Moreover, the close relationship to standard constructions
such as bar complexes and Massey products means plenty of connection to previous topological
studies of groups.  We defer to the paper
of the first author on letter braiding invariants \cite{LetterBraiding} for a thorough, though certainly incomplete, discussion in its Introduction.

In relatively recent work on a complementary approach,  Rivera--Zeinalian \cite{rivera-cobar} showed that an almost dual theory, 
the cobar construction on singular chains of a space $X$, yields a chain model for the loop space $\Omega X$, without  restrictive hypotheses. 
In particular, they deduce that the $0$-th cobar homology of $X$ is  exactly the group ring $\Z[\pi_1(X)]$ -- no completion necessary.
While sharper, the difficulty of their theory in comparison to ours is that theirs is not based on a quasi-isomorphism invariant functor, and for example one cannot use a small simplicial chain model for $X$.
Relatedly, the cobar spectral sequence does not converge, for example when $\pi_1(X)$ is a nonabelian simple group, and the natural filtrations are not exhaustive.  Thus, the ability to capture the fundamental group precisely stands in conflict with calculations, as one should expect in light of fundamental undecidability of the Word Problem. 

Our theory makes contact with Rivera--Zeinalian through an infinitary Bar complex: in Section \ref{beyond} we indicate that an unbounded generalization of our invariants can and does detect words that vanish in every nilpotent quotient of the group and thus represent the zero element in the Malcev group. This suggests an ad-hoc approach to attacking the Word Problem in specific cases. If one is able to construct a single infinitary letter linking function taking a nonzero value on a word, then the word must be nontrivial. Of course, there is no algorithm that can produce such a function in general, but geometric considerations might offer a guide in favorable situations.

\bigskip

We share both our plans and possible areas of application more broadly.

\begin{enumerate}[I.]
    \item\label{milnorfuture}
 Our first planned application is to Milnor invariants for links in three-manifolds, for which we give a 
taste in Sections~\ref{geomtop}~and~\ref
{milnor}. Stees \cite{stees2024milnor} recently initiated the theory of Milnor invariants in 
arbitrary three-manifolds from the completely different
perspective of concordance with a fixed representative
``unlink.''   We will proceed from a basic idea which is central to our present work: measuring homotopy classes of curves through 
patterns of intersections with Seifert surfaces and
coboundings of their intersections.  Our current framework allows for broadening the notion of Seifert surfaces to include possibly singular homologies between 
multiple link components.  Preliminary evidence show that the framework developed in this paper,
using Seifert surfaces as input for cochains, yields a robust theory.\\

\item The same techniques also lead to a geometric approach to the Johnson filtration of mapping class groups.  Consider a surface
with one boundary component for simplicity (but not for necessity - see Section~\ref{surfaces}), and fix 
a collection of disjoint curves $\alpha_i$  which represent
first cohomology -- for example ``belt $1$-disks'' in a handlebody presentation. 
Let $f$ be a diffeomorphism and for any based loop $c$ consider the following list of conditions
on the based loop $c^{-1} * f(c)$:
\begin{itemize}
    \item In the first weight, we count intersections with the $\alpha_i$.
    \item When those vanish, weight
    two counts are $\alpha_j$-intersections 
    occurring between cancelling pairs of $\alpha_i$-intersections. 
    \item General weight $n$ counts, defined when all lower counts
    vanish, are that of $\alpha_k$-intersections occurring between cancelling pairs of weight $(n-1)$ counts. 
\end{itemize}
A diffeomorphism $f$ has Johnson filtration degree exactly $n$ if and only
if the first $n-1$ counts vanish
for all curves, and some curve has
a non-zero $n$th count.  
Thus, for $f$ to be deep in the Johnson filtration forces complexity of some
curve $c^{-1} * f(c)$.  It would be interesting to tie this complexity to other notions of complexity 
of mapping classes.\\

\item A geometric reading of Theorem~\ref{thm-intro:universal property} using Thom forms, and analysis in the
Harrison spectral sequence (see Remark~\ref{HarrisonSS})
shows that the rational nilpotency class of a group is
governed by  Massey products of classes
represented by appropriate codimension one submanifolds.  That is, intersections of hypersurfaces within certain homology classes of fixed
complexity are forced.
It would be interesting to connect with the substantial 
literature on nilpotency class and geometry.\\

\item Our theory is in a sense dual to that of surface-towers, or gropes \cite{TeichnerICM}, which realize iterated commutators through surfaces.  Instead, we measure iterated commutators through systems of codimension one submanifolds.  A more direct comparison should be possible: realizing intersections between levels of a surface-tower with hypersurfaces arising by repeatedly cobounding intersections of lower-depth hypersurfaces.\\

\item\label{geomcochains} An important technical development for geometric application is the use of geometric cochains 
as developed by Friedman, Medina and the third author \cite{GeometricCochains}.  These are defined over the 
integers, and have a commutative product given by intersection, but  (then, of necessity) only have a partially defined product.  In that setting, the fundamental isomorphism  $ \int_{S^1} : H_\mE^0(S^1) \to
\Q$ might no longer hold.  We plan to explore the resulting differing theory, which conjecturally is closer to that
of letter braiding, in the course of our development 
of Milnor invariants.\\

\item In algebraic topology, our approach has an advantage over Chen integrals and others in that it is fully integrated into a Quillen model for rational homotopy
theory.  It would be helpful to more fully connect
with recent developments of Lie models based on
the Lawrence-Sullivan interval \cite{LieEncyclopedia}, which
indeed use the Harrison complex at some key points.  At a purely 
calculational level, there is a rich pairing between 
Lie algebras and Lie coalgebras at play (see \cite[Section 2]{Sinha-Walter1}).  And with our more 
explicit connections to group theory, there could
be an opportunity to transport questions between domains, such as understanding analogues in (geometric) group theory 
of dichotomy in rational homotopy theory. See \cite{HessSurvey} for a lovely survey
of rational dichotomy, and rational homotopy
theory more broadly.\\

\item A new tool we introduce to explore the interface of group theory and homotopy
theory is that of graphical models, used throughout
Section~\ref{examples} after being defined  in
Section~\ref{CurveForms}.  Briefly, take a presentation of a group by generators and relations, and use each relation to label the boundary of a disk, so there will be one marked point on the boundary of the disk for each letter in the relation, occurring in order.  
A graphical model is an immersed graph in each such disk in which edges intersect transversely and have endpoints either at marked points on the boundary or at intersections of other edges.  To such data we associate de\,Rham cochains of a manifold  with that fundamental group, which can be used to calculate Massey products and other structures related to our theory.

To our knowledge, these graphical partial models are new, and there are plenty
of open questions about them. For which groups can one find a graphical model which, while being partially defined,  suffices globally for all calculations of
the Lie dual?  Can de\,Rham theory be applied to
the disk to make these models have a fully defined product, for example making choices to understand products associated to curves which share a vertex?
Can group theoretic properties, such as nilpotency class, be read off from a (partial) graphical model?  For which groups can finite models be found?\\

\item 
A main focus in the present work is  computational efficacy. 
It is thus natural to 
implement on computer, as the second
author did for evaluating braiding invariants and finding bases 
of letter linking invariants. 
This work is available as Python and SageMath jupyter notebooks
\cite{Aydin-Github}.  
But there is still considerable room for optimizing algorithms, and in particular 
working with Lie coalgebras.  Steps in this direction are taken by the fourth 
author in Python libraries and Jupyter notebooks available on GitHub 
\cite{Walter-Github}.
Previous work of the fourth author \cite{Walter-Shiri} suggests that using
specific bases for Lie coalgebras such as the star basis may lead to quicker 
evaluation of cobrackets and computation of pairings.  
\\

\item
As we are not experts in algebra and group theory, we have less detailed broader speculation to share in that area.  
But one obvious line is to develop the theory beyond residual nilpotence, 
using the completed Harrison complex as in Section~\ref{beyondnilpotence}.  
Another is to more fully apply the theory over the integers or finite fields, using either geometric cochains and the Harrison complex, as proposed in item \ref{geomcochains} above, or simplicial cochains which lead to letter braiding, as already developed in \cite{LetterBraiding}.  
These theories shed light on group rings with their augmentation ideal filtration, rather than the Malcev Lie algebra.  The letter braiding setting, especially 
over finite fields, is ripe for computer implementation 
as initiated for letter linking in
\cite{Aydin-Github} and \cite{Walter-Github}.\\

\end{enumerate}

\color{black}      
\tableofcontents

\section{Definitions}\label{sec:prelims}

As discussed in Appendix \ref{BarDerivedIndec}, the Harrison complex is an explicit presentation of the derived indecomposibles
of  an augmented differential graded commutative algebra. 
Our preferred model for the Harrison complex is based on 
an operadic approach to Lie coalgebras, employing a model for the
Lie cooperad arising from cohomology of configuration
spaces originating in \cite{Sinha-Walter1}.  
We use
the name ``Eil'' for this operadic model for Lie coalgebras.

\subsection{Symbols, Eil trees, and cofree Lie coalgebras}\label{Eil definition}
We depart slightly from \cite{Sinha-Walter1} and use a new
model via rooted trees. 
This approach is more compatible with group theory,
ties directly to pre-Lie coalgebras~\cite{ChaLiv}, is more in line with the modern perspective of
$\infty$-operads~\cite{Dendroidal}, and seamlessly connects with our letter braiding formalism
from Section~\ref{combalg}.  
Tree symbols give a ``coalgebra-native'' presentation of Lie coalgebras, with its own 
distinct and interesting combinatorial structure.

\begin{definition}\label{Eil}
Write $\Tr_*$ for the vector space generated by rooted trees
with vertices labeled by elements from finite sets, and
write $\Tr_*(n)$ for the vector space generated by
rooted trees with vertices labeled uniquely by $\{1,\dots,n\}$.  We call $n$ the {\bf weight}.

Let $\Eil{\Tr_*}$, the space of rooted Eil trees, 
be the quotient of $\Tr_*$ 
by the following local relations, 
which involve changing the root-designated vertex and (for Arnold) moving
an edge to the new root.
\begin{align*}
\text{(root change)}\qquad & 
0 \quad = \quad 
\begin{xy}                           
  (-2,-5)*+[o][F-]{\color{blue}\scriptstyle \bm{a}}="a",    
  (2,0)*+UR{\color{red}\scriptstyle \bm{b}}="b",     
  "a";"b"**[thicker]\dir{-}, 
      "a"!<.3pt,0pt>;"b"!<.3pt,0pt> **\dir{-},  "a"!<-.3pt,0pt>;"b"!<-.3pt,0pt> **\dir{-},
  (-9,-4),{\ar@{. }@(u,u)(9,-4)},
  ?!{"a";"a"+/d:a(10)/}="a1",
  ?!{"a";"a"+/d:a(30)/}="a2",
  ?!{"a";"a"+/d:a(50)/}="a3",
  ?!{"b";"b"+/va(90)/}="b1",
  ?!{"b";"b"+/va(30)/}="b2",
  ?!{"b";"b"+/va(60)/}="b3",
  "b";"b1"**[red]\dir{-},  "b";"b2"**[red]\dir{-},  "b";"b3"**[red]\dir{-},
  "a";"a1"**[blue]\dir{-},  "a";"a2"**[blue]\dir{-},  "a";"a3"**[blue]\dir{-},
\end{xy}\ \ + \ \ 
\begin{xy}                           
  (-2,0)*+UR{\color{blue}\scriptstyle \bm{a}}="a",    
  (2,-5)*+[o][F-]{\color{red}\scriptstyle \bm{b}}="b",     
  "a";"b"**[thicker]\dir{-},  
      "a"!<.3pt,0pt>;"b"!<.3pt,0pt> **\dir{-},  "a"!<-.3pt,0pt>;"b"!<-.3pt,0pt> **\dir{-},
  (-9,-4),{\ar@{. }@(u,u)(9,-4)},
  ?!{"a";"a"+/d:a(00)/}="a1",
  ?!{"a";"a"+/d:a(30)/}="a2",
  ?!{"a";"a"+/d:a(60)/}="a3",
  ?!{"b";"b"+/d:a(-10)/}="b1",
  ?!{"b";"b"+/d:a(-30)/}="b2",
  ?!{"b";"b"+/d:a(-50)/}="b3",
  "b";"b1"**[red]\dir{-},  "b";"b2"**[red]\dir{-},  "b";"b3"**[red]\dir{-},
  "a";"a1"**[blue]\dir{-},  "a";"a2"**[blue]\dir{-},  "a";"a3"**[blue]\dir{-},
\end{xy} \\
\text{(root Arnold)}\qquad & 
0 \quad = \quad
\begin{xy}                           
  (-4.5,-5)*+[o][F-]{\color{blue}\scriptstyle \bm{a}}="a",    
  (0, 0)*+UR{\color{red}\scriptstyle \bm{b}}="b",   
  (4.5,0)*+UR{\color{OliveGreen}\scriptstyle \bm{c}}="c",   
  "a";"b"**[thicker]\dir{-}, 
      "a"!<.3pt,0pt>;"b"!<.3pt,0pt> **\dir{-},  "a"!<-.3pt,0pt>;"b"!<-.3pt,0pt> **\dir{-},
  "a";"c"**[thicker]\dir{-}, 
      "a"!<.4pt,0pt>;"c"!<.4pt,0pt> **\dir{-},  "a"!<-.4pt,0pt>;"c"!<-.4pt,0pt> **\dir{-},
  (-9,-4),{\ar@{. }@(u,u)(9,-4)},
  ?!{"a";"a"+/d:a(0)/}="a1",
  ?!{"a";"a"+/d:a(15)/}="a2",
  ?!{"a";"a"+/d:a(30)/}="a3",
  ?!{"b";"b"+/d:a(0)/}="b1",
  ?!{"b";"b"+/d:a(-20)/}="b2",
  ?!{"b";"b"+/d:a(20)/}="b3",
  ?!{"c";"c"+/d:a(00)/}="c1",
  ?!{"c";"c"+/d:a(-30)/}="c2",
  ?!{"c";"c"+/d:a(-60)/}="c3",
  "a";"a1"**[blue]\dir{-}, "a";"a2"**[blue]\dir{-}, "a";"a3"**[blue]\dir{-},
  "b";"b1"**[red]\dir{-}, "b";"b2"**[red]\dir{-}, "b";"b3"**[red]\dir{-},
  "c";"c1"**[OliveGreen]\dir{-}, "c";"c2"**[OliveGreen]\dir{-}, "c";"c3"**[OliveGreen]\dir{-},
\end{xy}\ \ + \ \                         
\begin{xy}                           
  (-4.5,0)*+UR{\color{blue}\scriptstyle \bm{a}}="a",    
  (0,-5)*+[o][F-]{\color{red}\scriptstyle \bm{b}}="b",   
  (4.5,0)*+UR{\color{OliveGreen}\scriptstyle \bm{c}}="c",   
  "b";"a"**[thicker]\dir{-},  
      "b"!<.3pt,0pt>;"a"!<.3pt,0pt> **\dir{-},  "b"!<-.3pt,0pt>;"a"!<-.3pt,0pt> **\dir{-},
  "b";"c"**[thicker]\dir{-}, 
      "b"!<.3pt,0pt>;"c"!<.3pt,0pt> **\dir{-},  "b"!<-.3pt,0pt>;"c"!<-.3pt,0pt> **\dir{-},
  (-9,-4),{\ar@{. }@(u,u)(9,-4)},
  ?!{"a";"a"+/d:a(0)/}="a1",
  ?!{"a";"a"+/d:a(30)/}="a2",
  ?!{"a";"a"+/d:a(60)/}="a3",
  ?!{"b";"b"+/d:a(0)/}="b1",
  ?!{"b";"b"+/d:a(-10)/}="b2",
  ?!{"b";"b"+/d:a(10)/}="b3",
  ?!{"c";"c"+/d:a(00)/}="c1",
  ?!{"c";"c"+/d:a(-30)/}="c2",
  ?!{"c";"c"+/d:a(-60)/}="c3",
  "a";"a1"**[blue]\dir{-}, "a";"a2"**[blue]\dir{-}, "a";"a3"**[blue]\dir{-},
  "b";"b1"**[red]\dir{-}, "b";"b2"**[red]\dir{-}, "b";"b3"**[red]\dir{-},
  "c";"c1"**[OliveGreen]\dir{-}, "c";"c2"**[OliveGreen]\dir{-}, "c";"c3"**[OliveGreen]\dir{-},
\end{xy}\ \ + \ \                            
\begin{xy}                           
  (-4.5,0)*+UR{\color{blue}\scriptstyle \bm{a}}="a",    
  (0, 0)*+UR{\color{red}\scriptstyle \bm{b}}="b",   
  (4.5,-5)*+[o][F-]{\color{OliveGreen}\scriptstyle \bm{c}}="c",   
  "c";"a"**[thicker]\dir{-},         
      "c"!<.4pt,0pt>;"a"!<.3pt,0pt> **\dir{-},  "c"!<-.4pt,0pt>;"a"!<-.4pt,0pt> **\dir{-},
  "c";"b"**[thicker]\dir{-},         
      "c"!<.3pt,0pt>;"b"!<.3pt,0pt> **\dir{-},  "c"!<-.3pt,0pt>;"b"!<-.3pt,0pt> **\dir{-},
  (-9,-4),{\ar@{. }@(u,u)(9,-4)},
  ?!{"a";"a"+/d:a(0)/}="a1",
  ?!{"a";"a"+/d:a(30)/}="a2",
  ?!{"a";"a"+/d:a(60)/}="a3",
  ?!{"b";"b"+/d:a(0)/}="b1",
  ?!{"b";"b"+/d:a(-20)/}="b2",
  ?!{"b";"b"+/d:a(20)/}="b3",
  ?!{"c";"c"+/d:a(00)/}="c1",
  ?!{"c";"c"+/d:a(-15)/}="c2",
  ?!{"c";"c"+/d:a(-30)/}="c3",
  "a";"a1"**[blue]\dir{-}, "a";"a2"**[blue]\dir{-}, "a";"a3"**[blue]\dir{-},
  "b";"b1"**[red]\dir{-}, "b";"b2"**[red]\dir{-}, "b";"b3"**[red]\dir{-},
  "c";"c1"**[OliveGreen]\dir{-}, "c";"c2"**[OliveGreen]\dir{-}, "c";"c3"**[OliveGreen]\dir{-},
\end{xy}                          
\end{align*}  
In the diagrams above the root is circled 
and $a$, $b$, $c$ stand for vertices of a tree which could possibly have
further branches (indicated by the colored``whiskers'') not modified by these relations.
\end{definition}  

In \cite{Sinha-Walter1} we use weight to refer to the number of edges rather than vertices, 
which differs by one.   
The current definition of weight matches operad/cooperad ``-arity'' grading and is more convenient
for later spectral sequence computations.

While root change and root Arnold relations only apply at the root
of a tree, the root may be moved arbitrarily, with sign change,
via the root change relation. 
Moving the root, applying the Arnold identity, then
moving it back generates a version of the Arnold relation
deeper within rooted trees.



Rooted trees can be written as nested partitions of leaf sets.  Including 
internal vertices into the nested partition structure results in a ``tree symbol''
for a vertex-labeled rooted tree.


%
%
%

\begin{definition}\label{tree symbol}
The {\bf tree symbol} for a vertex-labeled rooted tree
in $\Tr_*$ 
is defined as follows.

\begin{itemize}
 \item 
 A tree with only one vertex has symbol given by the vertex's label.
 
 \item 
 A corolla 
  $\begin{xy}                      
  (-7,3)*+UR{\scriptstyle a_1}="a", 
  (-3,3)*+UR{\scriptstyle a_2}="b", 
  (1,3)*+UR{\scriptstyle \cdots}, 
  (5,3)*+UR{\scriptstyle a_n}="c", 
  (0,-2)*+[o][F-]{\scriptstyle b}="d",
  "a";"d"**\dir{-},     
  "b";"d"**\dir{-},     
  "c";"d"**\dir{-},     
 \end{xy}$    
 where $a_i$ and $b$ are vertex labels with $b$ at the root, has symbol 
 $\tpc{a_1} \tpc{a_2}\! \cdots\! \tp{a_n} b$.  \\

 \item 
 More generally given a corolla
  $\begin{xy}                      
  (-7,3)*+UR{\scriptstyle A_1}="a", 
  (-3,3)*+UR{\scriptstyle A_2}="b", 
  (1,3)*+UR{\scriptstyle \cdots}, 
  (5,3)*+UR{\scriptstyle A_n}="c", 
  (0,-2)*+[o][F-]{\scriptstyle b}="d",
  "a";"d"**\dir{-},     
  "b";"d"**\dir{-},     
  "c";"d"**\dir{-}     
 \end{xy}$    
 where $b$ is the root vertex and $A_1,\dots,A_n$ are \textbf{subtrees}
 with symbols $\alpha_1,\dots,\alpha_n$
 respectively, the entire tree
 has symbol 
 $\tpc{\alpha_1} \tpc{\alpha_2}\!\cdots\!\tp{\alpha_n} b$. 
\end{itemize}
\smallskip
\end{definition}

\begin{example}\label{ex:trees}
The tree 
$\begin{xy}                      
  (-6,5)*+UR{\scriptstyle a}="a", 
  (-3,0)*+UR{\scriptstyle b}="b", 
  (0,-5)*+[o][F-]{\scriptstyle c}="c",
  "a";"b"**\dir{-},     
  "b";"c"**\dir{-},     
 \end{xy}$
has symbol
$\tp[\big]{\tp{a}b}c$.
The tree
$\begin{xy}                      
  (0,5)*+UR{\scriptstyle a}="a", 
  (6,5)*+UR{\scriptstyle b}="b", 
  (-3,0)*+UR{\scriptstyle c}="c", 
  (3,0)*+UR{\scriptstyle d}="d", 
  (0,-5)*+[o][F-]{\scriptstyle e}="e",
  "e";"c"**\dir{-},     
  "e";"d"**\dir{-},     
  "d";"a"**\dir{-},     
  "d";"b"**\dir{-}     
 \end{xy}$
has symbol $\tpc[\big] c \tp[\big]{\tpc a \tp b d} e$.
\end{example}

The definition above gives valid symbols for both planar and non-planar trees, that is with or without 
commutativity at corollas.  We employ  non-planar trees, and 
thus impose {\bf commutativity for tree symbols}; so for example
\[\tpc[\big] c \tp[\big] {\tpc a \tp b d} e = \tpc[\big] {\tpc b \tp a d} \tp[\big] c e.\]
We  also allow the unbracketed root vertex
to be written at arbitrary location, for example
\( \tpc a  \tp b c = \tp  a c\, \tp  b = c \,\tpc a \tpc b.\)

Tree symbols are essentially the formal braiding symbols of Definition~\ref{def:symbol}.  Eil tree symbols have root change and 
root Arnold identities imposed, as we discuss further momentarily. 
Letter braiding
invariants do not satisfy these identities universally but
do satisfy them in special cases.  See Theorem~\ref{freecase}.

\begin{remark}
Eil tree
symbols have their origin in \cite[Section 3]{monroe-sinha}, where they are called ``pre-symbols'' 
using parenthesis delimiters.
In \cite[Definition 3.1]{monroe-sinha} the non-parenthesized element was called the
``free letter''. Here, it serves as the root
of its subtree.  
One major difference with \cite{monroe-sinha} is  
that the current theory allows repeated letters.

In \cite[Definition 3.2]{monroe-sinha} a correspondence between symbols and trees is 
made via the containment partial ordering induced by nested parenthesizations in a  
symbol.  This coincides with
the canonical root maximal partial order on the vertices of a rooted tree. 
This  equivalence between rooted trees
and partial orders also appears in \cite[Lemma 3.2]{Dendroidal} in the context of 
dendroidal objects and $\infty$-operads.
The generating set of {long graphs} of \cite[Fig 1]{Sinha-Walter1}
correspond to {\bf linear trees} consisting of a single long branch extending from the   
root, as in the first case of Example \ref{ex:trees} above. 
 In the notation of 
\cite{Dendroidal} these are the total orders -- the 
image of simplicial sets inside of dendroidal sets. 
The basic roles of rooted and other similar trees
in \cite{Dendroidal} as well as in our  
approach to Lie coalgebras and Harrison homology
is more than coincidental, and we hope to explore it in 
future work.
\end{remark}

Extend the notation introduced in Definition~\ref{tree symbol} as follows.
Let $\tpc \alpha \beta$ denote the tree composed of disjoint 
subtrees $A$ and $B$ with symbols $\alpha$ and $\beta$,
where $B$ is the subtree containing the root
and there is an edge connecting the root of $A$ to the root of $B$.
We say $A$ is the {\bf branch subtree} and $B$ is the 
{\bf root subtree}.
For example $\tpc[\big] c \tp[\big] {\tpc a \tp b d} e$ may be written as
$\tp \alpha \beta$ where
$\alpha = \tpc a \tp b d$ and $\beta = \tp c e$.
With this notation, the root change and root Arnold
relations become the following.
\begin{align*}
\text{(root change)}\qquad & \qquad
0 \ = \ 
\tpc\alpha \beta\  +\  \alpha\, \tpc\beta, \\  
\text{(root Arnold)}\qquad & \qquad
0 \ = \
\tpc\alpha \tp\beta \gamma\ +\  
\tpc\alpha \,\beta\, \tpc\gamma\ +\  
\alpha\, \tpc\beta \tpc\gamma    
\end{align*}
where $\alpha$, $\beta$, $\gamma$ are expressions for subtrees.
Contained in the span of the root change and root Arnold relations 
we also have
\(0 = \sum_i 
   \tpc{\alpha_1}\cdots {\alpha_i} \cdots \tpc{\alpha_n}\),
which are called the Leibniz relations in \cite[Prop 3.10]{monroe-sinha}.

\begin{remark}
Applying root changes, we may rewrite the Arnold relation fixing the root vertex. 
\begin{equation}\label{eqn:nesting Arnold}
0 \ = \ \tpc \alpha \tpc \beta \gamma
      \ - \ \tpc[\big]{\tpc \alpha \beta} \gamma \ - \ \tpc[\big]{\alpha \tpc \beta} \gamma
\end{equation}
The resulting relation ``nests'' as a local relation on subsymbols.
\begin{equation}\label{eqn:nested Arnold}
\begin{aligned}[t]
\phantom{\tau}
  0 \ &= \ \tpc[\Big]{ \tpc \alpha \tpc \beta \gamma } \tau
       \ - \ \tpc[\Big]{ \tpc[\big] {\tpc \alpha \beta} \gamma } \tau
       \ - \ \tpc[\Big]{ \tpc[\big] {\alpha \tpc \beta} \gamma } \tau \\
  \ ``&=" \ \tpc[\Big]{ \tpc \alpha \tpc \beta \gamma 
       \ - \          \tpc[\big] {\tpc \alpha \beta} \gamma 
       \ - \          \tpc[\big] {\alpha \tpc \beta} \gamma } \tau.
\end{aligned}
\end{equation}
Rearranging Equation~(\ref{eqn:nesting Arnold}), it measures the 
failure of a similar local ``nested root change'' relation.
\begin{align*} 
\tpc \alpha \tpc \beta \gamma \ &= \ 
\tpc[\Big]{\tpc \alpha \beta} \gamma
  \ +\ \tpc[\Big]{\alpha \tpc \beta} \gamma
    \\
 \ ``&=" \  \tpc[\Big]{\tpc \alpha \beta
      \ +\     \alpha \tpc \beta} \gamma.
\end{align*}
Quotes around the final equal sign are due to the fact that, unlike braiding symbols, 
tree symbols are \emph{not} defined to be multilinear.  So the final statement is not 
defined, precisely because root change doesn't nest.
\end{remark}

Removing an edge from a tree yields a pair of rooted trees.
Unlike above, where we only attached at the root, we consider 
cutting (or ``pruning'')
a tree at any edge.
We abusively refer the result of pruning at an edge $e$ also as 
{\bf{branch}} and {\bf{root}} trees,
written $T^{\hat e}_b$ and $T^{\hat e}_r$.
The
edge cutting coproduct of \cite[Definition 2.6]{Sinha-Walter1}
translates to a tree pruning coproduct on $\Tr_*$, given by
\[ \Delta T = 
   \sum_{e\in E(T)} T_b^{\hat e} \otimes T_r^{\hat e}\]
where the sum removes edges from $T$ one at a time, 
separating branch and root subtrees $T^{\hat e}_b$ and $T^{\hat e}_r$.
Note that this is dual to the rooted tree pre-Lie product of \cite{ChaLiv}.
In terms of symbols, the tree pruning coproduct is given by excision of
subsymbols.

\begin{definition}\label{tree coprod}
Let  
\[ \Delta \psi = 
  \sum_{\psi = \cdots \tpc \alpha \cdots}  
     {\alpha} \otimes 
     {\psi^{\widehat{\alpha}}}, \]
where the sum ranges over all ways of picking a (nested) subsymbol $\alpha$ of $\psi$, and
the symbol $\psi^{\widehat{\alpha}}$ is the result of excising the (nested)
subsymbol $\tp \alpha $ from $\psi$.
 
  The anti-commutative Lie cobracket 
  is $\cobr \psi = \Delta \psi \, - \, \tau \Delta \psi$, 
  where $\tau$ is the twist map $\tau(a\otimes b) = b\otimes a$.
\end{definition}  

\begin{example}\label{Ex for Greg}
Compare the following coproducts of symbols.
 \begin{itemize}
   \item $\Delta \tpc a \tpc b c \ = \ a \otimes \tpc b c \ + \ \phantom{\tpc a}b \otimes \tpc a c$
   \item $\Delta \tpc[\big] {\tpc a b} c \ = \ a \otimes \tpc b c \ + \ \tpc a b \otimes c$.
 \end{itemize} 
\end{example}


In $\Tr_*(n)$ the tree pruning operation and the pruning coproduct should
also shift indices of resulting trees downward to be consecutive integers.  
For example pruning the subtree $\alpha$ from the tree $\tpc \alpha \beta$
in $\Tr_*(5)$ as indicated below should yield the given result in
$\Tr_*(3)\otimes\Tr_*(2)$.
\[ \tpc[\big] {\; \underbrace{\tpc 5 \tp 1 3}_\alpha\; }\!\underbrace{\tpc[\big] 4 2\,}_\beta 
\longmapsto \tpc 3 \tpc 1 2 \, \otimes \, \tpc 2 1. \]

Recall that \cite{Sinha-Walter1} uses oriented graphs instead of rooted trees -- see Appendix \ref{BarHarrison}. In the appendix we explain how these two models are naturally isomorphic, allowing
us to use results from \cite{Sinha-Walter1}.  There  we show that the 
kernel of the cobracket in weight greater than one is precisely the expressions generated by 
arrow reversing (root change) and Arnold identities.
Cobracket thus descends from $\Tr_*(n)$ to the quotient $\Eil{\Tr_*}(n)$, 
yielding a model for the Lie cooperad.
The perfect pairing between $\Eil{\Tr_*}(n)$ and $\Lie(n)$ is combinatorially rich - 
see
\cite[Definition 2.11]{Sinha-Walter1} for its description using graphs.  
We use our Eil symbols model of the Lie cooperad 
to construct (conilpotent) cofree Lie coalgebras in the usual manner.

\begin{definition}\label{cofreeE}
Let $W$ be a vector space. Define the  {\bf conilpotent cofree Lie coalgebra} on $W$ 
to be the vector space

\begin{equation*}
\cofreeE(W)\ \cong \ 
  \bigoplus_{n\ge 1} 
  \raisebox{3pt}{$\bigl(\Eil{\Tr_*}(n)\,\otimes\,W^{\otimes n}\bigr)$}
                \!\raisebox{-2pt}{$\diagup$}\!
                \raisebox{-4pt}{$\sym_{n}$},
\end{equation*}
with cobracket operation induced by the cobracket on $\Tr_*(n)$ and deconcatenation on $W^{\otimes n}$. Here, 
the symmetric group $\sym_{n}$ acts on $\Eil{\Tr_*}(n)$  by permuting labels $\{1,\dots,n\}$ 
and on $W^{\otimes n}$ by reordering, with Koszul signs if $W$ is graded.

The index of the summand in the decomposition of $\cofreeE(W)$ above is the {\bf weight}.  Let  $\cofreeE^{\le k}(W)$ denote the subcoalgebra given by 
restricting to weights less than or equal to $k$.
\end{definition}

As we did for braiding symbols, we may write elements in the $n$-th summand of $\cofreeE(W)$ using the elements in 
$W^{\otimes n}$ as labels of vertices within a tree symbol. 
For example 
$$\tpc[\big] c \tp[\big] {\tpc a \tp b d} e \ = \ 
 \tpc[\big] 1 \tp[\big] {\tpc 2 \tp 3 4} 5 \; \bigotimes \;
 \bigl(c\otimes a\otimes b\otimes d\otimes e\bigr).$$ 
Modulo the $S_{n}$-action we may always arrange for the indices of 
a tree symbol to be increasing from left to right.

We give detailed treatment, with examples, of the Koszul signs for cobracket as well as for the differentials in the Harrison complex in Appendix~\ref{BarHarrison}.

From \cite[Proposition 3.18]{Sinha-Walter1}, we have the following, which dualizes our understanding of free 
$k$-step nilpotent
Lie algebras as either ``$k$ or less nested brackets of generators'' or ``Lie
algebras modulo $(k+1)$ iterated brackets''.

\begin{proposition}\label{cofree conilpotent}
  $\cofreeE^{\le k}(W) = \mathrm{ker}\bigl(
    \,\cobr{\cdot}^{\circ k}:\cofreeE W \to (\cofreeE W)^{\otimes (k+1)}
   \bigr)$
  is the cofree $k$-conilpotent Lie coalgebra on $W$.
\end{proposition}


We use this model for the 
Lie coalgebraic bar construction, which underlies the Harrison complex.

\subsection{Harrison complex and Hopf invariants}\label{HopfInvar}

With cochain algebras in mind, we focus on differential graded-commutative, augmented, unital algebras  $(A,\ \mu_A,\ d_A)$, with augmentation ideal $\bar A$ so that $A = k \cdot 1 \oplus \bar A$ with $k = \Q$~or~$\R$.  
Key examples are commutative algebras of differential forms on simplicial sets and manifolds, augmented by restriction to a basepoint.

Recall that the desuspension of a cochain complex $(C^\bullet,d)$ is the shifted complex 
$s^{\inv}C^\bullet := C^{(\bullet-1)}$ with $d s^\inv = - s^\inv d$.

\begin{definition}\label{D:E}
The {\bf Harrison complex} of $A$ is the complex 
with underlying graded vector space $\cofreeE(s^{\inv}\! \bar A)$, the cofree Lie coalgebra on
the desuspension of the augmentation ideal, with its natural cohomological grading coming from tensor powers of $s^\inv\!\bar A$.

Its differential is the sum of two terms $d_\mE = d_{\cofreeE(s^{\inv}\!\bar A)} + d_\mu$.  The first is 
extended from that of $s^{\inv}\! \bar A$ to its tensor powers by the Leibniz rule, and 
the second is built from
the operation of ``contracting edges and multiplying labels'' as follows.
For rooted trees with a single edge define
$d_\mu\bigl( \tp {s^{\inv}a} s^{\inv} b \bigr) = (-1)^{|a|}\, s^{\inv}(ab)$.
As shown in \cite[Proposition 4.7]{Sinha-Walter1}, 
this (co)freely extends to a differential $d_\mu$ on all of $\cofreeE(s^{\inv} \bar A)$.

\end{definition}


Equivalently, the Harrison complex is the totalization of a second-quadrant bicomplex whose horizontal grading is by tensor degree in $\cofreeE(s^\inv\bar A)$, which is equal to the number of $s^{\inv}$ symbols in an expression, and vertical grading by cohomological degree inherited from $A$ (ignoring the desuspensions). 
The differential $d_{\cofreeE s^\inv\!\bar A}$ has bidegree $(0,1)$ and will be called the {\bf vertical differential}, while $d_\mu$, which is defined through edge contraction and multiplication, 
has bidegree $(1,0)$ and will be called the {\bf horizontal differential}. That is
$$\mE(A) =
    {\rm Tot}_\oplus \bigl(\cofreeE(s^\inv\bar A),\ 
  d_{\cofreeE s^\inv\!\bar A},\ d_\mu\bigr),$$
from which a natural spectral sequence can be constructed, as discussed in Remark~\ref{HarrisonSS}.

 In general, the maps above should involve Koszul signs tracking the interaction of differentials with different gradings.  We give details in Appendix~\ref{signs}, as these will be needed in particular to integrate our theory of de-Rham-Harrison Hopf invariants 
for the fundamental group with the higher homotopy group version, capturing the action of the fundamental group on higher homotopy groups.  In the present setting of interest, desuspended one-cochains are in degree zero,  in which case these Koszul signs all vanish. 

\medskip

Recall  the bar complex $\mB(A)$, defined similarly and more readily 
as the total complex of the coassociative cofree tensor
coalgebra $T(s^\inv \bar A)$ with first differential 
$d_{s^\inv\!\bar A}$ and second differential given by summing over the ways to
multiply pairs of consecutive terms of a tensor -- see Definition~\ref{StandardBar}.  
 Harrison's original definition of $\mE$ for commutative algebras in \cite{Harrison} was as a quotient of the bar complex modulo shuffles, in particular quotienting by
cycles of the form 
$s^\inv a | s^\inv b \,\mp\,  s^\inv b | s^\inv a$, 
which lie in the kernel of $\mu_A$ tautologically by commutativity of $A$.  
Our presentation of Harrison's complex as part of an explicit operadic framework is more flexible
and pairs naturally with Lie algebraic constructions, as described in \cite[Sec 2]{Sinha-Walter1}.  
There is a surjection from the classical bar complex mapping tensors to linear trees 
(the ``long graphs'' of \cite{Sinha-Walter1})
$$ s^\inv a_1|s^\inv a_2|\cdots | s^\inv a_n  \longmapsto
 \tpc[\Big] {\!\cdots\!\tpc[\big] {\tpc {s^\inv a_1} s^\inv a_2} \cdots } s^\inv a_n.$$
Because these span $\Eil\Tr_*(n)$, the bar classes above indeed span
$\cofreeE(s^\inv \bar A)$.  
Furthermore, mapping bar expressions to long trees preserves coproduct -- sending the 
cut coproduct of bar expressions to the excision coproduct on linear trees, as shown in 
the second part of Example~\ref{Ex for Greg}.

In the setting of $1$-connected
differential graded commutative algebras we show in \cite[Section 4.1]{Sinha-Walter1} that the functor $\mE$ is part of a Quillen adjoint pair.  
In Proposition~2.10 of \cite{LieEncyclopedia}, the authors show that even without any
connectivity hypotheses, $\mE$ is quasi-isomorphism invariant,
which leads to the following.

\begin{definition}
Let $X$ be a connected topological space. Define the {\bf  Harrison cohomology of $X$}, denoted 
$H_\mE^*(X)$, to be the cohomology of $\mE(C^*(X))$ for
any augmented commutative cochain algebra $C^*(X)$ 
quasi-isomorphic as differential graded associative algebras to the singular cochains
of $X$.  

\end{definition}

\begin{example}\label{WedgeCircles}
    Because a wedge of circles is formal, we may replace cochains by cohomology, which in addition to having zero differential has trivial product.  Thus the zeroth 
    Harrison cohomology of a wedge of circles is canonically isomorphic to the cofree Lie coalgebra on the first cohomology:
    \[
    H_\mE^0\bigl(\bigvee_{s\in S} S^1\bigr) \cong \cofreeE\bigl(\Q^S\bigr).\]
\end{example}


 Harrison cohomology provides a {dual} to $\pi_1(X)$ because the two admit a natural
pairing, which we now define.  The definition works equally well
for general homotopy groups, building on the following.

\begin{proposition}\label{Ground}
    For every $n\geq 1$ there is a canonical isomorphism 
    \[H^{n-1}_\mE(S^n) \cong  H^n(S^n) \xrightarrow[\int_{S^n}]{\cong} \Q.\]
\end{proposition}

\begin{proof}
    Since spheres are formal, we may pick $H^*(S^n)= \Q[\alpha]/(\alpha^2)$ with $d(\alpha) = 0$ as our cochain model. Then,
    \[
    \mE(s^\inv \bar{C}^*(S^n)) = \mE(\Q\langle s^\inv \alpha \rangle)
    \]
    is the cofree Lie coalgebra on a single element of degree $n-1$, with zero differential. It follows that $H^{n-1}_\mE(S^n) \cong  \Q\langle s^\inv \alpha \rangle$. Evaluation on the fundamental class of $S^n$ defines the second claimed isomorphism.
    
\end{proof}

Specializing to the case $n=1$, we introduce an invariant of loops in a space.
\begin{definition}\label{Hopf}
    Let $[\gamma] \in H^{0}_{ \mE}(X)$ and 
    $f: S^1 \to X$ a pointed map representing $[f]\in \pi_1(X,*)$. 
    The {\bf Hopf invariant} associated
    to $\gamma$ is the function
    $h_\gamma \maps [f] \mapsto \int_{S^1} f^* \gamma $.

    Equivalently, the Hopf invariant defines a pairing $H^{0}_{ \mE}(X)\times \pi_1(X) \to \Q$
    by
    \[
    \bigl\langle [\gamma], [f] \bigr\rangle \; = \; \int_{S^1}f^*(\gamma).
    \]
\end{definition}

In practical terms,
to perform this evaluation we pull back $\gamma$ to the Harrison complex
of $S^1$ and then find a homologous Harrison cocycle in weight one, which must
exist by Proposition~\ref{Ground}.   Then evaluate the corresponding cocycle as an ordinary cochain
on the fundamental class of $S^1$, typically by counting points or by integration.  The process of finding a homologous class in weight one is called {\bf weight reduction}, often performed inductively.
For weight one Harrison cocycles, which are ordinary 1-cocycles of $X$, Hopf invariants simply
evaluate $H^1(X)$ on the fundamental group.
Evaluation of higher weights can be viewed as  derived versions of this, and it can detect nontrivial elements deep in the lower central series.

Hopf invariants are equally effective regardless of which space is used to present a given fundamental group.

\begin{lemma}
If $f\maps X'\to X$ is a map of pointed connected spaces inducing an isomorphism on $\pi_1$, then
\[
H^0_\mE(f) \maps H^0_\mE(X') \to H^0_\mE(X)
\]
is an isomorphism.
\end{lemma}
\begin{proof}
    Let $f^*\maps C^*(X')\to C^*(X)$ be the induced map of Sullivan minimal models for $X$ and $X'$. Then the hypothesis implies it induces an isomorphism on cohomology in degrees zero and one, and an injection in degree two. It follows that its desuspension $s^{\inv} f^*$ is $0$-connected. By the K\"unneth isomorphism and exactness of $\sym_n$-coinvariants over $\Q$, the induced maps
    \[
    \raisebox{3pt}{$\bigl(\Eil{\Tr_*}(n) \otimes s^{\inv}\bar C^*(X')^{\otimes n}\bigr)$}
    \!\raisebox{-2pt}{$\diagup$}\!
    \raisebox{-4pt}{$\sym_n$} \to 
    \raisebox{3pt}{$\bigl(\Eil{\Tr_*}(n) \otimes s^{\inv}\bar C^*(X)^{\otimes n}\bigr)$}
    \!\raisebox{-2pt}{$\diagup$}\!
    \raisebox{-4pt}{$\sym_n$}
    \]
    are similarly $0$-connected for all $n\geq 1$. From here, a standard spectral sequence argument shows that $\mE(f^*) \maps \mE(C^*(X')) \to \mE(C^*(X))$ is $0$-connected, and in particular induces an isomorphism on $H^0_\mE$.
\end{proof}

We will take full advantage of the freedom to choose $X$ with fundamental group $G$.  In some cases, there are good
manifold models, in others we can directly analyze 
two-complexes, as we do in Section~\ref{examples} using   Appendix~\ref{Thomforms}.
If
we take $X$ as the simplicial model for $BG$ and
use PL-forms then we immediately get functorality with
respect to group homomorphisms.

\begin{remark}\label{HarrisonSS}
The central role of $H^{0}_{ \mE}(X)$ in our theory motivates its calculation.  A standard tool is the spectral sequence of a bicomplex, which leads to the {\bf Harrison spectral
sequence}.  (See, among many sources, \cite{Elias} for an introduction to the spectral sequence of a bicomplex in general and Appendix~\ref{StandardBar} for an introduction to the
corresponding spectral sequence for the Bar complex.)
In degree zero, the $E_1$-page of the Harrison spectral sequence is identified with the cofree Lie coalgebra
on the first cohomology of $X$, namely $\cofreeE (\pi_1(X)^\#)$ cogenerated by the linear dual $\pi_1(X)^\# = \Hom_{Grp}(\pi_1(X),\Q) \cong H^1(X; \Q)$.

Since the cohomology of the augmentation ideal for cochains of a connected space vanishes in degree zero, the Harrison $E_1$-page  vanishes in negative total degree.  Thus no differentials land in total degree zero.
So it is the vanishing 
of differentials out of $\cofreeE (\pi_1(X)^\#)$, which are essentially 
Massey products of $H^1(X)$-classes, that ultimately dictate
the structure of the zeroth Harrison cohomology. 
Relations in a group give rise to products in cohomology,
thus ``more" products mean ``smaller" fundamental group.  Wedges of circles are an extreme case in which all products vanish. 
They are also universal: the fact that wedges of circles can induce surjections onto the fundamental group of any space is dual to the observation that all Hopf invariants inject into those of wedges of circles through the Harrison spectral sequence.

At the other extreme, products of circles and their abelian fundamental group have injective $d_1$-differential, killing all but weight one elements. In this case the only Hopf invariants come directly from cohomology:
\[
    H_\mE^0\biggl(\prod_{s\in S} S^1\biggr) \cong H^1 \biggl(\prod_{s\in S} S^1\biggr)  \cong \Q^S
\]
with zero cobracket.

See Section~\ref{commutatormassey} for an intermediate example.  Explicitly working through the Harrison spectral
sequence in this case with its $d_2$ differential is a useful exercise.
\end{remark}

\section{Central Example: Handlebodies and linking of letters for free groups}\label{FreeCase}

We connect the Hopf invariant framework of the previous section with the purely combinatorial approach to group theory introduced in Section~\ref{combalg} and 
developed in \cite{monroe-sinha}.  
Ultimately, this section establishes
the (co)free case of our
main structural theorems, which will be used in proving them in Section~\ref{structural}.

\subsection{First handlebody example}\label{first}

Recall that if $M$ is a
manifold and $W$ is a codimension-$k$ submanifold, there
is a Thom $k$-form which is 
supported in a neighborhood of 
$W$ and ``counts intersections with $W$'' multiplied
by a scalar we call the {mass}.  
See Appendix~\ref{Thomforms}.  
The simplest
example is the most relevant to
our theory, namely $M = S^1$ where $W$ is a finite collection
of points.  A Thom form for $W$ is a sum of ``bump forms'' $\omega = \Sigma f_i \,  d\theta$, 
with  each  $f_i$ supported in a 
neighborhood of a point. 
If the total integral of such a Thom form is zero, then there is a cobounding
anti-derivative zero-form which we call $d\,^\inv \omega$, uniquely determined by the requirement that it vanish at the basepoint $1\in S^1$ (as shown in Figure~\ref{fig:bumpforms}). 
This anti-derivative is a ``mostly constant'' function
whose values are the partial sums of the masses.

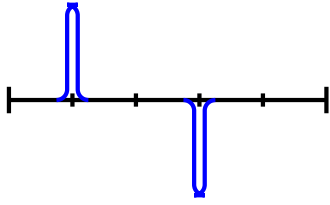
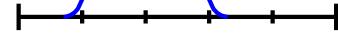
\begin{figure}[h]
	\centering
	\begin{subfigure}[b]{0.40\textwidth}
        \begin{tikzpicture}[scale=0.35]
\draw[ultra thick] (0,0)--(12,0);
\draw[ultra thick](0,-.5)--(0,.5);

\draw[ultra thick] (2.4,-.25)--(2.4,.25);
\draw[ultra thick] (4.8,-.25)--(4.8,.25);
\draw[ultra thick] (7.2,-.25)--(7.2,.25);
\draw[ultra thick] (9.6,-.25)--(9.6,.25);

\draw[ultra thick] (12,-.5)--(12,.5);

\draw[blue, ultra thick,rounded corners] (1.8,0) --(2.2,0)--(2.2,3.6) --(2.6,3.6)--(2.6,0)--(3,0);
\draw[blue, ultra thick ,rounded corners] (6.6,0)--(7,0)--(7,-3.6) --(7.4,-3.6)--(7.4,0)--(7.8,0);

\end{tikzpicture}
        \vspace{-.6cm} 
		\caption{\label{fig:thom of points}A  Thom form $\alpha$ for two points on the circle, with masses $\pm 1$.}
	\end{subfigure}
	\hspace{0.5cm}
	\begin{subfigure}[b]{0.40\textwidth}
    \centering
        \begin{tikzpicture}[scale=0.35]

\draw[ultra thick] (0,0)--(12,0);
\draw[ultra thick] (0,-.5)--(0,.5);

\draw[ultra thick] (2.4,-.25)--(2.4,.25);
\draw[ultra thick](4.8,-.25)--(4.8,.25);
\draw[ultra thick] (7.2,-.25)--(7.2,.25);
\draw[ultra thick] (9.6,-.25)--(9.6,.25);

\draw[ultra thick] (12,-.5)--(12,.5);

\draw [blue,ultra thick,rounded corners] (1.7,0) to(2.1,0) to (2.6,1.1)to (7.0,1.1)to (7.5,0) to (7.9,0);

\end{tikzpicture}
        \vspace{.5cm} 
    \caption{\label{fig:thom of interval}The antiderivative of $\alpha$, denoted $d\,^\inv \alpha$ is a Thom form for the interval.}
	\end{subfigure}
	\caption{Basic Thom forms on the circle.}
	\label{fig:bumpforms}
\end{figure}

Consider for the moment a general smooth manifold $M$ with  $\{ W_i \}$  
a collection of codimension one submanifolds.
If $f: S^1 \to M$ is a smooth loop transverse to $W_i$, which is a generic
condition, then the pullback of a Thom form for $W_i$ is a
Thom form for the intersection points $f^\inv(W_i)\subset S^1$.  
For example, if we take the Borromean rings as presented in Figure~\ref{borromean3} and the Thom form for the red Seifert
surface,  the pullback will look like the Thom form pictured in Figure~\ref{fig:thom of points}.

In Example~\ref{WedgeCircles} we saw that for wedges
of circles the zeroth Harrison cohomology is a 
cofree Lie coalgebra.  We now calculate  the
second simplest associated Hopf invariant, as 
 evaluation on ordinary cohomology is simplest.
Write $B_m$ for the
three-dimensional handlebody obtained
by adding $m$ one-handles to a three-ball, yielding a $3$-manifold with boundary homotopy equivalent to  $\bigvee_m S^1$.
Focusing on $B_2$, label the two handles $a$ and $b$.  
By abuse we use
the same letters for the generators of $\pi_1$ represented by loops
along their core disks.
Let $D_a$ be 
a belt disk of the handle $a$ as in Figure~\ref{handlebody}, and let $A$ be a
Thom form for $D_a$ of mass one.
The form $A$ counts the number of times, with signs, a closed loop winds around the handle $a$.
Similarly define $D_b$ and $B$ for the handle $b$.

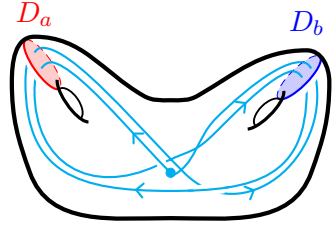
\begin{figure}
  \begin{tikzpicture}[scale=0.60]

\filldraw[cyan] (0,0) circle (1mm);
\node[red] at (-3,3+.5) {$D_a$};
\node[blue] at (3,3+.3) {$D_b$};

\draw[thick,cyan]{(0,-.085) to[out=25,in=225] 
                     (1,1.2) to[out=45,in=100]
                     (3.2,2.1)};         
\begin{knot}[
consider self intersections=true, 
clip width=5,
clip radius=6pt,looseness=1.3]
\strand[thick,cyan] (3.2,2.1)  to[out=-80,in=-45]
                     (.3,0) to[out=135,in=90] 
                     (-3.2,2) to[out=-90,in=180] 
                     (0,-.4) to[out=0,in=-90] 
                     (3,1.8) to[out=90,in=45] 
                     (1,.9) to[out=225,in=25]
                     (-.1,.25);            
\strand[thick,cyan]  (-.3,.16) to[out=205,in=-90] 
                     (-3,1.8) to[out=90,in=135] 
                     (0,0);   
\flipcrossings{1}
\end{knot};

\begin{scope}[thick,decoration={
    markings,
    mark=at position 0.3 with {\arrow[scale=1,cyan]{ang 90}}}
    ] 
    \draw[draw=none,postaction={decorate},yshift=.4mm] (0,0) to[out=135,in=90] (-3,1.8);
    \end{scope}
    
\begin{scope}[thick,decoration={
    markings,
    mark=at position 0.4 with {\arrow[scale=1,cyan]{ang 90}}}
    ] 
    \draw[draw=none,postaction={decorate},yshift=-2.74mm,xshift=3mm] (0,-.4) to[out=0,in=-90] (3,1.8);
\end{scope}

\begin{scope}[thick,decoration={
    markings,
    mark=at position 0.1 with {\arrow[scale=1,cyan]{ang 90}}}
    ] 
    \draw[draw=none,postaction={decorate},yshift=.11mm,xshift=-4mm] (0,-.4) to[out=180,in=-90] (-3.2,2);
\end{scope}

\begin{scope}[thick,decoration={
    markings,
    mark=at position 0.4 with {\arrow[scale=1,cyan]{ang 90}}}
    ] 
    \draw[draw=none,postaction={decorate},yshift=0mm,xshift=-.65mm] (1,.9) to[out=45,in=90] (3,1.8);
\end{scope}

\filldraw[white,rotate around={(125:(-3.2,2.72)}] (-3.2,2.72)  arc[start angle=0, end angle=360, x radius=.65, y radius=.2];
\filldraw[white,xshift=-.95cm,yshift=-.95cm,rotate around={(225:(3.3,2.37)}] (3.3,2.37)  arc[start angle=0, end angle=360, x radius=.65, y radius=.2];

\draw[fill=red!20,draw=none,rotate around={(125:(-3.2,2.92)}] (-3.2,2.92)  arc[start angle=0, end angle=360, x radius=.65, y radius=.2];
\draw[thick, red, rotate around={(125:(-3.2,2.92)}] (-3.2,2.92)  arc[start angle=0, end angle=180, x radius=.65, y radius=.2];
\draw[line width=.1mm,red,densely dashed,rotate around={(125:(-3.2,2.92)}] (-3.2,2.92)  arc[start angle=0, end angle=-180, x radius=.65, y radius=.2];

\draw[very thick,xshift=.66cm,yshift=-.9cm,rotate around={(125:(-3.16,2.95)}] (-3.16,2.95)  arc[start angle=20, end angle=170, x radius=.6, y radius=.2];
\draw[line width=.25mm,xshift=.71cm,yshift=-1.1cm,rotate around={(125:(-3.16,2.95)}] (-3.16,2.95)  arc[start angle=0, end angle=-180, x radius=.4, y radius=.2];

\draw[xshift=-.95cm,yshift=-.95cm,fill=blue!20,draw=none,rotate around={(225:(3.43-.2,2.5)}] (3.3-.2,2.5)  arc[start angle=0, end angle=360, x radius=.65, y radius=.2];
\draw[xshift=-.95cm,yshift=-.95cm,thick,blue, rotate around={(225:(3.43-.2,2.5)}] (3.3-.2,2.5)  arc[start angle=0, end angle=180, x radius=.65, y radius=.2];
\draw[xshift=-.95cm,yshift=-.95cm,line width=.1mm,blue,densely dashed,rotate around={(225:(3.43-.2,2.5)}] (3.3-.2,2.5)  arc[start angle=0, end angle=-180, x radius=.65, y radius=.2];

\draw[very thick,xshift=-1.6cm,yshift=-1.635cm,rotate around={(225:(3.3,2.6)}] (3.3,2.6)  arc[start angle=20, end angle=170, x radius=.6, y radius=.2];
\draw[line width=.25mm,xshift=-1.43cm,yshift=-1.56cm,rotate around={(225:(3.3,2.6)}] (3.3,2.6)  arc[start angle=0, end angle=-180, x radius=.4, y radius=.2];

\draw[cyan,thick] (-2.85,2.3) to[out=55,in=175] (-2.6,2.435);
\draw[cyan,thick] (-3,2.65) to[out=55,in=175] (-2.85,2.75);
\draw[cyan,thick] (2.9,2.1) to[out=120,in=0] (2.67,2.24);
\draw[cyan,thick] (3.1,2.3) to[out=120,in=0] (2.9,2.435);

\draw[ultra thick]  (0,1.6) to[out=0,in=210]
                    (.6,-.3+2) to[out=30,in=180] 
                    (3,.5+2.2) to[out=0,in=90] 
                    (3.5,0+2) to[out=-90,in=70] 
                    (3.1,-2+2) to[out=250,in=00] 
                    (0,-3+2) to[out=180,in=-70]
                    (-3.1,-2+2) to[out=110,in=-90] 
                    (-3.5,0+2) to[out=90,in=180] 
                    (-3,.5+2.5) to[out=0,in=150] 
                    (-.6,-.3+2) to[out=-30,in=180] 
                    (0,1.6);
%

\end{tikzpicture}

  \vspace{-1cm}  
\caption{A commutator in a handlebody}\label{handlebody}
\end{figure}

Consider the loop  $ f\maps S^1\to B_2$ with $[f]  = a b a^\inv b^\inv \in \pi_1(B_2)$ shown in Figure~\ref{handlebody}.
We calculate the Hopf invariant associated to  $ \tpc A B$ on $f$. (Readers who are, in a first reading, substituting the bar complex
for the Harrison complex can subsitute $A | B$
for $\tpc A B$.)
With appropriate choices of parametrization and Thom forms,
$f^* A = \alpha$, where $\alpha$  is shown in Figure~\ref{fig:thom of points}, and we
set $\beta = f^* B $, given by the horizontal shift 
of $\alpha$ by $+\frac{1}{5}$.  Thus
$f^* \tpc A B = \tpc \alpha \beta$.
 Using $d\,^\inv \alpha$ from Figure~\ref{fig:thom of interval},
 the Harrison $(-1)$-cochain 
$\tpc{d\,^\inv \alpha}{\beta}$
will have total differential  
\begin{equation}\label{SimpleWeightReduction}\begin{split}
d_{\mE} \Bigl(\tpc
{d\,^\inv \alpha}{\beta}\Bigr) &= \tpc{d (d\,^\inv\alpha)}{\beta} 
+ \tpc{d\,^\inv\alpha}{d \beta} - \ (d\,^\inv \alpha) \wedge   \beta  \\
&=\qquad  \tpc{\alpha}{\beta} \qquad \qquad \quad - \ (d\,^\inv \alpha) \wedge \beta.  
\end{split}\end{equation}
This implies that $ \tpc \alpha \beta$ is homologous
to $(d\,^\inv \alpha) \wedge \beta$ in the Harrison complex
of the circle.  The product 
$(d\,^\inv \alpha) \wedge \beta$ is a form of mass one supported near $t = \frac{2}{5}$, so we 
conclude that 
\[\bigl\langle \tpc A B,\, f  \bigr\rangle \ = \ 
\int_{S^1} d\,^\inv \alpha \wedge \beta \ = \ 1.\]
This nonzero value is to be contrasted with any cohomological evaluation, which vanishes on commutators.

The calculation can immediately be generalized.  
If $\langle A, f \rangle = 0$ then we can
pair $a$'s with $a^\inv$'s in the word $[f]\in \pi_1(B_2)$. For any such $a$-$a^\inv$ pairing,
$\bigl\langle \tpc A B,\, f \bigr\rangle$ 
is equal to the 
total number of $b$'s which occur between $a$-$a^\inv$ pairs, counted with appropriate signs.
This count can be seen as the linking 
number of the zero-manifolds $f^\inv D_a$ and $f^\inv D_b$ inside $S^1$,
as explained at the very beginning
of \cite{monroe-sinha}.
Such a count also arose for the letter braiding 
invariant $\lp{A}B$ in Example~\ref{ex:first braiding example}, a relationship we will fully expand and establish in this section.

\subsection{Equating topology and combinatorics of letter linking}\label{topcombo}

In Section~\ref{combalg} we used simple combinatorics to define letter braiding functions on free groups.  In Section~\ref{HopfInvar} we defined Hopf
invariants topologically.
We now show that these coincide when considering 
fundamental groups of handlebodies and restricting
to ``good'' symbols.

Let $F_S = \langle s_1, s_2, \cdots, s_m \rangle$ be a free group. We present it as the fundamental group of a handlebody $B_m$, with handles corresponding to the $s_j$, whose first cohomology is ${F_S}^\#= \Hom_{Grp}(F_S,\Q)$.  Then ${F_S}^\#$ is spanned by the indicator homomorphisms $S_j$, respectively sending $s_j$ to $1$ and all other generators to $0$.

The braiding symbols defined in Section \ref{combalg}
 map to the Harrison
homology of $B_m$, essentially by changing $\lp {\;} {}$'s to $\tp {\;} {\!}$'s after
applying multi-linearity, as follows. (Recall that tree symbols $\tp{\,}$ do not allow
nested sums, though braiding symbols $\lp{\;}$ do.)

\begin{definition}
    Let $\br_{S}$ denote the vector space spanned
    by braiding symbols for the free group  $F_S$, defined in Section \ref{combalg}, modulo multi-linearity.

    A formal braiding symbol is {\bf homogeneous} if it is either a single variable or, inductively, it is a formal braiding product of homogeneous braiding symbols (or equivalently, if it does not contain nested sums). A braiding symbol is homogeneous when its formal symbol is.
\end{definition}

Elements of $\br_S$ can be written as sums of homogeneous symbols.
On the other hand, since a handlebody is formal and has trivial cup products, there is a canonical isomorphism $H^0_\mE(B_m) \cong \E (H^1(B_m))$. 
Therefore, the Hopf invariants of $F_S$ are represented by rooted trees with vertices decorated by homomorphisms $h_v\maps F_S\to \Q$, just as homogeneous braiding symbols are.

\begin{definition}
Let $|\!\cdot\!|$ be the map of vector spaces that sends a homogeneous
braiding symbol to the  class it represents in $H^0_\mE(B_m)$, converting $\lp {\;} {}$'s to $\tp {\;} {\!}$'s.
\end{definition}

While letter braiding invariants of words and Hopf invariants of (free) groups are 
both defined for general braiding and tree symbols, 
letter braiding invariants only agree with Hopf invariants
in the following  special setting.  

\begin{definition}\label{def:lp}
A letter braiding product $\lpw f g$ is said to be a {\bf linking product} when $\int f = 0$.

A word $w$ is an {\bf eigenword} for a braiding symbol $\sigma_{\vec{h}}$ if all of the constituent braiding products in $w^*\sigma_{\vec{h}}$
are linking products.  The restriction of a letter braiding
function to the subset of eigenwords is called a {\bf letter linking  function}. 
\end{definition}

The term ``letter linking'' arises as  discussed at the end of Section~\ref{first}. For example, we consider the indicator homomorphism $A$ on a word where 
the total count of $a$'s is zero.  Putting ``strings''
between $a$-$a^\inv$ pairs, the values
of $\lpw A$ record how many such strings croass
a given letter  
of $w$.  If the total count of $a$'s is not zero then we do not have such ``closed strings'' connecting $a$'s and $a^{-1}$'s, so the calculation involves ``open strings'', hence the term letter braiding.

Let $f_w : S^1 \to B_m$ denote a map
which represents $w$ in $\pi_1(B_m)$, a
word of length $n$, by traveling through the core of each 
handle at constant speed and returning to the basepoint at times $\frac{k}{n}$.

\begin{theorem}\label{freecase}
If $w$ is an eigenword for  $\sigma_{\vec{h}}$ then 
\[
\sigma_{\vec{h}}(w) \ = \ \int_{[n]} w^*\sigma_{\vec{h}} \ = \ 
\int_{S^1} {f_w}^*\, |\sigma_{\vec{h}}| .
\]
\end{theorem}

The first equality is the definition of letter braiding functions. We gave an  example of the second equality in Section~\ref{first}.  
The rest of this section is devoted to the proof in general.
We relate the two evaluations on $w$ by carefully choosing cochain
representatives for $H^0_\mE(B_m)$ generalizing the constructions in Section~\ref{first}.  To do so is elementary, but making the equality 
apparent requires some development.

First, we define a family of Thom forms of belt disks. Fix a non-negative one-form $\varphi$ with support on the
interior of $[0,1]$ and total
integral one.  For subintervals $J \subset [0,1]$, 
let $\varphi_J$ be a one-form given by zero outside $J$ and otherwise
the pull-back of $\varphi$ by an order-preserving linear homeomorphism 
 $J\to[0,1]$.
On a handle $D^2 \times [0,1]$, let $\Phi_J$ denote the pullback of $\varphi_J$ along the projection to the interval.

To consider all possibilities for the placement of such forms on a handle, let $\kappa_\ell$ be the set of $\ell$-tuples of subintervals of $[0,1]$. Explicitly, $\kappa_
\ell$ is in natural bijection with $(\Delta^2)^{\times \ell}$. (These sets  have a natural operadic structure, which contains the little one-disks as a sub-operad, but we will not use this fact.)

Braiding symbols are generated under the braiding product by homomorphisms. Consider linear combinations of such symbols and extend the braiding product to be bilinear.
Then letter braiding symbols are spanned by the $\sigma_{\vec{h}}$ with each $h_v$ an indicator homomorphism $S_j$ for some $j$.  
Call such elements {\bf indicator symbols}.

We now construct a cochain-level version of the map 
$|\cdot|$ above which encodes the correspondence between braiding symbols and the Harrison cohomology of handlebodies. 

\begin{definition}
Let $\sigma_{\vec{h}}$ be a weight $\ell$ indicator symbol and 
$\vec{J} = (J_1,\cdots, J_\ell) \in \kappa_\ell$. 
 Define the realization $\rho(\sigma_{\vec{h}}, \vec{J}) \in \mE\bigl(\Omega_{dR}(B_m)\bigr)$ 
 to be the Harrison cochain represented by the same rooted tree (changing $\lp {\;} {}$'s to 
 $\tp {\;} {\!}$'s), whose vertex is decorated by the form
 $\Phi_{J_v}$ supported on the handle indexed by $s_j$ if the vertex $v$ of $\sigma_{\vec{h}}$ 
 has indicator function $h_v=S_j$. 

Extend the realization $\rho$ linearly to all symbols to obtain a map 
$\rho: \cofreeE^{\ell}(F_S^\#) \times \kappa_\ell \to \mE\bigl(\Omega_{dR}(B_m)\bigr)$.

\end{definition}

    In summary, we realize braiding symbols at the cochain level as linear combinations of Thom forms for belt disks which are supported on sets of intervals parametrized by $\kappa_\ell$. 
Below, we utilize particular choices of such realizations.

Following \cite{Dendroidal} we order vertices of rooted trees using the reverse of the 
tree-order. That is, leaves are minimal, the root is the unique maximal element, and $u<v$ if $v$ is on 
the path from $u$ to the root.
Call a pair $(\sigma_{\vec{h}}, \vec{J})$ {\bf order
compatible} if whenever $u < v$ in $\sigma$, the interval corresponding to $u$ is entirely to the \textbf{right} of the one corresponding to $v$.
We claim that an order compatible realization 
$\rho(\sigma_{\vec{h}}, \vec{J}) \in \mE\bigl(\Omega_{dR}(B_m)\bigr)$ pulls back to the circle along the map $f_w\maps S^1\to B_m$ so that this pullback is homologous to a single form that matches the d-function $w^*\sigma_{\vec{h}}$ combinatorially.

To make this precise, we say a form on the interval {\bf represents} a d-function $f\maps [n]\to \Q$ if it is supported on the interiors the sub-intervals $(\frac{i-1}{n}, \frac{i}{n})$ and its mass on the $i$-th sub-interval is $f(i)$. 
Clearly, the integral over $S^1$ of a form which represents a d-function is equal to the discrete integral over $[n]$ of that
d-function.  Thus Theorem~\ref{freecase} is an immediate corollary of the following.

\begin{theorem}  Suppose $w$ is an eigenword for  $\sigma_{\vec{h}}$, and
    let $\vec{J} \in \kappa_\ell$ be order compatible with $\sigma_{\vec{h}}$.  Then the pullback of the realization $f^*_w \, \rho (\sigma_{\vec{h}}, \vec{J})\in \mE\bigl(\Omega_{dR}(S^1)\bigr)$ weight reduces to a form
    which represents $w^* \sigma_{\vec{h}}$.
\end{theorem}

\begin{proof}
We will prove the following stronger statement by induction on weight. Given a subinterval $J\subseteq (0,1)$, 
let $J(i)\subset (\frac{i-1}{n}, \frac{i}{n})$ denote the image of $J$ under either an {increasing} linear homeomorphism $(0,1)\to (\frac{i-1}{n}, \frac{i}{n})$ if $w$ has $\epsilon_i=1$, or else a {decreasing} linear homeomorphism if $\epsilon_i=-1$. Then the pullback $f^*_w \,\rho (\sigma_{\vec{h}}, \vec{J})$  can be weight reduced to a form whose restriction to the interval $(\frac{i-1}{n}, \frac{i}{n})$ has total mass $w^*\sigma_{\vec{h}}(i)$ and is supported on the interval $J_\sigma(i)$, where $J_\sigma\in \vec{J}$ is the interval corresponding to the root vertex.



In weight one, no weight reduction is necessary. If $h$ is a homomorphism decorating the unique
vertex in $\sigma$ then the realization $\rho(h,J_1)\in \Omega^1(B_m)$ is the form restricting to each handle $s_j$ as some bump form with mass $h(s_j)$. Thus its pullback along the loop $f_w$ is indeed a sum of bump forms with respective mass $h(w_i)$ supported in the sub-interval $J_1(i)$. Note that if a letter in $w$ is the inverse of a generator $s_j^{-1}$ then the loop $f_w$ travels through the handle $s_j$ backwards, mapping through the order reversing homeomorphism $(\frac{i-1}{n},\frac{i}{n}) \to (0,1)$.

Proceeding by induction, consider a braiding product $\sigma = \lp {\alpha} \beta$, suppressing 
decorating homomorphisms from the notation. Note that the eigenword assumption means in particular that $w$ is an eigenword for both $\alpha$ and $\beta$.

Let $J_{\alpha}\in \vec{J}$ denote the sub-interval corresponding to the root of $\alpha$. Then by the inductive hypothesis the pullback $f^*_w\,\rho(\alpha,\vec{J})$ can be weight reduced to a form $\alpha_0$ on $S^1$ with support in $J_{\alpha}(1)\cup \ldots\cup J_{\alpha}(n)$ and total mass $\int_{[n]} w^*\alpha$. 
The eigenword assumption implies that the total mass is zero. Therefore $\alpha_0$ has an anti-derivative $d\,^\inv \alpha_0$ as a form (as in Figure~\ref{fig:thom of interval}), constant outside of the intervals $J_{\alpha}(i)$, with values given by the partial sums
of the masses of $\alpha_0$. 
Similarly, $f^*_w\,\rho(\beta,\vec{J})$ can be weight reduced to a form $\beta_0$ supported on intervals $J_\beta(i)$ for $1\leq i\leq n$. 

It follows that the entire pullback $f^*_w\,\rho(\lp {\alpha} \beta)$ can be weight reduced to the form $d\,^\inv\alpha_0 \wedge \beta_0$. This product is immediately seen to have support contained in that of $\beta_0$, as claimed. Furthermore, the intervals $J_\alpha(i)$ are disjoint from $J_\beta(i)$ (from the order compatibility assumption), so the form $d\,^\inv\alpha_0$ is locally constant on the support of $\beta_0$.

Finally, we  calculate the mass of the weight reduced form on $(\frac{i-1}{n}, \frac{i}{n})$. Since $\vec{J}$ was chosen to be order compatible with $\lp \alpha \beta$, the interval $J_\alpha$ is entirely to the right of $J_\beta$. 
After applying the homeomorphism $(0,1)\to (\frac{i-1}{n}, \frac{i}{n})$, we see that $J_\alpha(i)$ is still to the right of $J_\beta(i)$. In this case, the (constant) value of $d\,^\inv\alpha_0$ on $J_\beta(i)$ is the partial sum of $w^*(\alpha)$ up to \emph{but not including} the $i$-th value. 
However, if $\epsilon_i=-1$ then orientation is reversed so that $J_\alpha(i)$ is to the left of $J_\beta(i)$, which implies that the value of $d\,^\inv\alpha_0$ on $J_\beta(i)$ now includes the $i$-th value. This is exactly the form of the d-function $\lpw {w^*\alpha}$ multiplying the d-function $w^*\beta$, thus
agreeing with the braiding product
from Definition~\ref{BraidingProduct},
as claimed.
\qedhere

\end{proof}

\subsection{Configuration braiding of trees and words}\label{sec:configuration braiding}
Rooted trees give
an alternate way to view the evaluation of letter braiding functions
$\sigma_{\vec{h}}(w)$
from Section~\ref{combalg}, an
approach related to 
to the  configuration pairing of \cite{Sinha-Walter1}.
Let $w=s_1^{\epsilon_1}\dots s_n^{\epsilon_n}$ be a word in $F_S$
and $T$ be a rooted tree with vertices labeled by homomorphisms 
$(F_S)^\#$. 
As noted in the previous subsection,
braiding symbols $\sigma_{\vec{h}}$ are  equivalent to (sums of) such rooted trees by
formally converting $\lp \,$ to $\tp \,$.

\begin{definition}
A {\bf configuration} of $T$  
is a non-decreasing function $c:\mathrm{Vertices}(T)\to [n]$ with respect to the
(root-maximal) partial ordering on $\mathrm{Vertices}(T)$. 

A configuration is {\bf proper with respect to} $w$ if
$c$ is strictly increasing at $i$ whenever $\epsilon_i=1$. 
That is, $u<v$ implies $c(u)\lew c(v)$.

The {\bf associated pairing} is
$\langle T, w\rangle_c = \prod h_v\bigl(w_{c(v)}\bigr)$ where 
$h_v\in (F_S)^\#$ is the label of vertex $v$ and $w_i=s_i^{\epsilon_i}$. 
\end{definition}


\begin{proposition}\label{prop:configuration braiding}
   Letter braiding functions count proper configurations of $\sigma_{\vec h}$ in $w$.
   $$\displaystyle  
                \sigma_{\vec h}(w)
                = \hspace{-.5cm}\sum_{\substack{ \mathrm{proper}\\ \mathrm{configuration}}}\hspace{-.5cm}
                  \bigl\langle \sigma_{\vec h},\, w\bigr\rangle_c$$
\end{proposition}

The restriction to proper configurations is necessary to account for the 
use of $\lew$ when computing $\lpw f$ in Definition~\ref{BraidingProduct}. 
The proposition follows directly from the observation that coboundings of indicator d-functions 
are given by signed counts of prior letter occurrences.
\[
 \lpw A (j) \ = \!\!\sum_{\substack{i \,\lew\, j \vphantom{|}\\[2pt] s_i\  =\  a\phantom{x}}} \!\!\epsilon_i 
\]

    It is instructive to compare $\lp A A$ evaluated on $aa$ versus $a^\inv a^\inv$.  On $aa$ there is only one proper configuration $\lp 1 2$, leading to a value of $1$.  On $a^\inv a^\inv$ there are three proper configurations
    $\lp 1 1$, $\lp 1 2$, and $\lp 2 2$, leading to a value of $3$.

\begin{example}
  Consider evaluating $\lp A B$ on the word $aba^\inv b^\inv$.  Only three configurations 
  yield nonzero pairings.  
  \begin{itemize}
      \item $\lp A B \mapsto \lp 1 2$, contributing $A(a)\,B(b) = 1$
      \item $\lp A B \mapsto \lp 1 4$, contributing $A(a)\,B(b^\inv)=-1$
      \item $\lp A B \mapsto \lp 3 4$, contributing $A(a^\inv)\,B(b^\inv)=1$
  \end{itemize}
  The sum of these is $1$, recovering $\lp A B\,(aba^\inv b^\inv)$.
\end{example}

\section{Structural theorems}\label{structural}

We establish the main formal theorems governing Hopf invariants, starting with the fact that the Hopf pairing establishes a Lie coalgebraic duality, then 
the Lifting Criterion which relates Hopf
invariants of two complexes and their one skeleton,
and finally we establish universality using the
Lifting Criterion.

\subsection{Lie pairing}\label{sec:lie pairing}

For this argument we use the well-known approach to Lie coalgebras through coassociative algebras and Hopf algebras.  
We build on invariants of fundamental groups  in the  (co)associative setting, which were developed
by Chen.

We first recall the standard presentation of the
Malcev Lie algebra. For a group $G$ let $\widehat{\Q[G]}$ be the completion of its group ring at the ideal generated by elements of the form $(\gamma-1)$ for $\gamma\in G$.
For example, the power series
\[
\log(\gamma) = \sum_{n=1}^{\infty} \frac{(-1)^{n+1}}{n}(\gamma-1)^n
\]
converges in $\widehat{\Q[G]}$. This ring is naturally a Hopf algebra, with product $(\gamma,\gamma') \mapsto \gamma \gamma'$ and coproduct $\Delta(\gamma) = \gamma\otimes \gamma$ for all group elements $\gamma,\gamma'\in G$. Recall that primitives
in a Hopf algebra form a Lie algebra.

\begin{definition}
    The Malcev Lie algebra $\Mal{G}$ of the group $G$ is the Lie algebra of primitive elements in $ \widehat{\Q[G]}$.

\end{definition}

Quillen showed in \cite{Quillen-rational} that this is a complete filtered Lie algebra, topologically generated by elements of the form $\log(\gamma)$ for $\gamma\in G$.
Moreover, he showed that the logarithm function 
preserves filtrations when $G$ and $\Mal{G}$ are filtered by their respective lower central series, inducing an isomorphism
of associated graded vector spaces after tensoring with $\Q$.  For a  thorough treatment see Appendix A in \cite{Quillen-rational}, and Section~2 of \cite{Massuyeau-TreeLevelLMO}.

 Recall 
the notion of a Lie pairing from Definition~\ref{LiePairing}.  The
following is a restatement of the Lie
pairing portion of Theorem~\ref{thm-intro:universal property}.

\begin{theorem} \label{thm:pairing}

    Hopf invariants extend uniquely to a  Lie pairing, with
        $$\bigl\langle T, \log(w)\bigr\rangle_H = \int_{S^1}w^*(T),$$ for every $T \in H^0_\mE(X)$ and group element $w\in \pi_1(X)$.
\end{theorem}

Uniqueness of such a pairing will follow from the stated bracket-cobracket duality,
which forces the pairing to be continuous with respect to the filtration by commutator depth, while 
 $\Mal[\big]{\pi_1(X)}$ is topologically generated by  the $\log(\gamma)$ elements.  We call
the resulting pairing the {\bf Hopf pairing}.

 Recall that the Baker-Campbell-Hausdorff series is a formal power series in a completed Lie algebra satisfying
    \[
    \exp(\ell)\exp(\ell') = \exp\bigl(\BCH(\ell,\ell')\bigr).
    \]

    
\begin{corollary}[Aydin's Formula]\label{ozbek}
    Hopf invariants of products in $\pi_1(X)$ are given by the Baker-Campbell-Hausdorff formula: 
   \begin{align*} 
    \int_{S^1}(ab)^*(T) &= \bigl\langle T,\ \BCH\bigl(\log(a), \log(b)\bigr) \bigr\rangle_H \\
    &= \Bigl\langle T,\ \log(a) + \log(b) + \frac{1}{2}\bigl[\log(a),\log(b)\bigr] + \frac{1}{12}(\cdots) + (\cdots) \Bigr\rangle_H.
   \end{align*} 
\end{corollary}
\begin{proof}
   Apply the defining property of the BCH series to
    $\bigl\langle T, \log(ab)\bigr\rangle_H = \int_{S^1} (ab)^*(T)
    $
    taking $a = \exp(\log(a))$ and $b = \exp(\log(b))$.
\end{proof}
Note the natural appearance of fractions in the Hopf invariant of a product, and most notably, Bernoulli numbers. This aspect of the theory was rather unexpected and poorly understood when first encountered by the authors, misinformed by the strictly integral weight-reduction procedure in the context of letter linking in \cite{monroe-sinha}.

\begin{remark}\label{C-infty}
    The second author originally found a proof of Corollary~\ref{ozbek} with finite generation hypotheses using the theory of $C_\infty$-algebras.  
    Briefly, a $C_\infty$-algebra is a graded vector space $C$ with a coderivation on $\cofreeE(C)$  which squares to zero, defining a Harrison complex.  Morphisms
    are just maps on Harrison complexes.

    In \cite{LieEncyclopedia} the authors show that the simplicial cochains on a simplicial set have a $C_\infty$ structure extending cup product 
    which is quasi-isomorphic
    to the PL-de\,Rham cochains on that simplicial set.  In the case
    of the one-simplex this is dual to the Lawrence-Sullivan interval. These models prove Corollary~\ref{ozbek}.

    For an alternative approach, the ``universal product'' on fundamental groups is given by precomposition with a collapse map $S^1 \to S^1 \vee S^1$.
    The second author made explicit calculations in the Harrison model to prove that weight reduction in the Harrison complex
    leads to Baker-Campbell-Hausdorff product.  He then used this to achieve our first proof of the Lifting Criterion below, in the finitely generated setting.
\end{remark}

\begin{proof}[Proof of Theorem \ref{thm:pairing}]
    Let $A^*(X)$ be  de\,Rham forms (smooth or piecewise linear)  for $X$ and consider a loop $\gamma\maps S^1\to X$. We consider the interplay between the coassociative and coLie (Harrison) bar constructions
    \[
    \H^0\bigl(\B(A^*(X)\bigr) \onto \H^0\bigl(\mE(A^*(X)\bigr).
    \]
    Chen defined an invariant of loops using iterated integrals associated to the bar construction.  If $\gamma$ is a loop, $\omega_i\in A^1(X)$ are a $1$-forms and $\gamma^*(\omega_i) = f_i(\theta)d\theta$ then
    \[
    \bigl\langle \omega_1|\ldots|\omega_n , \gamma\bigr\rangle_C \ = \  \int_{t_1<t_2<\ldots<t_n} \hspace{-1.5cm} f_1(t_1)\cdots f_n(t_n)\  dt_1\ldots dt_n.
    \]
    Chen's construction gives a linear map $\langle -,\gamma\rangle_{C} \maps \H^0(\B (A^*(X))) \longrightarrow \Q$ 
    which extends to a bilinear pairing with the group ring $\Q[\pi_1]$. Letting $\bar\Delta$ denote the reduced coproduct of $\H^0(\B (A^*(X)))$, Chen proves in \cite[Eq. (1.6.1)]{chen1971pi1} that
    \[
    \bigl\langle \bar\Delta(T) ,\; (\gamma-1)\otimes (\gamma'-1) \bigr\rangle_C = \bigl\langle T,\; (\gamma-1)(\gamma'-1) \bigr\rangle_C
    \]
    from which it follows that weight $p$ bar cocycles vanish on the 
    $(p+1)$-st power of the augmentation ideal in $\Q[\pi_1]$.  Thus the pairing extends to the completion $\widehat{\Q[\pi_1]}$. 
    We define a  pairing on $\H^0(\mE(A^*(X))$ by restriction to primitive elements.  If $\ell\in \widehat{\Q[\pi_1]}$ is primitive, set
    \[
    \langle \overline{T}, \ell \rangle_{\hat{H}} = \langle T, \ell\rangle_C
    \]
    where $T$ is any lift of $\overline{T}\in\H^0\bigl(\mE(A^*(X)\bigr)$ to the bar complex, say the one given by 
    Barr's splitting \cite{Barr}.
    To see that this is independent of the choice of lift $T$, recall that under Chen's pairing the shuffle product on the Bar complex is dual to the group ring coproduct \cite[Eq. (1.5.1)]{chen1971pi1} from which it follows that nontrivial shuffles must vanish on primitive elements. 
    Since the Harrison complex is the quotient of the coassociative bar complex by nontrivial shuffles, the pairing on primitives descends to Harrison homology.

    Bracket-cobracket duality of this pairing follows quickly from Chen's product-coproduct duality,
    \[
    \bigl\langle\, \cobr{\overline{T}} , \ell\otimes \ell' \bigr\rangle_{\hat{H}} = \bigl\langle\, \cobr{T} , \ell\otimes \ell' \bigr\rangle_C = \bigl\langle \Delta(T), \ell\otimes \ell' - \ell'\otimes \ell \bigr\rangle_C = \bigl\langle T, \ell\cdot \ell' - \ell'\cdot \ell\bigr\rangle_C = \bigl\langle \overline{T}, [\ell,\ell'] \bigr\rangle_{\hat{H}}.
    \]
    Vanishing of weight $p$ cocycles on $(p+1)$-fold nested commutators follows immediately from this duality.

    It remains to show that the pairing $\langle\, , \rangle_{\hat{H}}$ agrees with the Hopf pairing using weight reduction on the Harrison complex of the circle. By work of Gugenheim \cite{gugenheim}, Chen's iterated integrals extend to natural isomorphisms between the top two rows of the commutative diagram below.
    \[
    \xymatrix{
    \H^0(\B\bigl(C^*(X))\bigr) \ar[r]^{\gamma^*}
    & \H^0(\B\bigl(C^*(S^1)\bigr) \ar@{=}[r]  & \Q[e] \ar@/^{25pt}/[dd]^{\log} \ar@{->>}[r] &  \Q e \\ 
     \H^0\bigl(\B(A^*(X))\bigr) \ar[r]^{\gamma^*} \ar@{->>}[d] \ar[u]^{\cong} 
     \ar@{.>}@(u,d)[rrru]|(.445){\phantom{XX}}|(.680){\phantom{XX}}|(.755)\hole
    & \H^0(\B\bigl(A^*(S^1)\bigr) \ar@{=}[r]\ar@{->>}[d] \ar[u]^{\cong} & \Q[d\theta] \ar[u]^{\cong} \ar@{->>}[d] & \\
     \H^0\bigl(\mE(A^*(X))\bigr) \ar[r]^{\gamma^*} & \H^0\bigl(\cofreeE(A^*(S^1)\bigr) \ar@{=}[r] & \Q d\theta &
    }
    \]
    
   \vspace{-6.4em}\hspace{1.8cm}
   \(\xymatrix{
     T \vphantom{\H^0\bigl(} \ar@{|->}[d] \\
    \overline{T} 
   }\)
   

   \smallskip

   \noindent
     Chen's invariant is the dashed arrow. Here, the surjection between the bottom two rows is Harrison's quotient
     by shuffles.

    Letting $I_{\gamma}\maps \H^0(\B(A^*(X))) \to \Q $ denote the coefficient of $e$ in the dashed arrow, observe that it is the corestriction onto the cogenerator space of the cofree coalgebra $\Q[e]$.     Since Gugenheim's maps between the top two rows are coalgebra homomorphisms, the image of $T$ in $\Q[e]$ is
    \[
    \sum_{n=1}^{\infty} I_{\gamma}^{\otimes n}\bigl( \overline\Delta^{n-1}(T)\bigr) \cdot e^n = 
    \sum_{n=1}^{\infty} \,\bigl\langle T,\, (\gamma-1)^n\bigr\rangle_C \cdot e^n.
    \]
    
    On the other hand, the map denoted by $\log$ is the corestriction of the inverse to Gugenheim's isomorphism, which is computed in \cite[Proposition 1.7]{model-of-interval} (where both $e$ and $d\theta$ are denoted by $dt$) to be
    \[
    e^n \longmapsto \frac{(-1)^{n+1}}{n} \cdot d\theta.
    \]
    With this, the image of $T$ in the bottom-right corner is seen to be
    \[
    \sum_{n=1}^{\infty} \bigl\langle T, (\gamma-1)^n\bigr\rangle_C \cdot \frac{(-1)^{n+1}}{n} \cdot d\theta \ = \ \bigl\langle T, \log(\gamma) \bigr\rangle_C \cdot d\theta.
    \]
    This shows that
    \[
    \int_{S^1}\gamma^*(\overline{T}) \ = \ \bigl\langle T, \log(\gamma)\bigr\rangle_C \ = \  \bigl\langle \overline{T}, \log(\gamma)\bigr\rangle_{\hat{H}}
    \]
    for all $\overline{T}\in \H^0\bigl(\mE(A^*(X)\bigr)$. 
\end{proof}

\subsection{The Lifting Criterion}\label{sec:lifting}

We relate the Harrison cohomology of a two dimensional CW complex to that of its one skeleton.  This result
is the most important technical step
in the paper, essential to prove universality of Hopf invariants and to  extend the theory of letter linking invariants 
to general groups.

Suppose $X$ is a two dimensional CW complex with one skeleton
$X^{(1)} \simeq \vee_{S} S^1$ and two cells attached via maps
$\{r\maps \partial D^2 \to X^{(1)}\}_{r\in R}$.
The fundamental group of $X$
is thus the free group $F_S$ modulo the images in $\pi_1$
of the maps $r$, which by abuse we also call $r$.
To fix models, 
let $A^*X$ be the commutative differential graded algebra of PL differential forms.

As discussed in Example~\ref{WedgeCircles} the zeroth Harrison cochomology of a wedge of circles 
is just the cofree Lie coalgebra on the linear dual ${F_S}^\#$.  Our goal is to understand the map induced
by inclusion of the one skeleton $i^* : H^0_\mE (X) \to H^0_\mE(\vee_S S^1) \cong \cofreeE({F_S}^\#)$,
which by the following also yields a calculation of the
zeroth Harrison cohomology of $X$.

\begin{proposition}\label{H0injective}
    The map $i^*$ is injective.
\end{proposition}

\begin{proof}

As shown in Remark~\ref{HarrisonSS}, there are
no differentials into total degree zero for the Harrison
spectral sequence of any connected space from the $E_1$ page onwards.  So 
zeroth Harrison cohomology includes into the $E_1$-page. 
Since $i^*$ is injective on first cohomology, it is also
injective in total degree zero on $E_1$ (note that total degree 0 involves only $1$-forms).  
But the spectral sequence collapses for a wedge of circles.  So $i^*$ is a composite of injective maps.
\end{proof}

\begin{definition}\label{def:subcoalg vanishing on relations}
    Let $E_R$ denote the maximal Lie sub-coalgebra of
    $\cofreeE({F_S}^\#) \cong H^0_\mE(\vee_S S^1)$ whose Hopf invariants vanish on the
    relations in $R$.  
\end{definition}


\begin{theorem}[Lifting Criterion]\label{LiftingCriterion}
    The image of $i^*$ is $E_R$.
\end{theorem}

Informally, a Harrison cocycle on the one skeleton lifts to the full space if and only if it and all of its cobrackets vanish on all relations $r\in R$.
That the image of $i^*$ must vanish on relations is immediate, as by naturality the Hopf pairing can be computed on $X$ where the lift of any relation $r$ is null-homotopic. 
We will see examples below for which the closure under cobracket, whose necessity is also immediate, excludes some classes which vanish on all relations.

We analyze $i^*$ at the level of Harrison complexes, 
letting $Z\subseteq \liebar{A^*X}{0}$ denote the collection of $\dHar$-closed elements.  
Noting that everything in 
$\liebar{A^* (\vee S^1)}{0}$ is a cocycle and thus defines Hopf invariants, 
let $ \mE_R$ be its maximal sub-coalgebra whose Hopf invariants
all vanish on the attaching maps.  
Thus restriction to the one skeleton gives a commutative diagram
\begin{equation}\label{LiftingSquare}
\begin{split}
       \xymatrix{  Z\, \ar@{^(->}[r]\ar[d]_(.4){\operatorname{res}} & \liebar{A^*X}{0} \ar@{->>}[d]^(.4){i^*} \\ \mE_R\, \ar@{^(->}[r] & \liebar{A^*(\vee S^1)}{0}}
\end{split} 
\end{equation}
The Lifting Criterion of Theorem~\ref{LiftingCriterion}  follows immediately from the following.
    \begin{lemma}[Lifting Lemma]\label{lem:lifting lemma}
     There exists a section $s\maps \mE_R \to Z$ which is a Lie coalgebra homomorphism. In particular, the restriction map $Z\to \mE_R$ is surjective.
\end{lemma}

The proof of this claim is the most technical aspect of our work.
The key to our argument is to relate
Hopf invariants of attaching maps to the Harrison complex of the attached disks. 

Recall that Harrison cochains of weight one in $\liebar{A^*X}{k}$ are simply elements in $A^{k+1}(X)$. In the text below we will abusively write
$\alpha$ for both a weight one element of $\liebar{A^*X}{k}$ as well as the
underlying form in $A^{k+1}X$.


\begin{lemma}[Circles and Disks] \label{circles-and-disks}
 Consider $b:S^1 \into D^2$, the inclusion map of the boundary. 
 Every Harrison cocycle $\alpha\in \mE^0{(A^* S^1)}$ is realized as a pullback from the disc $\alpha = b^*(\bar{\alpha})$ for some
 cochain $\bar{\alpha}\in \liebar{A^* D^2}{0}$ such that:
 \begin{enumerate}
     \item $\dHar\bar{\alpha}$ has weight $1$; \label{one}
     \item The $2$-form $\dHar\bar{\alpha}$ is supported on the interior of $D^2$;
     \item $\int_{D^2} \dHar\bar{\alpha} = \int_{S^1} \alpha$.\label{three}
  \end{enumerate}

\end{lemma}

 The integral on the left of (\ref{three}) is the ordinary integration of a $2$-form on the disc, while the integral on the right is the evaluation defined in Proposition~\ref{Ground}.  An illustration of (\ref{three}) is discussed in Section~\ref{commutatormassey}.

\begin{proof}

The map $b^* :A^*D^2 \to A^*S^1$ is a surjection, so the 
induced map on Harrison complexes 
$\mathcal{E}A^*b:\liebar{A^*D^2}{*}\to\liebar{A^*S^1}{*}$ 
is as well.
Define $\liebar{D^2,S^1}{*}$
to be the kernel of $\mathcal{E}A^*b$ and write
$\HHarr{*}{D^2,S^1}$ for the cohomology of the resulting complex. 
The long exact sequence of relative Harrison cohomology reads
\[
 \cdots\to\HHarr{*}{D^2}\to\HHarr{*}{S^1}
  \overset{\delta}{\longrightarrow}
  \HHarr{*+1}{D^2,S^1} \to \HHarr{*+1}{D^2}\to\cdots.
\]
Homotopy invariance implies $\HHarr{*}{D^2}=0$,
so the connecting homomorphism $\delta$ is an isomorphism.

For a cocycle $\alpha\in\liebar{S^1}{0}$, 
the connecting homomorphism is defined by lifting and then applying $\dHar$, namely $\delta\alpha = \dHar\hat\alpha \in \liebar{D^2,S^1}{1}$.
This respects filtration by weight, since both lifting and $\dHar$ do. 
  Since $\HHarr{0}{S^1}$ is supported in (spanned by) weight one, so is $\HHarr{1}{D^2,S^1}$.
 We can thus choose an $\eta\in \liebar{D^2,S^1}{0}$ such that $\dHar \hat{\alpha} + \dHar\eta = d(\hat\alpha + \eta)$ has weight one.  Let $\bar\alpha = \hat\alpha + \eta$.  
Since $\eta|_{S^1}=0$ (every
summand of $\eta$ contains at least one form that vanishes on $\partial D^2$), the cochain $\bar\alpha$ itself restricts to $\alpha$. 
Thus $\bar \alpha$ has coboundary in weight one, proving statement~(\ref{one}).

To prove statement~(\ref{three}), we first apply Proposition~\ref{Ground} to obtain an  $\alpha_0 = \alpha + \dHar\xi$ of weight one.
Extending $\alpha_0\in A^1(S^1)$ to the interior of the disc results in a form $\bar{\alpha}_0\in A^1( D^2)$ such that
\[
\int_{ D^2} d\bar{\alpha}_0 \ = \ \int_{S^1} \alpha_0 \ = \ \int_{S^1}\alpha,
\]
where the first equality is by Stokes' Theorem.
So it suffices to show that $\dHar\bar{\alpha} = d\bar{\alpha}_0 + d\zeta_0$ for some relative $1$-form $\zeta_0\in A^1(D^2,S^1)$.

Observe that both $\dHar\bar\alpha$ and $d\bar{\alpha}_0$ represent the element $\delta\alpha\in \liebar{D^2,S^1}{1}$ 
and they therefore differ by some $\dHar \zeta$ for $\zeta\in \liebar{D^2,S^1}{0}$, possibly of  weight larger than one. If we can show that $\zeta = \zeta_0 + \dHar \mu$ where $\mu\in \liebar{D^2,S^1}{-1}$ and $\zeta_0$ has weight one, then
\[
\dHar\bar\alpha = d\bar\alpha_0 + \dHar(\zeta_0+\dHar\mu) = d\bar\alpha_0 + d\zeta_0
\]
as desired. 

We iteratively show that $\zeta$ is cohomologous to something of lower weight.
Suppose $\zeta$ has weight $p>1$ and let $\zeta_p$ be its summand of weight $p$. The vertical component of $\dHar$ must vanish on $\zeta_p$ since this is the weight $p$ term of $\dHar \zeta$ which is only supported in weight one. 
So $\zeta_p$ is a $0$-cocycle in the complex $\liebar{D^2,S^1}{*}$ equipped with only the internal (vertical) differential.  But both $\liebar{D^2}{*}$
and $\liebar{S^1}{ }$ have no vertical cohomology in degree zero in this weight by the K\"unneth 
Theorem, so neither does  $\liebar{D^2,S^1}{ }$.
Thus $\zeta_p$ is the image under the vertical differential of some  $\mu_p\in \liebar{D^2,S^1}{-1}$. It follows that $\zeta' = \zeta - \dHar \mu_p$ has weight $p-1$ and still cobounds the difference of $\dHar \bar\alpha- d\bar\alpha_0$. Performing this weight reduction iteratively lets us subtract coboundaries from
$\zeta$ until the result $\zeta_0$ has weight one, completing the proof.
\end{proof}

With these facts in place we can prove the Lifting Lemma.
\begin{proof}[Proof of Lemma \ref{lem:lifting lemma}]

Extend the diagram from Equation~(\ref{LiftingSquare}) to the right using maps \(\pi\) projecting
onto cogenerators: 
\[
\xymatrix{
 Z \ar@{^(->}[r] \ar@{->}[d]^(.45){\mathrm{res}} 
  & \mE{A^*X} \ar@{->>}[r]^{\pi} \ar@{->>}[d]_(.45){\mE i^*} 
  & A^*X \ar@{->>}[d]_(.45){i^*}
  & X \\
 \mE_R \ar@{^(->}[r]              \ar@{.>}[urr]|(.44)\hole_(.65){\hat s} 
   \ar@{.>}[ur]^(.45){\bar s} \ar@{.>}@/^1pc/[u]^(.55){s}
  & \mE{A^*(\vee  S^1)} \ar@{->>}[r]^{\pi} 
  & A^*(\vee  S^1) \ar@/_1pc/@{-->}[u]_{\sigma}
  & X^{(1)} \simeq \vee S^1\hspace{-1cm} \ar@{^(->}[u]_{i}
}
\]
We define the desired section by climbing the weight filtration $\{\mE_R^{\le n}\}$ inherited 
from $\liebar{A^*(\vee  S^1)}{0}$.
To start, choose splittings $\mE_R^{\le n} =  \mE_R^{\le (n-1)} \oplus \mE_R^n$ so that $\mE_R = \bigoplus \mE_R^n$. Furthermore, pick an arbitrary section $\sigma:A^*(\vee  S^1)\to A^*X$ by extending the domain of
definition of differential forms to the interiors of the $2$-cells in $X$.
By construction, $\sigma$ is a section of the restriction map 
$i^*:A^*X \to A^*(\vee  S^1)$.

We build  $s:\mE_R \to Z$ by constructing a coalgebra homomorphism $\bar s: \mE_R \to \mE{A^*X}$ landing in $\dHar$-closed elements. 
Since $\mE{A^*X}$ is cofreely cogenerated by $A^*X$, such a map must be the cofree coextension of some $\hat s: \mE_R \to A^*X$.  We use the weight filtration on $\mE_R$
to define $\hat s$.
Begin with 
\[ 
 \hat s_{0}:\mE_R \xrightarrow{\ } \cofreeE{A^* (\vee  S^1)} 
         \xrightarrow{\;\pi\;} A^*(\vee  S^1)
         \xrightarrow{\;\sigma\;} A^*X,
\]
and let $\bar s_{0}$ be the cofree extension to a homomorphism into $\mE{A^*X}$. The induced homomorphism $\bar s_{0}$ is indeed a section of the restriction $\cofreeE i^*$, however it does not necessarily send $\mE_R$ to $\dHar$-closed elements.
We perform a modification  at 
each step to get maps $\hat s_{n}$ whose cofree extensions to homomorphisms
$\bar s_{n}$ map 
$\mE_R^{\le n}$ to $\dHar$-closed elements.

Consider this modification in detail from $\hat s_{0}$ to $\hat s_{1}$.
Given $v\in \mE_R^1$, we calculate 
\[
\cobr[\Big]{\dHar \bar{s}_{0}(v) }\; = 
 (\dHar\otimes I + I\otimes \dHar)\circ (\bar{s}_{0}\otimes \bar{s}_{0})\ \cobr{v}\;=0,
\]
since in $\mE_R^1$ we have $\cobr{v}=0$.
Because cobrackets are injective in weight $\geq 2$ by Proposition~\ref{cofree conilpotent} 
 we deduce that $\dHar \bar{s}_{0}(v)$ has weight one. 
By assumption, $v\in \mE_R$ has vanishing Hopf invariants on all $r\in R$ so by the {Circles and Disks} 
Lemma \ref{circles-and-disks},
\[
\int_{D^2_r} \dHar \bar{s}_{0}(v) = \int_{S^1} r^*v = 0,
\]
where $D^2_r$ is the two-disk attached by $r$.  
By the relative cohomology long exact sequence for $X$ and its one-skeleton, any cohomology class which
evaluates to zero on all $D^2_r$ must cobound.
Thus the $2$-form $\dHar \bar{s}_{0}(v) \in A^2(X)$ has a cobounding $\omega_v\in A^1(X,\vee S^1)$ such that 
$d\omega_v = \dHar \bar{s}_{0}(v)$. 

Pick a basis $\{ v_i\}_{i\in I}\subset \mE_R^{1}$ and define $\hat s_{1}\maps \mE_R\to A^1X$ by linearly extending
\[
\hat s_{1}(v_i) = \hat s_{0}(v_i) - \omega_{v_i}  
\]
on $\mE_R^1$ and $\hat s_{1} = \hat s_{0}$ one higher weights.   
Since the forms $\omega_{v_i}$ restrict to zero on $\vee S^1$, we have not changed 
images after restriction
$i^*\, \hat{s}_1 = i^*\, \hat{s}_{0}$,
and $\bar{s}_{1}(v)$ is now $\dHar$-closed by construction.

The general case is almost identical. Given $\hat{s}_{n}$ extending to $\bar{s}_{n}$ with the property that
$\bar{s}_{n}$ sends $\mE_R^{\le n}$ to $\dHar$-closed elements, we first note that every $v\in \mE_R^{n+1}$ must have $\dHar \bar{s}_{\leq n}(v)$ of weight one. Indeed, this follows
again by applying injectivity of cobracket, but now using the fact that 
tensor factors of $\cobr{v}$ have weight $\le n$ but $\dHar \bar{s}_{n}$ of such terms is zero.  That is,
\[
\cobr[\Big]{\dHar \bar{s}_{\leq n}(v) }\; = 
 (\dHar\otimes I + I\otimes \dHar)\circ (\bar{s}_{\leq n}\otimes \bar{s}_{\leq n})\ \cobr{v}\; = 0.
\]
Thus we can apply the {Circles and Disks} Lemma as before to find coboundings of $\dHar \bar{s}_{n}(v)$ 
and use these to define $\hat{s}_{n+1}$ such that $i^*\bar{s}_{n+1} = i^*\bar{s}_{n}$ and $\bar{s}_{n+1}$ of 
$\mE_R^{\le (n+1)}$ is now $\dHar$-closed.
\end{proof}

\newcommand{\varS}{{\scriptstyle \Sigma}}
\newcommand{\varD}{{\scriptstyle \Delta}}

\subsection{Letter braiding algorithms}\label{Algorithms}

There is  a simple, quick algorithm which can be used to evaluate the
letter braiding functions from Section~\ref{combalg} on words.  
Computation is made via a single pass across the word.  
To evaluate we store homogeneous symbols as tree data structures
where each vertex $v$ records a d-function $h_v$, 
along with two counters $\varS_v$ and $\varD_v$ which are 
initially set to 0 and updated as we move across the word.
Evaluation is computed by iterating over the ordered list of letters in the word, and for each letter we iterate depth-first (away from the root)
in the tree symbol.
For each letter $\ell_i\in S^{\pm 1}$ and vertex $v$, the $\varS_v$ and $\varD_v$ 
values are updated as follows.
\begin{enumerate}
\item
First increment $\varS_v$ by
the current $\varD_v$ value, and reset $\varD_v$ to 0.
\item
Next, compute the sum of all children's $\varS$-values multiplied by the value $h_v(\ell_i)$ at the given letter.  \\
If the node $v$ has no children then use the value $h_v(\ell_i)$ itself.
\item
If letter $\ell_i$ is an inverse
then add the number computed in the previous step to $\varS_v$. Otherwise, add it to $\varD_v$.
\end{enumerate}
After evaluating over all vertices at all letters, the value of the braiding symbol on 
the given word is the sum $\varS_r+\varD_r$ at the root vertex of the tree.

To see that this algorithm indeed computes the braiding function observe that during the course of computation, at letter $\ell_i$ of the word and vertex $v$ of 
the braiding symbol,  $\varS_v = \lpw {\Phi_v} (i)$ where $\Phi_v$ is the 
subsymbol supported by vertex $v$.  Thus the sequence of values 
computed in $\varS_v$ 
are the d-function $\lpw {\Phi_v}$.  See \cite{Walter-Github} for an implementation
of this algorithm in Python, using multiplication with EilWord and EilTree objects.

\begin{example}
Recall Example~\ref{BraidingProductExample} 
computing $\lpw {B-2A} C$ evaluated on 
$w=bca b^\inv c^\inv bb$.  The algorithm described above 
would run as follows (computed top-to-bottom then left-to-right).

\begin{center}
\begin{tabular}{cc||c|c|c|c|c|c|c}
   &  & $\ \ b\ \ $ & $\ \ c\ \ $ & $\ \ a\ \ $ & $\ \ b^\inv\ $ & $\ \,c^\inv\ $ & $\ \ b\ \ $ & $\ \ b\ \ $ \\ \hline\hline
    & &&&&&&& \\[-2ex]
\multirow{2}{*}{$B-2A$} &
$\varD_0$ & $1$ & $0$ & $-2$ & $0$ & $0$ & $1$ & $1$ \\ 
 & $\varS_0$   & $0$ & $1$ & $1$ & $-2$ & $-2$ & $-2$ & $-1$ \\ \hline\hline 
   &  &&&&&&& \\[-2ex]
\multirow{2}{*}{$C$} &
$\varD_1$ & $0$ & $1$ & $0$ & $0$ & $0$ & $0$ & $0$ \\ 
 & $\varS_1$   & $0$ & $0$ & $1$ & $1$ & $3$ & $3$ & $3$ \\  
\end{tabular}
\end{center}

The final value $\varS_1+\varD_1=3$ gives the result of the braiding product evaluated
on $w$.
Note that the values of $\varS_0$ track the values of the d-word $\lpw{B-2A}$
calculated in Example~\ref{BraidingProductExample}.  The variable $\varD_v$
is used to temporarily store values which get added to $\varS_v$ ``between letters''
due to evaluating at a generator.  We can also see that
$\bigl(B-2A\bigr)(w) = 0$  
because $\varS_0+\varD_0 = 0$ after the
last step of the algorithm.  
\end{example}

\medskip

It is substantially more involved to generate letter braiding functions which are well-defined on a group
$G =  \langle S\,|\,R\rangle$.  Applying the Lifting Criterion from Section~\ref{sec:lifting} we know these
are invariants on $F_S$ 
  whose iterated cobrackets all vanish on the relations in $R$.
While the vanishing of iterated cobrackets on relations seems like a large set of conditions to check,
it leads to fairly simple recursive algorithms for generating letter
linking invariants up to a fixed weight.  

Let us describe this procedure. Start in weight one, where we recall that the indicator homomorphisms $\{S_i\}_{i=1}^m$ are a basis for 
$\Hom(F_S, \Q)$.  These evaluate on relations $r\in R$
by counting (signed) occurrences of letters.  Thus vanishing on a given 
relation $r$ imposes a linear condition on the coefficients of invariants $\sum \lambda_i S_i$. 
Calculating the kernel of  all such conditions yields the subspace
 $E_R^1 = \Hom(G,\Q) \subset \Hom(F_S, \Q)$ 
of weight one letter linking invariants for $G$.  

In higher weights, the vanishing of cobrackets on relations comes into play.
Suppose we have already calculated $E_R^{\le k} \subset \cofreeE^{\le k} ({F_S}^\#)$ -- the vector space of 
weight $\le k$
letter linking invariants of $F_S$  
which lift to $G$ in the sense of Theorem~\ref{LiftingCriterion}.
We may then compute the 
vector subspace $V_{k+1}\subset \cofreeE^{\le k+1} ({F_S}^\#)$ given by the preimage of 
$E_R^{\le k}\otimes E_R^{\le k}$ under cobracket (implicitly identifying $\cofreeE^{\le k} \subset \cofreeE^{\le k+1}$).
The next set of invariants, $E_R^{\le k+1}$, is the vector subspace of 
$V_{k+1}$
which vanishes on all relations in $R$. To find it, one may pick a basis $\{ T_\alpha \}_{\alpha\in A} \subset V_{k+1}$, evaluate its elements on all relations $r\in R$ using the evaluation algorithm above, and obtain linear constraints on the coefficients of invariants $\sum_\alpha \lambda_\alpha T_\alpha$.
Note that $V_{k+1}$ contains $V_k$ (as well as all of $\cofreeE^0$),
so $E_R^{\le k} \subset E_R^{\le k+1}$.

\smallskip

The second author gave an alternative implementation \cite{Aydin-Github} in SageMath, using BCH expansions and his Formula (Corollary~\ref{ozbek})
to check vanishing on relations. 
Computations are done in the dual of the free Lie algebra $(\mathbb{L} S)^*$ 
using the SageMath implementation of Hall
basis and BCH expansions for $\mathbb{L}S$.  

To elaborate, given a word $r$ in the set of generators $S$, let $BCH(r)$ be the BCH expansion
of $r$, extending the BCH product:
\[ \mathrm{BCH}(a_1a_2\dots a_k)\ =\ 
   \log\bigl(e^{a_1}e^{a_2}\cdots e^{a_k}\bigr) 
   \in \mathbb{L}S.\]
Write $\mathrm{BCH}_k(r)$ for the truncation of $\mathrm{BCH}(r)$ to include only 
weight $\le k$ terms, thus projecting to the free $k$-step nilpotent Lie algebra,
$\mathbb{L}^{\le k}S$.  
If $b$ is a Hall basis element, then pairing the dual element
$\bigl\langle b^*, \mathrm{BCH}(r)\bigr\rangle$ yields the coefficient of $b$ in the BCH 
expansion of $r$.
Written in terms of a Hall basis, $\mathrm{BCH}_k(r)$ is the matrix of 
coefficients of the BCH expansion through weight $k$.

The implementation \cite{Aydin-Github} proceeds as follows.
\begin{enumerate}
\item 
Make the vector space $(\mathbb{L}^{\le 1}S)^* \cong \Q^S$.

\item
Construct a linear transformation 
$\mathrm{BCH}_1:(\mathbb{L}^{\le 1}S)^* \to \mathbb{Q}^n$
evaluating on relations $\{r_1,\dots,r_n\}$.

\item
Compute the kernel of $\mathrm{BCH}_1$. 
This is $E_R^1$, the weight one letter linking invariants.

\item
Loop over the following steps.

\begin{enumerate}
\item
 Expand the previously constructed vector space 
 $(\mathbb{L}^{\le k}S)^*$ to $(\mathbb{L}^{\le k+1}S)^*$
 by including the (dual) Hall basis elements of weight $k+1$.

\item
 Make tensor product vector space $(\mathbb{L}^{\le k}S)^* \otimes (\mathbb{L}^{\le k}S)^*$

\item 
 Realize the cobracket linear transformation 
  $$\cobr{-}:(\mathbb{L}^{\le k+1}S)^* \longrightarrow 
            (\mathbb{L}^{\le k}S)^*\otimes (\mathbb{L}^{\le k}S)^*$$
 as the transpose of the $(k+1)$-step nilpotent bracket transformation
  $$[-,-]:(\mathbb{L}^{\le k}S)\otimes (\mathbb{L}^{\le k}S) 
         \longrightarrow \mathbb{L}^{\le k+1}S$$

\item 
 Having previously caltulated the vector subspace $E_R^{k}\leq (\mathbb{L}^{\le k}S)^*$, calculate the preimage of the vector subspace $E_R^{k}\otimes E_R^{k} \subset
 (\mathbb{L}^{\le k}S)^*\otimes (\mathbb{L}^{\le k}S)^*$ under the cobracket.
 This preimage is the subspace $V_{k+1}$ inside $(\mathbb{L}^{\le k+1}S)^*$.

\item 
 Make a linear transformation 
 $\mathrm{BCH}_{k+1}:(\mathbb{L}^{\le k+1}S)^* \to \mathbb{Q}^n$ by evaluation on $\{r_1,\ldots,r_n\}$.

\item 
 Calculate the kernel of $\mathrm{BCH}_{k+1}$ restricted to $V_{k+1}$.
 This is the subspace $E_R^{k+1}$ we sought.
\end{enumerate}
\end{enumerate}

\medskip

Many optimizations are possible.  
For example, BCH computations could be sped up by using the left-greedy Hall basis
so that Lie coalgebra pairings give rewriting formulas (see \cite{Walter-Shiri}).
Even better would be to remove BCH altogether via the ``one-pass" evaluation procedure described earlier in this subsection. Together with coalgebraic techniques that  remove redundancies, calculations of  letter-linking invariants of a finitely presented group should be relatively efficient, especially when building on
 the quick evaluation procedure which has already been implemented as part of the Lie algebra/coalgebra
toolbox in \cite{Walter-Github}.

\subsection{Proving the Fundamental Theorem}\label{ProofFundamental}

We restate the universality portion of Theorem~\ref{thm-intro:universal property}, which is the only part
we have left to prove.
   
\begin{theorem}\label{thm:universality restated}
     The Hopf  pairing of $\HHarr{0}{A^*(X)}$ with the Malcev Lie algebra of $\pi_1(X)$ is a universal Lie pairing.
    
\end{theorem}

The Lifting Criterion immediately gives a way to deduce this theorem for general groups from the free group case.
We start with the following refinement of the calculation in Example~\ref{WedgeCircles}.
\begin{lemma}
    Let $X = \vee_S S^1$
    with $i_s\maps S^1\into X$ the inclusion map of circle $s\in S$. Under the canonical isomorphism $\HHarr{0}{X}\cong \cofreeE(\Q^S)$ we have
    \[
    \int_{S^1} i^*_s(T) = \begin{cases}
        f(s) & \text{ if } T=f\in \cofreeE(\Q^S) \text{ has weight $1$} \\
        0 & \text{ otherwise}
    \end{cases}
    \]
      for every homogeneous expression $T\in \cofreeE(\Q^S)$.
\end{lemma}
\begin{proof}
    A commutative cochain model for a wedge of circles is given by the square-zero algebra $1\oplus \Q^S$, whose Harrison complex is $\cofreeE(\Q^S)$ concentrated in degree $0$.
The circle inclusion $i_s$ is modeled at the cochain level by the algebra homomorphism
    \[
    A^*(X) \simeq 1\oplus \Q^S \longrightarrow 1\oplus \Q \simeq A^*(S^1)
    \]
    sending a function $f\in \Q^S$ to its evaluation $f(s)$.
    This pullback induces a weight preserving map on Harrison complexes $i_s^*\maps \cofreeE(\Q^S) \to \cofreeE(\Q)$. But since the target is one dimensional and concentrated in weight one, it follows that all positive weight homogeneous elements are sent to $0$.
\end{proof}

\begin{proof}[Proof of Universality (Theorem \ref{thm:universality restated})]
First consider the case $X = \vee_S S^1$.  Then $\pi_1(X)\cong F_S$, the free group on basis $S$. Also its Harrison cohomology is isomorphic to $\cofreeE(\Q^S)$. 

To see that the Hopf pairing is universal, let $\langle -,-\rangle_C \maps C\times \Mal{F_S}\to \Q$ be any Lie pairing. Then the map $\varphi_0\maps C\to \Q^S$ given by
\[
\varphi_0(T)\maps s\longmapsto \bigl\langle T,\;\log(s)\bigr\rangle_C
\]
extends uniquely to a Lie coalgebra homomorphism $\varphi\maps C\to \cofreeE(\Q^S)$ by cofreeness of the target. We show
\[
\bigl\langle \varphi(T),\;\ell \bigr\rangle_H = \bigl\langle T,\;\ell\bigr\rangle_C
\]
for all $T\in C$ and $\ell\in \Mal{F_S}$. First, consider $\ell = \log(s)$ for a generator $s\in S$.  The previous lemma shows that the Hopf pairing vanishes on homogeneous elements of weight $\geq 2$ in $\cofreeE(\Q^S)$, so we have
\begin{align*}
\bigl\langle \varphi(T), \log(s)\bigr\rangle_H 
& = \bigl\langle \varphi_0(T) + (\text{higher weight terms}), \; \log(s)\bigr\rangle_H \\
& = \bigl\langle \varphi_0(T),\ \log(s)\bigr\rangle_H \\
& = \bigl\langle T,\ \log(s)\bigr\rangle_C
\end{align*}
by definition of $\varphi_0$.

Now, suppose by induction that the equality holds for elements $\ell_1, \ell_2\in \Mal{F_S}$.  Then 
\begin{align*}
\bigl\langle \varphi(T),\ [\ell_1,\,\ell_2] \bigr\rangle_H 
& = \bigl\langle \cobr{\varphi(T)}\,,\ \ell_1\otimes \ell_2 \bigr\rangle_H  \\
& = \bigl\langle (\varphi\otimes\varphi)\, \cobr{T}\,,\ \ell_1\otimes \ell_2 \bigr\rangle_H  \\
& = \bigl\langle \cobr{T}\,,\ \ell_1\otimes \ell_2\bigr\rangle_C \\
& = \bigl\langle T,\ [\ell_1,\ell_2] \bigr\rangle_C
\end{align*}
since both pairings are Lie pairings. This shows that equality holds on the Lie algebra generated by $\log(s)$ for $s\in S$. Lastly, suppose $T$ has weight $n$. The Baker--Campbell--Haudorff relation shows that every $\ell\in \Mal{F_S}$ coincides with an element $\ell'$ from the latter Lie subalgebra up to $(n+1)$-fold nested brackets, so the two elements have equal pairings with $T$. This completes the proof in the case that $X = \bigvee_S S^1$.

\smallskip

Now, suppose $X$ is a two-complex with $\pi_1(X)=G$ and one-skeleton 
$X^{(1)} = \bigvee_S S^1$. The inclusion $i\maps X^{(1)}\into X$ induces a presentation $i_*\maps F_S\onto G$ with kernel normally generated by the attaching maps of $2$-cells $r\in R$.

If $\langle -,-\rangle_C\maps C\times \Mal{G} \to \Q$ is a Lie pairing, then it induces a Lie pairing with $F_S$ by precomposition with $\Mal{i_*}\maps \Mal{F_S}\to \Mal{G}$. Thus, by the already proved case of a wedge of circles there exists a unique Lie coalgebra homomorphism
$
\varphi\maps C \to \cofreeE(\Q^S)
$
pulling back the Hopf pairing to $\langle-,-\rangle_C$.

But, since every relation $r\in R$ is killed by the projection $i_*$, the image of $C$ pairs trivially with all relations. That is, the image $\varphi(C)$ is contained in the subcoalgebra $E_R$ from Definition \ref{def:subcoalg vanishing on relations}. 
Since the Lifting Criterion in Theorem \ref{LiftingCriterion} shows that the pullback $i^*\maps \HHarr{0}{A^*(X)} \to \cofreeE(\Q^S)$ is an isomorphism onto $E_R$, and since  $i^*$ is a coalgebra homomorphism that preserves the Hopf pairing, we get a map
\[
C \xrightarrow{\ \ \varphi\ \ } E_R \xrightarrow{(i^*)^{-1}} \HHarr{0}{A^*(X)}
\]
pulling back the Hopf pairing to $\langle-,-\rangle_C$. Uniqueness of this map follows from the fact that after composition with $i^*$ we must get back the uniquely determined map $\varphi$, but $i^*$ is injective.
\end{proof}

In the course of our proof we obtain an independent calculation that the Malcev Lie algebra of a free
group is a free Lie algebra.  Moreover
as a corollary of this, Theorem~\ref{freecase}, and
Quillen's theorem comparing lower central series and 
Malcev Lie algebras, 
we get a new topological proof of the main theorem
of \cite{monroe-sinha}, that letter-linking invariants
obstruct the lower central series of free groups.

We can now establish our main  application in combinatorial group theory, that letter linking invariants
obstruct the rational lower central
series in general.

\begin{proof}[Proof of Theorem \ref{thm:into-combinatorial-alg}]
    The theorem is phrased in terms of letter linking functions, so we first translate the statement into one regarding Hopf invariants. 
    In Section \ref{topcombo} we proved that the combinatorially defined braiding functions agree with Hopf invariants of corresponding Harrison cocycles. Moreover, braiding functions  span all Hopf invariants for the free group, since the Harrison spectral sequence for a handlebody collapses to $\cofreeE({F_S}^\#)$. Finally, by the Lifting Criterion in Theorem \ref{LiftingCriterion}, a Hopf invariant of $F_S$ descends to $G$ if and only if it belongs to the Lie dual of $G$. Thus it suffices to prove that the image of $w\in F_S$ in $G$ has a power that is a product of $k$-fold nested commutators if and only if all weight $\leq k-1$ elements in the Lie dual of $G$ vanish on $w$.

    Let $\gamma_k G$ denote the lower central series of $G$ (setting $\gamma_1 G=G$ and $\gamma_{k+1}=[\gamma_k G,G]$).  
    Quillen proved \cite[Equation (1.5)]{Quillen-graded-grp-rings} that the natural map $\log\maps G\to \Mal{G}$ induces an isomorphism
    \[
    \bigoplus_{k\geq 1} \sfrac{\gamma_k G}{\gamma_{k+1}G} \underset{\mathbb{Z}}{\,\otimes\,} \mathbb{Q} \longrightarrow \bigoplus_{k\geq 1} \sfrac{\gamma_k \Mal{G}}{\gamma_{k+1}\Mal{G}}.
    \]
    It follows that some power of $x\in G$ lies in $\gamma_{k+1}G$ if and only if its logarithm lies in $\gamma_{k+1}\Mal{G}$.

    Thus it is enough to show for $E$, a Lie dual of $G$, that $\gamma_{k+1}\Mal{G}$ is exactly the kernel of the evaluation on the weight $\leq k$ elements of $E$. Such evaluations are trivial because if $T\in E$ has weight $\leq k$ then its $(k+1)$-fold iterated cobracket is zero, which by bracket-cobracket duality implies vanishing on $(k+1)$-fold nested commutators.
    
    For the converse, suppose $\log(w)\notin\gamma_{k+1}\Mal{G}$.  
    Then $\log(w)$ represents a nontrivial element in the finite dimensional Lie algebra $L=\Mal{G}/\gamma_{k+1}\Mal{G}$. 
    The linear dual $L^{\#}$ is a $(k+1)$ co-nilpotent Lie coalgebra 
    which pairs with $\Mal{G}$. Furthermore because it is the linear dual of $L$, it contains a functional $t$ which does not vanish on $\log(w)$. By the universal property of the Lie dual, there exists a coalgebra map $\varphi\maps L^{\#}\to E$ that pulls back the Lie pairing. But then
    \[
    0 \ \neq \ \bigl\langle t,\; \log(w)\bigr\rangle_{L^{\#}}
      \ =\ \bigl\langle \varphi(t),\; \log(w)\bigr\rangle_E
    \]
    exhibiting $T := \varphi(t)$ as an element of weight $\leq k$ on which $\log(w)$ does not vanish.
\end{proof}

The first theoretical algorithm for understanding the rationalized lower central series appeared 
over half a century ago \cite{Chen-Fox-Lyndon}.
A refined inductive approach was implemented
relatively recently \cite{Nickel-Algorithm}.  
One benefit of our approach can  be seen when trying to calculate
a ``multiplication table'' for the commutator operation.  In previous approaches, one would
generate bases for all commutators up to a certain weight, and then would need to use the reduction algorithms repeatedly to calculate commutator brackets.  In our
approach one would just generate Hopf invariants, whose cobracket structure is obtained by
an immediate calculation which can be carried out in the cofree Lie coalgebra.  Bracket-cobracket duality then determines the commutator bracket immediately.

\color{black}     


\section{ Examples, including graphical models}\label{examples}

\subsection{Orientable surfaces}\label{surfaces}

Let us consider the case of a torus, providing a counterpoint to the free group setting.
The standard cohomology model for the torus is a free
differential graded commutative algebra 
 on generators $A$ and $B$ in degree
one, thus with a non-trivial product $AB$ in
degree two and no other classes. 
By Theorem~\ref{HarrisonDerivedIndec},  Harrison cohomology  is
generated by $A$ and $B$ themselves
in weight one.  As Hopf invariants,
these span the dual of the fundamental
group, which is isomorphic to its Malcev Lie algebra via the $\log$ map.

Higher genus surfaces provide an opportunity to see both geometry and algebra.
We use Thom classes of curves, as developed in Appendix~\ref{Thomforms}.
The case of a surface of genus two is drawn
in Figure~\ref{fig:genus 2}, with
``dual'' curves $A$, $B$, $C$ and $D$ 
yielding corresponding
Thom forms which by abuse we give the same name 
generating 
first cohomology.
As pictured, we can choose these Thom forms 
to have  
$A B - C  D = dS$ at the cochain level.

\begin{figure}
\begin{tikzpicture}[scale=.30]

\begin{scope}[ultra thick,decoration={
    markings,
    mark=at position 0.6 with {\arrow[scale=1.0]{ang 90}}}
    ] 
    \draw[postaction={decorate}] (4,0)--(0,4);
    \draw[postaction={decorate}] (0,8)--(0,4);
    \draw[postaction={decorate}] (8,-0)--(4,0);
    \draw[postaction={decorate}] (4,12)--(0,8);
    \draw[postaction={decorate}] (4,12)--(8,12);
    \draw[postaction={decorate}] (8,12)--(12,8);
    \draw[postaction={decorate}] (12,4)--(12,8);
    \draw[postaction={decorate}] (8,0)--(12,4);
\end{scope}

\coordinate (B) at (8.515,6.935);

\draw[name path=node1,ultra thick][yellow] (2.6,1.4) to[out=70,in=-70] (2.6,10.6);
\draw[name path=node2,ultra thick][green] (5.5,0) to[out=90,in=0] (0,6.5);
\draw[name path=node3,ultra thick][blue] (9.4,1.4) to[out=110,in=-110] (9.4,10.6);
\draw[name path=node4,ultra thick][red] (12,5.5) to[out=180,in=-90] (6.5,12);

\node[fill,circle, scale=.64][violet] at (B){};
\draw[fill,name intersections={of= node1 and node2}][violet] (intersection-1) circle[radius=.32];
\draw[out=0, in =170,ultra thick][violet] (intersection-1)  .. controls (4.6,5.2) and (5.1,8.2).. (B);

\node[scale=1.0] (bi) at (6.3,13.2){$b^{-1}$};
\node[scale=1.0] (b) at (13,6){$b$};
\node[scale=1.0] (di) at (6.3,-1){$d^{-1}$};
\node[scale=1.0] (d) at (-1.2,6){$d$};
\node[scale=1.0] (c) at (1,10.41){$c$};
\node[scale=1.0] (ai) at (11.4,10.41){$a^{-1}$};
\node[scale=1.0] (ci) at (0.8,1.41){$c^{-1}$};
\node[scale=1.0] (ai) at (10.8,1.41){$a$};

\node[scale=1.1][green] (D) at (5.7,3.2){$D$};
\node[scale=1.1][blue] (A) at (9.4,4.3){$A$};
\node[scale=1.1][yellow] (C) at (4.0,9){$C$};
\node[scale=1.1][red] (B) at (7.5,10.5){$B$};
\node[scale=1.1][violet] (S) at (5.7,5.9){$S$};
\end{tikzpicture}
\vspace{-1.0cm} 
\caption{Genus two surface with Thom forms}\label{fig:genus 2}
\end{figure}
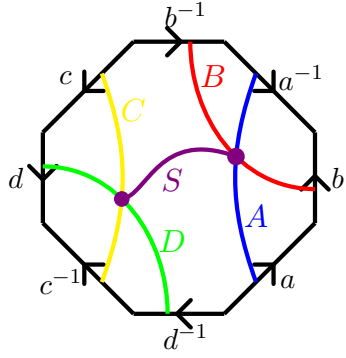

In weight one, the Harrison cocycles are spanned by $A,B, C$ and $D$.  These count intersections with the corresponding curves, 
which in the generators-and-relations presentation of the fundamental group means counting occurrences
of the corresponding letter.
In weight two, $\linep{A}{C}$ and similar expressions involving 
pairs of cochains with zero product (no intersection) at the cochain
level  are all cocycles.
Following the analysis 
of Section~\ref{first}, the  Hopf invariant associated to 
$\linep{A}{C}$ would count (signed) occurrences of a loop
 crossing curve $C$ 
between fixed pairs of crossings with $A$ with opposite orientation, in the eigen setting 
when the total number of such crossings with $A$ vanish.  When the total number of crossings with $A$ does not vanish,
the corresponding invariant can be calculated
either by Aydin's formula using Baker-Campbell-Hausdorff multiplication, as stated in Corollary~\ref{ozbek}, or directly from the definitions as we will do in Section~\ref{beyond}.

Differences of pairs of intersecting curves also yield Harrison cocycles, in particular 
$\omega = \linep{A}{B} - \linep{C}{D} - S$.  In the eigen setting this corresponds to the 
counting invariant given by
the signed number of 
crossings of a loop with $B$ between its crossings with $A$, minus 
($D$-crossings between $C$-crossings), minus (crossings with $S$).  In Figure \ref{fig:genus 2}, a homotopy moving the loop ``from left to right'' across the octagon would result in this crossing count having contribution from the three summands: either intersecting the green and yellow curves, or the purple curve, or the blue and red curves, so that their total sum remains invariant.  

If we consider only loops on the 1-skeleton and simple homotopies 
across the 2-cell, we obtain the standard generators and
relations presentation of this group.
Because $[a,b] [c,d] = 1$ the number of $b$'s between $a-a^\inv$ pairs and that of $d$'s between $c-c^\inv$ 
pairs may change, but their difference will be 
invariant.  This difference, represented by the braiding
symbol $\lp A B - \lp C D$, is the letter-braiding representative 
of the Hopf invariant associated to $\omega$.  
As we effectively are pulling back to the one-skeleton, there is no evidence of
the ``correction curve'' $S$, which topologically is
disjoint from the one-skeleton.  

As an alternative approach to calculation we may compute the Lie dual using a different cochain model.
The genus two surface is formal and has 
cohomology algebra $\mathcal{A}$ equal to the commutative differential graded algebra 
on $A, B, C, D$ in degree one with all products zero except
$AB = CD$.  Thus $\tpc A B - \tpc C D$ is a cocycle in the Harrison complex.  
This example generalizes to present Harrison cocycles of higher weight as symbols
that include $\tpc A B - \tpc C D$ as a subsymbol,
such as 
$$ \tpc[\big] {\tpc D{\color{red} \tpc A B}} B\ {\color{red}-}\ \tpc[\big] {\tpc D{\color{red}\tpc C D}} B$$ 
and iterated versions such as 
$$\begin{alignedat}{4}
 && 
  \tpc[\big] {\tpc D{\color{red}\tpc A B}} \tpc[\big] {{\color{blue}\tpc A B}} B
 \ &{\color{red}-}&\
  \tpc[\big] {\tpc D{\color{red}\tpc C D}} \tpc[\big] {{\color{blue}\tpc A B}} B
 \\ 
 \ &{\color{blue}-}&\  
  \tpc[\big] {\tpc D{\color{red}\tpc A B}} \tpc[\big] {{\color{blue}\tpc C D}} B
 \ &{\color{purple}+}& \
  \tpc[\big] {\tpc D{\color{red}\tpc C D}} \tpc[\big] {{\color{blue}\tpc C D}} B.
 \end{alignedat}$$
 
Generally, let $\tau_k$ be any symbol with $k$ copies of ``isolated" $A$-$B$ or $C$-$D$ subsymbols, meaning
that for the corresponding trees these edges are not contained in weight three subsymbols containing another such edge, such as $\tpc[\big] {\tpc A B} A$ or $\tpc A \tpc A B$, nor ones that exhibit such an edge after a subsymbol exchange such as
$\tpc[\big] {\tpc A B} C$ or $\tpc A \tpc C B$. Then
the alternating sum of all $2^k$ possible subsymbol exchanges, as above, will
be a cocycle. 
Indeed, the Harrison differential is the sum over all single-edge contractions, and every $A$-$B$ edge contraction has a matching $C$-$D$ edge contraction producing the same symbol with opposite sign and vice versa. All other edge contractions vanish due to the triviality of the product.

In fact the cocycles thus presented generate the entire $0$-th Harrison cohomology, 
because the disallowed subsymbols can be rewritten away via the nested Arnold relation
(\ref{eqn:nested Arnold}). 
For example,
\[ 
\tpc[\Big] {\tpc[\big]{ \tpc A B} A}\tau = \tpc[\Big] {\tpc A \tpc B A}\tau -
  \tpc[\Big] {\tpc[\big]{A \tpc B} A}\tau  
\]
and
\[
\tpc[\big] {\tpc A \tpc A B}\tau = \tpc[\Big] {\tpc[\big] {\tpc A A}B} \tau + 
  \tpc[\Big] {\tpc[\big] {A \tpc A }B} \tau.
\]
Thus we have representatives for the Lie dual of the genus two surface group.

Higher genus surfaces behave similarly.
This calculation speaks to the strength of the Eil symbol
approach to Lie coalgebras, since we critically used the Arnold identity
in the rewriting above.
We believe it would be quite difficult to 
find a suitable basis using Harrison's original presentation of bar elements modulo shuffles.

\subsection{Graphical models for cochains on thickened 2-complexes}\label{CurveForms}

For use in many further examples below and future 
work we develop a graphical method make calculations in Harrison cochains
on a manifold $M_G$, constructed in Appendix~\ref{sec:MG},  with given fundamental group 
presentation $G=\langle S\;|\;R\rangle$.
To briefly review, start with a four-ball, 
and for each generator $s\in S$ blow-up
at a chosen pair of points $\partial^{\pm}_s$ on the boundary and
attach $1$-handles $(D^3\times D^1)_s$
to obtain a handlebody $M_S$.  
Then for each relation $r\in R$ we blow-up along the image of an embedding $r:S^1\to \partial M_S$ which represents $r$ in the fundamental group 
of $M_S$ and  attach
$2$-handles $(D^2\times D^2)_r$
to define $M_G$, which is unique (only) up to homotopy. 

Recall Definition \ref{def:Thom form} that a Thom form associated to a submanifold $W^d$ of the four-manifold $M_G$ is a differential form of degree $4-d$, supported near $W$, whose integral over some $d$-dimensional submaniold counts intersections with $N$, 
when those are transverse. Below we use Thom forms associated to three types of submanifolds of $M_G$, namely 
{\bf belt disks} in the 1-handles and
{\bf thickened points} and {\bf curves} in the 2-handles,
as we now make precise. 

\begin{definition}

Write 
$D^3_{s,t} = D^3\times\{ t \} \subset (D^3\times D^1)_s\subset M_S$
for the {\bf belt disk} at time $t$ in the $1$-handle corresponding to $s\in S$. 
In the construction of $M_G$ by blowing up along relations $r:S^1\to \partial M_S$, each such belt disk is blown up at a finite collection of points
on its boundary, with one point for each occurrence of 
$s^{\pm 1}$ in each relation $r\in R$.
Call the resulting blowup of the disc
$(D^3_{s,t})_{\widehat R}$.
For every relation $r$ write 
$D^2_{s\in r,t}$
for the union of $2$-disks on the
boundary of $(D^3_{s,t})_{\widehat R}$ resulting from blowing up 
along its intersection with $r$.
Here we recall that blowing up at a point in $\partial D^3$ results in a $2$-disc whose boundary is an $S^1$-shaped corner.

The $2$-handle attachments resulting in ${M}_G$ identify each $2$-disk 
in $D^2_{s\in r,t}$  with some $\{y\}\times D^2$ in a handle 
$(D^2\times D^2)_r$ corresponding to a relation $r$.
Here $y \in \partial D^2$. The subspace
\[
\{ (\delta \cdot y , p) \mid \delta\in [1-\epsilon, 1] \text{ and } p\in D^2 \} \subseteq (D^2\times D^2)_r
\]
forms a three-dimensional \textbf{tab} that extends slightly into the interior of $(D^2\times D^2)_r$.

The union of $(D^3_{s,t})_{\widehat R}$ with all of its tabs
is a submanifold
with corners of ${M}_G$, which we call the {\bf indicator submanifold} associated to 
$s\in S$ and will denote by
${\mathcal{I}}_{s,t}$.
\end{definition}

The variable $t$ lets us choose multiple parallel representatives, as will
be needed for our models.  
 By abuse, if  $t$ is not provided it can be assumed to be $\frac{1}{2}$.
 Indicator submanifolds are codimension one, and their 
 Thom forms restrict to the Thom forms of belt disks on $M_S$ used in Section~\ref{topcombo}.  

 The boundary of ${\mathcal{I}}_{s,t}$ is internal to $M_G$ and consists of a union of $2$-discs $\{(1-\epsilon)\cdot y\}\times D^2 \subseteq (D^2\times D^2)_r$ where $(1-\epsilon)\cdot y$ is at distance $\varepsilon$ to $\partial D^2 \subset \partial M_G$ and there is one such disc for each occurrence of $s^{\pm 1}$ 
in the relation $r$.

\begin{definition}
    For each generator $s\in S$,  a Thom form for some $\mathcal{I}_{s,t}$ is called an {\bf indicator form} for $s$.  A choice of such will be denoted
     by $\tau_{s,t}$. If no $t$ is specified, it may be taken to be $t=\frac{1}{2}$.
\end{definition}

Indicator forms are supported on 1-handles as well as within
$\varepsilon$ of the boundary of two-handles. We usually fix a collection of such forms with disjoint support by varying the parameter $t$ slightly.  

We next define
forms with complementary supports. These are the focus of the figures given in the examples  later in this section.

\begin{definition}
    Assume a collection of belt disc indicator submanifolds has been fixed for each generator. Consider a relation $r\in R$.  We define
    the {\bf relation disk} to be the
    concentric disk of radius $1 - \varepsilon$ inside of $D^2 \times \{0\} \subset (D^2\times D^2)_r$.
    
    The boundary of a relation disk
    contains a collection of points $y$ such that
    $\{y\} \times D^2$ is a boundary component of some
    indicator submanifold, 
    corresponding to
    some generator $s\in S$ appearing in $r$.  Call these {\bf 
    indicator points}.  
    We say that a curve in the relation disk is {\bf admissible} if each of its endpoints is either an interior
    point or an indicator point.
\end{definition}

\begin{definition}
    Let $S$ be a submanifold -- that is, a union of points or curves -- in a relation disk.  The {\bf thickening}
    of $S$ is just $S \times D^2 \subset (D^2 \times D^2)_r$.  
    
    We use $\tau_S$ to denote a Thom form for the thickening of $S$, which we call a {\bf point form} or {\bf curve form} when $S$ is a point
    or curve respectively.  

    Point forms, curve forms and indicator forms above are collectively called {\bf graphical cochains.}
\end{definition}


We now introduce a combinatorial tool for calculation in the de Rham cochain algebra of $M_G$.

\begin{definition}
     A {\bf graphical calculation} in $\Omega_{dR}^*(M_G)$ is a set of graphical cochains in the set of relation disks such that:
    \begin{itemize}
        \item The coboundary
    of a curve form is the difference of the point forms at its endpoints.
    \item  The wedge product of two curve forms, for curves which intersect transversally in their interiors, is the signed sum of point forms for their intersection.
        \item  The coboundary of an indicator
        form associated to $s\in S$ is the sum of point forms for the indicator points corresponding to the generator $s$
        in all relation disks.
        The mass at an indicator point is $1$ for an occurrence of $s$ and $-1$ for an occurrence of $s^\inv$.
    \end{itemize}
\end{definition}

One can use graphical calculations to perform calculations in the Harrison complex of $M_G$.  At their simplest form, such calculations compute  products.  For example, for the free abelian group $\langle a,b\;|\; aba^{-1}b^{-1}\rangle$, the manifold $M_G$ is a thickened torus, and the product in its cohomology is represented by a graphical calculation whose underlying graph consists of two perpendicular segments
intersecting at a point.

\begin{convention}\label{curveconvention}
We use the capitalized name of a generator in $ S$ to label a choice of curve which cobounds the collection of indicator
points labeled by that generator. If $a \in S$, the curve form $\tau_A$ added to   indicator
form $\tau_a$ then restricts to the corresponding indicator form on $M_S$ with the same
name, as employed in Section~\ref{FreeCase}. 
Moreover, by abuse, in the examples below we abbreviate the  form $\tau_a + \tau_A$
by just $A$ when we make calculations in the
Harrison complex.
\end{convention}

When 
a linear combination of  curve and indicator forms so named is a cocycle,  for example $B-C$ in Figure~\ref{fig:favoritelittle}, 
this convention realizes the inclusion of $H^1(M_G)$ in  $H^1(M_S)$ at the cochain level.
Such naming thus records the map on zeroth Harrison cohomology, 
which is injective by Proposition~\ref{H0injective}.
This convention is also consistent with the surface
case, giving the same algebraic cochain calculations as 
illustrated in Figure~\ref{fig:genus 2}. (The
forms involved would differ at a point-set level since
$M_G$ is not simply the surface crossed with a two-disk.)

Constructing graphical calculations  requires some care. For example, a curve $C$ with 
two endpoints $x$ and $y$ in a
relation disk would be part of the data of a graphical calculation only if 
$d \tau_C = \tau_x - \tau_y$. This can be arranged easily if one first chooses $\tau_C$
and then defines $\tau_x$ and $\tau_y$ to 
be summands of $d \tau_C$.  
This is not as immediate if the Thom forms at $x$ and $y$ had already been fixed, say when $x$ and $y$ are intersection points of other curves and so their Thom forms are specified by the wedge product of the two curve forms. In that case we use Theorem~\ref{Thomboundary} to know an appropriate
$\tau_C$ exists.   

We leave for future work the formalization of which graphs occur as data for graphical calculations along with 
a proof that forms can always 
be associated to such
graphs in order to define graphical calculations.
Instead we give the following example relevant to Section~\ref{commutatormassey}, omitting similar reasoning for other cases below.

\begin{example}
    We show that there are forms which realize the graph presented in Figure~\ref{MasseyFig} as data for a graphical calculation of a Massey product.  First, we choose an $\varepsilon$
    so that  tubular neighborhoods around these curves of that
    width only intersect in neighborhoods of intersection points.
    
    We start with an indicator form
    for $a$ supported within $\varepsilon$ 
    of the indicator submanifold. 
    Its coboundary is given by point forms, also with
    support within $\varepsilon$, at four places on the boundary.  These four have been cobounded in Figure~\ref{MasseyFig} by a (disconnected) curve, which following Convention~\ref{curveconvention} is labelled $A$. This curve is oriented from the $a^\inv$ end to the to $a$ end on each component.  We use Theorem~\ref{Thomboundary} to set the curve form $\tau_A$ to cobound the four point forms, noting that we can continue to have
    support within $\varepsilon$.  The sum of $\tau_a$ and $\tau_A$ is thus a cocycle,
    representing the cohomology
    class corresponding to the indicator homomorphism 
    sending $a \mapsto 1$, and $b, c \mapsto 0$.
    We choose $\tau_b$ to then construct
    $\tau_B$ and $\tau_c$ to construct $\tau_C$ 
    similarly.  

    As $A$ and $B$ intersect transverally,
    the wedge product of $\tau_A$ and $\tau_B$ is a sum of point forms at the two indicated points,
    with opposite signs.
    The curve $S$ has been provided to cobound these
    two points, and we again use Theorem~\ref{Thomboundary} to choose a $\tau_S$ whose coboundary is these two point forms.  

    Finally, the wedge product of $\tau_C$ and $\tau_S$ is a Thom form for the intersection point, which represents
    the non-zero Massey product $\langle A, B, C \rangle$, graphically calculated. 
\end{example}

 In future work we plan
to develop  graphical calculations  fully, 
showing they can be used for any calculation in the Harrison complex of $M_G$ and how, at least in some cases, they give 
partially defined algebras which provide models for the cochain algebras 
of two-complexes. 
These models are interesting in their own right, leading to questions such
as how they reflect group theoretic properties and for which groups such models are finite.
We also wonder whether
geometry can be used to 
determine all products, either by refined
choices of Thom forms or by ambient geometry provided in cases such as that of hyperbolic groups.

\subsection{An interesting little 2-complex}\label{littletwo}

Consider the group $G$ presented as
$\langle a, b, c\; |\; a b a^\inv c = 1 \rangle$,  the fundamental group of a two-complex $X$ homeomorphic
to a cylinder with a pair of points on its boundary
identified.  It is homotopy equivalent to a wedge of two circles,
so $G$ is isomorphic to a free group on two generators, say
$a$ and $b$.  Let $A$, $B$ and $C$ be indicator 
Thom cochains associated to curves
as pictured in Figure~\ref{fig:favoritelittle}.

\begin{figure}
    \begin{tikzpicture}[scale=.4]

\coordinate (1) at (0,0);
\coordinate (2) at (10,0);
\coordinate (3) at (10,6);
\coordinate (4) at (0,6);

\draw (1)--(2)--(3)--(4)--(1);

\draw node[fill, circle, scale=.7] at (1) {};
\draw node[fill, circle, scale=.7] at (2) {};
\draw node[fill, circle, scale=.7] at (3) {};
\draw node[fill, circle, scale=.7] at (4) {};

\begin{scope}[ultra thick,decoration={
    markings,
    mark=at position 0.55 with {\arrow[black,scale=1.0]{ang 90}}}
    ] 
    \draw[postaction={decorate}] (2)--(1);
    \draw[postaction={decorate}] (2)--(3);
    \draw[postaction={decorate}] (3)--(4);
    \draw[postaction={decorate}] (4)--(1);
\end{scope}


\draw node[black, scale=.8] at (5,-1.2) {$a^{-1}$};
\draw node[black, scale=.8] at (5,7.2) {$a$};
\draw node[black, scale=.8] at (-1.2,3) {$b$};
\draw node[black, scale=.8] at (11.2,3) {$c$};


\coordinate (p) at (3.5,3.5);
\coordinate (q) at (5.5,2.8);

\begin{scope}[ultra thick,decoration={markings,
    mark=at position 0.8 with {\arrow[blue,scale=1.0]{ang 90}}}
    ] 
    \draw[postaction={decorate}][blue] 
       (4.2,5.93) to[in=90,out=270] (p) 
                  to[out=270,in=90] (4.2,0.07);
\end{scope}
\draw node[blue, scale=1.1] at (2.8,1.6) {$A$};
\draw node[blue, fill, circle, scale=.3] at (4.2,6) {};
\draw node[blue, fill, circle, scale=.3] at (4.2,0) {};

\begin{scope}[ultra thick,decoration={markings,
    mark=at position 0.5 with {\arrow[red,scale=1.0]{ang 90}}}
    ] 
    \draw[postaction={decorate}][red] 
       (0.07,2.2) to[out=0,in=180] (p) 
                  to[out=0,in=180] (q);
\end{scope}
\draw node[red, scale=1.1] at (2.0,4.4) {$B$};
\draw node[red, fill, circle, scale=.3] at (0,2.2) {};

\begin{scope}[ultra thick,decoration={markings,
    mark=at position 0.70 with {\arrow[ForestGreen,scale=1.0]{ang 90}}}
    ] 
    \draw[postaction={decorate}][ForestGreen] 
       (9.93,3.8) to[out=180,in=0] (7.5,2.8) 
                  to[out=180, in=0] (q);
\end{scope}
\draw node[ForestGreen, scale=1.1] at (7.2,1.6) {$C$};
\draw node[ForestGreen, fill, circle, scale=.3] at (10,3.8) {};

\draw node[violet, fill, circle, scale=.6] at (p) {};

\begin{scope}[ultra thick,decoration={markings,
    mark=at position 0.70 with {\arrow[violet,scale=1.0]{ang 90}}}
    ] 
    \draw[postaction={decorate}][violet] 
       (9.93,4.2) to[out=190,in=0] (6.5,4.5) 
                  to[out=180, in=45] (p);
\end{scope}
\draw node[violet, scale=1.1] at (6.6,5.2) {$S$};
\draw node[violet, fill, circle, scale=.3] at (10,4.2) {};

\draw node[black, fill, circle, scale=.3] at (q) {};

\end{tikzpicture}
    \vspace{-1cm}   
    \caption{}\label{fig:favoritelittle}
\end{figure}

In weight one $A$ and $B-C$ generate the cohomology of $X$.
In weight two 
$\tpc{B-C} A - C$
is a Harrison cocycle.
It pulls back to a cocycle of the same name on the one-skeleton, which presents the free group $\langle a,b,c\rangle$
to which we associate the letter linking invariant 
$\lp{B-C} A - C$.
In the free group, the count of $a$'s between collections of cancelling $b$'s and $c$'s changes under introduction of the relation word $aba^{-1}c$,
increasing by one, but as it does so does the number of $c$'s. Their difference thus descends to an invariant on $G$.
One could also choose $\lp{B-C}{A} - B$ as a representative, differing
by $B-C$ which is itself a cocyle of lower weight.  

In higher weights the story is conceptually simple but the
combinatorics can be complicated.  Consider the cofree
Lie coalgebra generated by $A, (B-C)$.  The Hopf invariants
for $G$ must be isomorphic to this coalgebra as $G$ is free on $a$ and $b$.
Accordingly, we claim that any
element of the cofree Lie
coalgebra on $A$ and $B-C$
can be completed to a cocycle. This is because
the Harrison differential
involves products, realized by Thom forms of points in the square in Figure \ref{fig:favoritelittle},
but these can be cobounded by a cochain in which one replaces every such point by a parallel copy of the curve $B$. In some cases, this replacement itself has Harrison differential involving further products, which lead to 
further accompanying correction terms
in lower weights. Thus, the correspondence between Hopf invariants
for $G$ and those for the free group on two generators
is not tautological.

\subsection{Milnor invariants for Borromean rings}\label{milnor}

Application of this theory to Milnor invariants has been a prime motivation, as first discussed in Section~\ref{geomtop}.

Consider the Borromean rings, with the two presentations given in Figure~\ref{borromean}.
\begin{figure}
\begin{tikzpicture}
\draw[draw=none,ultra thick] (0,0) circle (1.05);
\draw[draw=none,ultra thick] (3,0) circle (1.05);

\draw[ name path=curveRed,red, line width=2mm] (.3,0) to[out=-80,in=230] (2,-.4) to[out=50,in=100] (2.7,-.4) to[out=-80,in=30] (2.27,-.81) to[out=210,in=-80] (-.1,-.15);
\draw[ name path=curveWhite,white, line width=1.6mm] (.3,0) to[out=-80,in=230] (2,-.4) to[out=50,in=100] (2.7,-.4) to[out=-80,in=30] (2.27,-.81) to[out=210,in=-80] (-.1,-.15) to[out=100,in=225] (.375*1.05,.927*1.05) to[out=35,in=180] (1.5,1.3) to[out=0,in=90] (2.7,.6) to[out=-90,in=-90,looseness=1.4] (2.20,.5) to[out=90,in=100] (.3,0);

\draw[ultra thick, name path=curveMain] (.3,0) to[out=-80,in=230] (2,-.4) to[out=50,in=100] (2.7,-.4) to[out=-80,in=30] (2.27,-.81) to[out=210,in=-80] (-.1,-.15) to[out=100,in=225] (.375*1.05,.927*1.05) to[out=35,in=180] (1.5,1.3) to[out=0,in=90] (2.7,.6) to[out=-90,in=-70,looseness=1.4] (2.4,.6) to[out=110,in=100] (.3,0);

\node[fill,white,rotate=60,scale=1] at (.4, .97)   (a) {};
\node[fill,white,rotate=30,scale=1] at (.81, .68)   (b) {};
\node[fill,white,rotate=119,scale=1] at (2.57, .96)   (c) {};
\node[fill,white,rotate=30,scale=.5] at (2.32, -.79)   (c) {};
\node[fill,white,rotate=30,scale=.5] at (2.20, -.83)   (c) {};
\node[fill,white,rotate=30,scale=.5] at (2.19, -.71)   (c) {};
\node[fill,white,rotate=30,scale=.5] at (2.36, -.9)   (c) {};

\draw[ultra thick] (.324,-.999) to[out=200,in=0] (0,-1.05) to[out=180,in=-90] (-1.05,0) to[out=90,in=180] (0,1.05) to[out=0,in=90] (1.05,0) to[out=-90,in=65] (.952,-.444);
\draw[ultra thick] (.675,-.804) to[out=35,in=-130] (.755,-.729);
\draw[ultra thick] (2.091,-.525) to[out=-70,in=180] (3,-1.05) to[out=0,in=-90] (4.05,0) to[out=90,in=0] (3,1.05) to[out=180,in=40] (2.398,.860);
\draw[ultra thick] (2.196,.675) to[out=225,in=90] (1.95,0) to[out=-90,in=101] (1.966,-.182);

\draw[red] (-.129,-.11)  arc[start angle=110, end angle=438, x radius=.093, y radius=.05];
\draw[red] (.27,.04)  arc[start angle=110, end angle=438, x radius=.093, y radius=.05];
\draw[red,dot diameter=.2pt,dot spacing=.6pt,dots] (1.7,-.575) arc[start angle=30, end angle=348, x radius=.035, y radius=.093];
\draw[red,dot diameter=.2pt,dot spacing=.6pt,dots] (1.4,-1.02) arc[start angle=30, end angle=348, x radius=.035, y radius=.090];
\draw[violet] (2.67,-.35)  arc[start angle=110, end angle=478, x radius=.087, y radius=.05];

\node[color=red] (L0) at (1.5,1.6){$L_0$};
\node[color=red] (L1) at (0,-1.35){$L_1$};
\node[color=red] (L2) at (3,-1.35){$L_2$};

\begin{knot}[
clip width=5,
clip radius=6pt,looseness=1.3]
\strand[ultra thick] (6.3,0+.5) arc[start angle=180, end angle=-180, x radius=1, y radius=1]node[red,above left]{$L_1$};
\strand[ultra thick] (9.4,0+.5) arc[start angle=0, end angle=360, x radius=.9, y radius=.8] node[red,above right]{$L_2$};
\strand[ultra thick] (7.8,-1.6+.5) arc[start angle=-90, end angle=270, x radius=1.2, y radius=.8] node[below,red]{$L_0$};
 \flipcrossings{3,4}
\end{knot}

\end{tikzpicture}
\vspace{-1cm}  
\caption{Two presentations of the Borromean rings}\label{borromean}
\end{figure}
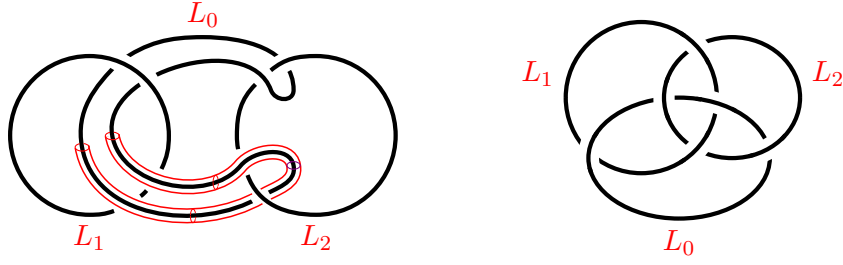
For brevity of discussion we ignore basepoints and orientations, and
we consider the component $L_0$ as defining an element
in the fundamental group of the complement of the other
components, $X = \R^3 - \{L_1 \cup L_2\}$.  
We again use Thom forms. 

In the presentation on the left of Figure \ref{borromean} we assume $L_1$ and $L_2$ are 
planar circles, and let $S_1$ be the planar disk which cobounds $L_1$ and
similarly for $S_2$.  Because these are disjoint,
the associated Thom forms $\tau_{1}$ and $\tau_{2}$, which form
a basis for first cohomology of $X$, define a Harrison 
cocycle $\linep{\tau_{1}}{\tau_{2}}$. With appropriate parametrizations,
its pullback 
to $L_0$ coincids with the Harrison cocycle on the circle
$\linep {\alpha}{\beta}$ from Section~\ref{first},
pulled back from a commutator on a handlebody.  
Thus the Hopf invariant will be given by traversing $L_0$ and counting 
intersections with $S_2$ between cancelling pairs of those with $S_1$.
This will evaluate to $\pm 1$.

Following Cochran \cite{cochran} 
one may make an alternate choice $S_1'$ for the
first Seifert surface that does not meet $L_0$, obtained by removing small disks at the intersection
points of $L_0$ with $S_1$ and replacing them by 
a cylinder $S^1 \times [0,1]$ which is the unit normal bundle to $L_0$
on one side, as pictured in red on the left in Figure~\ref{borromean}.  
Two things change when
using $S_1'$ in place of $S_1$:
first, $L_0$ does not intersect $S_1'$; and second,  
$\longlinep{\tau_1'}{\tau_2}{4}$
is not a Harrison cocycle on $X$, since $S_1' \cap S_2$ is nonempty. The latter intersection is a small circle
around an one of the intersection point $p$ of $L_0$ with $S_2$.  

To get a cocycle from $\longlinep{\tau_1'}{\tau_2}{4}$ we add a correction term. 
Choose any surface $S_{12}$ which cobounds the circle $S_1'\cap S_2$ -- 
for example a small disk around $p$, or its complement in $S_2$, or 
a cobordism between this circle and $L_1$. Choosing a Thom form $\tau_{12}$ appropriately will complete
a Harrison cocycle $\longlinep{\tau_1'}{ \tau_2}{4} - \tau_{12}$. 
Note that any 
choice for $S_{12}$ and $\tau_{12}$ will yield the same value of $\pm 1$ when
evaluated on $L_0$, as this will always be the linking number of $L_0$
with the circle $S_1' \cap S_2$.  
Indeed, Cochran's perspective
is that high Milnor invariants of links are realized as lower weight Milnor invariants of a ``derivative''
link: this is the intersection circle. 

We can then modify $S_2$ to eliminate its intersections with $L_0$. However,
``nesting'' the cylinder added to $S_2$ for that purpose with the existing 
one, the intersection of $S_1'$ and $S_2'$ will be essentially
the same as that of $S_1'$ and $S_2$, amenable to identical analysis.

\medskip

Instead of the approaches given above, one can work directly with a standard diagrammatic 
presentation, given on the right of Figure~\ref{borromean}.
Following Seifert's original cobounding algorithm, suppose each link component
to be planar except on some arbitrarily small intervals near
crossings, with the $L_0$ plane lower than that of $L_1$, which
is in turn lower than that of $L_2$. Thus where $L_2$ under-crosses
$L_1$ it must ``push down into the page,'' which can be assumed to happen in an
arbitrarily small neighborhood of the crossing.  

We choose ``mostly 
flat'' Seifert surfaces $S_1$ and $S_2$ laying their respective planes,
disjoint
except near the $L_1$-$L_2$ crossings, where they intersect due to 
$L_2$ dipping below the $L_1$ plane for the undercrossing.  
The intersections
are curves with one endpoint on $L_1$ and the other on
$L_2$, which in the projection appear arbitrarily close
to the crossing points.  Thus the associated Thom forms $\tau_1$
and $\tau_2$ have wedge product which is a Thom form for
the union of these two curves.  

To complete $\longlinep{\tau_1}{\tau_2}{4}$ to a Harrison cocycle
we find a cobounding surface for this intersection.  Once again
using the projection as faithfully as possible, we can take
the closed loop $L'$ consisting of the $S_1\cap S_2$ intersection
curves along with
a portion $L_1$ between these crossings, and a portion 
of $L_2$ between them, say the shorter portions for definiteness.
As shown in Figure~\ref{fig:s12 in borromean}, 
we cobound this closed loop with a disk $S_{12}$
which agrees with $S_2$
except near the crossings and near the $L_1$ part of the boundary
where it must ``fall downward''.  With these new definitions, our Harrison cocycle has
the same form $\tp {\tau_1} \tau_2 - \tau_{12}$ as above.

\begin{figure}
\begin{tikzpicture}
 \begin{knot}
[clip width=5,
 clip radius=6pt,looseness=1.3]

\strand[ultra thick] (6.3,0+.5) arc[start angle=180, end angle=-180, x radius=1, y radius=1]node[red,above left]{$L_1$};
\strand[ultra thick] (9.4,0+.5) arc[start angle=0, end angle=360, x radius=.9, y radius=.8] node[red,above right]{$L_2$};
\strand[ultra thick] (7.8,-1.6+.5) arc[start angle=-90, end angle=270, x radius=1.2, y radius=.8] node[below,red]{$L_0$};
 \flipcrossings{3,4}
\end{knot}
\draw[draw=blue,fill=blue!20,ultra thick] 
  (7.66,.8) 
    to[out=180+65,in=180-58] (7.72,.1) 
    to[out=0  +30,in=180-15] (8.25,.20) 
    to[out=0  +70,in=-78]    (8.29,.65)
    to[out=180+15,in=-30]    (7.65,.8)
    node[blue,left]{$S_{12}$};
\draw[draw=blue!20,line width=1.7mm]
  (8.335,.41) to[out=180-7,in=-5] (8,.48);
\draw[ultra thick]
  (8.35,.41) to[out=180-15,in=-5] (8,.48);
\draw[blue,fill=blue] (8,.48) circle (1pt);
\end{tikzpicture} 
\caption{}\label{fig:s12 in borromean}
\end{figure}

By construction, since $L_0$ is generally below all the chosen surfaces,
the intersection of $L_0$ with these surfaces would occur
only at crossings of $L_0$ over $L_1$ or $L_2$.  But since $L_0$ only
crosses under $L_2$, the pullback of 
$\linep{\tau_1}{ \tau_2}$ to
the circle parameterizing $L_0$ will be zero, as the pullback
of $\tau_2$ is zero.  The $\tau_{12}$ term evaluates as an intersection 
count of $L_0$ with $S_{12}$, or equivalently a linking number
with the ``derived'' link component $L' = \partial S_{12}$.  
There is one such intersection point,
namely near the $L_0$-over-$L'$ crossing.

Thus, in the setting of links, our formalism gives distinct approaches to realizing the same set of Milnor invariants: one can
calculate directly from a projection, or look to minimize some types
of intersections, or make use of some other geometric 
structure. As discussed in Section~\ref{Introwrap-up}, part~\ref{milnorfuture}, we plan to develop this further in a sequel.

\subsection{Beyond letter linking}\label{beyond}

Letter
linking is by definition a type of \emph{obstruction
theory} for the lower central series of free groups, as each invariant requires vanishing of earlier invariants to be defined.  But Proposition~\ref{Ground} shows that all
Harrison cocycles on the circle can be reduced to weight one.  
Attempts to make this reduction explicit produce surprisingly rich
combinatorics, and also motivated the first author
to understand weight reduction on the classical
bar complex in \cite{LetterBraiding}.

Using the notation of Section~\ref{first},
we calculate
the Hopf pairing of the cocycle $\linep{A}{B}$ 
with the word $ab$ in the free group realized as the fundamental
group of the handlebody $B_2$. 
For definiteness, parameterize so that the curve defining 
$ab$ passes through the  $D_a$ at $t = \frac{1}{3}$
and $D_b$ at $t = \frac{2}{3}$. Then $\linep{A}{B}$ pulls back
to $\longlinep{\omega_{\frac{1}{3}}}{\omega_{\frac{2}{3}}}{6}$, where
$\omega_t$ is a Thom form at point $t \in S^1$.

Neither entry of this Harrison cochain cobounds, so classical letter linking is undefined.
But we
can still weight reduce by taking advantage of the fact that, by definition, any weight two
cochain with repeated entries is zero by
the anti-symmetry relation in the Harrison
complex.  Let 
$$I(t) = \int_0^t \omega_{\frac{1}{3}} - \omega_{\frac{2}{3}},$$
so that $dI = \omega_{\frac{1}{3}} - \omega_{\frac{2}{3}}$.  Then 
\begin{equation*}
d_\mE \left( \longlinep{I\,}{\omega_{\frac{2}{3}}}{5} \right) = \longlinep{\omega_{\frac{1}{3}}}{\omega_{\frac{2}{3}}}{6} - \longlinep{\omega_{\frac{2}{3}}}{\omega_{\frac{2}{3}}}{6} - I \cdot {\omega_{\frac{2}{3}}}.
\end{equation*}
As the  term $\longlinep{\omega_{\frac{2}{3}}}{\omega_{\frac{2}{3}}}{6}$
is zero, this exhibits a weight reduction of our original
pullback to the weight one cochain $I \cdot {\omega_{\frac{2}{3}}}$.  This then can be evaluated to $\frac{1}{2}$, as predicted by  Aydin's Formula, given in Corollary~\ref{ozbek}.
Similar arguments in higher weight naturally lead to Bernoulli polynomials in $I$ and invariants involving Bernoulli numbers, as dictated by the BCH expansion.

\subsection{Commutators and Massey products}\label{commutatormassey}

Our theory extends the time-honored connection between vanishing of  Massey products 
and homomorphisms out from the fundamental group.  
Consider $G = \bigl\langle a, b, c\ \bigl| \; \bigl[[a,b],c\bigr] = 1 \bigr\rangle$.  We again
realize $G$ as the fundamental group of a two-complex with one
cell, and employ Thom cochains for dual curves as pictured in Figure~\ref{MasseyFig}.

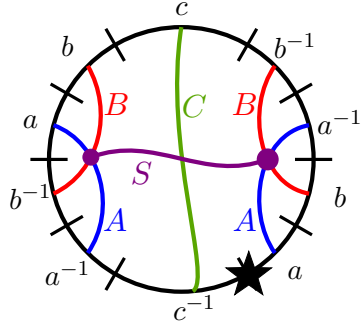
\begin{figure}
    \begin{tikzpicture}[scale=.35]

\draw[fill=none, ultra thick](0,0) circle (5);

\node[rotate=30,rectangle, very thick, black, inner sep=0pt, minimum height=0pt, minimum width=14pt,draw] at (4.33,2.5){};
\node[rotate=60,rectangle, very thick, black, inner sep=0pt, minimum height=0pt, minimum width=14pt,draw] at (2.5,4.33){};
\node[rotate=150,rectangle, very thick, black, inner sep=0pt, minimum height=0pt, minimum width=14pt,draw] at (-4.33,2.5){};
\node[rotate=120,rectangle, very thick, black, inner sep=0pt, minimum height=0pt, minimum width=14pt,draw] at (-2.5,4.33){};
\node[rotate=150,rectangle, very thick, black, inner sep=0pt, minimum height=0pt, minimum width=14pt,draw] at (4.33,-2.5){};
\node[rotate=210,rectangle, very thick, black, inner sep=0pt, minimum height=0pt, minimum width=14pt,draw] at (-4.33,-2.5){};
\node[rotate=240,rectangle, very thick, black, inner sep=0pt, minimum height=0pt, minimum width=14pt,draw] at (-2.5,-4.33){};
\node[rotate=300,rectangle, very thick, black, inner sep=0pt, minimum height=0pt, minimum width=14pt,draw] at ( 2.5,-4.33){};
\node[rotate=0,rectangle, very thick, black, inner sep=0pt, minimum height=0pt, minimum width=14pt,draw] at (5,0){};
\node[rotate=0,rectangle, very thick, black, inner sep=0pt, minimum height=0pt, minimum width=14pt,draw] at (-5,0){};

\draw[blue, ultra thick] (4.830,1.294) to[out=190,in=125] (3.536,-3.536){};
\draw[name path=path1,blue, ultra thick] (-4.830,1.294) to[out=-20,in=55] (-3.536,-3.536){};
\draw[red, ultra thick] (3.536,3.536) to[out=-125,in=170]  (4.830,-1.294){};
\draw[green!65!red!100, ultra thick] (.436,-4.980) ..controls (1.2, -4.5) and (-.7,0.5) ..(0,5);
\draw[name path=path2,red, ultra thick] (-3.536,3.536) to[out=-60,in=25] (-4.830,-1.294){};

\coordinate (B) at (3.265,0);

\draw[violet,fill,name intersections={of= path1 and path2}] (intersection-1) circle[radius=.30];
\draw[ultra thick,violet,out=20, in =200] (intersection-1)  to (B);
\node[violet,fill,circle, scale=.80] at (B){};

\node[scale=1.0,black] (bi) at (4.3,4.3){$b^{-1}$};
\node[scale=1.0,black] (b) at (-4.3,4.3){$b$};
\node[scale=1.0,black] (c) at (0,5.7){$c$};
\node[scale=1.0,black] (ci) at (.5,-5.6){$c^{-1}$};
\node[scale=1.0,black] (a) at (4.3,-4.3){$a$};
\node[scale=1.0,black] (ai) at (-4.3,-4.3){$a^{-1}$};

\node[scale=1.0,black] (a) at (6,1.41){$a^{-1}$};
\node[scale=1.0,black] (ai) at (-5.7,1.41){$a$};
\node[scale=1.0,black] (b) at (6,-1.41){$b$};
\node[scale=1.0,black] (bi) at (-5.7,-1.41){$b^{-1}$};

\node[scale=1.1,red] (B) at (-2.45,2.1){$B$};
\node[scale=1.1,red] (B) at (2.40,2.1){$B$};
\node[scale=1.1,blue] (A) at (-2.5,-2.3){$A$};
\node[scale=1.1,blue] (A) at (2.4,-2.3){$A$};
\node[scale=1.1,violet] (S) at (-1.5,-.5){$S$};
\node[scale=1.1,green!65!red!100] (C) at (.5,2){$C$};

\node [fill=violet,black,star,star points=5,star point height =2.7mm,scale=.70,draw] at (2.55,-4.38){};

\end{tikzpicture}
    \vspace{-1cm}  
    \caption{Graphical cochains for a group defined as a quotient of a two-fold commutator}\label{MasseyFig}
\end{figure}

Building on our work in  Section~\ref{CurveForms}, the Thom cochains of curves
$A$, $B$ and $C$ along with corresponding
indicator forms define cocycles.  These
are Harrison cocycles of $M_G$ 
extending the corresponding cocycles on the 
handlebody $M_{\{a,b,c\}}$, which are given the same
name in Section~\ref{FreeCase}.
Moreover, in the Harrison complex 
$\linep{A}{B}$ extends to the cocycle $\linep{A}{B} - S$, 
$\linep{A}{C}$ extends to $\linep{A}{C}$ with no correction 
term needed, and similarly for $\linep{B}{C}$.  The Harrison
cochain $\graphpp{A}{B}{C}$ however cannot be completed to a 
cocycle.  An attempt to complete it to a cocycle begins with the correction term
$\graphpp{A}{B}{C} - \linep{S}{C}$, whose total coboundary
has weight one, but this coboundary
is a Thom cochain of the intersection
point of $S$ and $C$, which does not cobound any 
cochain in weight one of the Harrison complex.

The integral of the last total coboundary on the fundamental class
of the two-disc is $1$, as there is one intersection point. This value coincides with the letter linking invariant $\lp[\big]{\lp A B} C$ 
associated to $\graphpp{A}{B}{C}$
computed for the defining relation $r = \bigl[[a,b],c\bigr]$.  Such an equality holds
in general, as proven in Lemma~\ref{circles-and-disks}.  We can 
see in this example how the process of picking cobounding curves
and intersecting them in the disc parallels
the cobounding and intersecting of letters in words as in letter linking.
For example, in Figure \ref{MasseyFig} the $A$ curves each are ``parallel to'' 
intervals of letters between $a - a^\inv$ pairs on the boundary, and
their intersections with $B$ curves can be traced back to occurrences of $b$'s between those pairs. In other
words, if we restrict the extensions of these curves
from the boundary to the two-disk to lie in an epsilon neighborhood of the boundary, their combinatorics matches that of letter linking invariants. We originally pursued
proofs of the lifting Theorem~\ref{LiftingCriterion} which codified
this.

This calculation coincides with that of 
a Massey product, namely $\langle A, B, C \rangle = 
S \cdot C$, in
the cohomology of the given two-complex.  Indeed, 
as mentioned in Remark~\ref{HarrisonSS} and Appendix~\ref{StandardBar}, Massey products define the
differentials into the weight one column in the Harrison spectral sequence.
As such, Massey products tautologically
obstruct the possibility to extend cocycles from the Harrison complex of the one-skeleton
of a CW-complex to the whole complex. 

\subsection{Pure braid groups}\label{purebraid}

We apply machinery of the Harrison complex to study the braid group, realized as the fundamental group of the space ${\rm Conf}_k(\C)$ of configurations of $k$ distinct labelled points in the plane. Interestingly, this space is intimately tied to our presentation of the Harrison complex as its $(k-1)$st cohomology is isomorphic to ${\Eil(n)}$ (up to ``signs'').

The Harrison complex of the configuration space thus consists of Eil-graphs decorated by Eil-graphs, though for clarity we use the more
traditional presentation of cohomology as generated by classes
$ a_{ij} \in H^1{\rm Conf}_k(\C)$ modulo 
the Arnold identity.

At the cochain level let $\alpha_{ij}$ be the Thom form associated to the locus having point $i$ directly above point $j$, i.e.  $x_i - x_j$ is a positive scalar multiple of the imaginary unit $\sqrt{-1}$.  
We (co)orient this submanifold so that the ``positive'' normal direction is when point $i$ is slightly to the left.  Then for example the wedge product $\alpha_{ij}\wedge \alpha_{jk}$ will be a Thom form for the submanifold where \{point $i$ is above $j$
which is itself above $k$\}.

Consider the codimension one submanifold with boundary $Ar_{ijk}$ defined as the locus of tuples in which \{point $i$ lies above point $j$,
and point $k$ is to the right of the vertical line through $i$ and $j$\}. In other words, the real part of point $k$ is at no smaller than the shared real part of points $i$ and $j$.
The boundary $\partial Ar_{ijk}$ is the locus where the real part of $k$ is equal to that of the other two distinguished points. It naturally consists of three components, depending on whether point $k$ is above $i$, in between, or below $j$.  Thus $Ar_{ijk}$ cobounds the Arnold identity:
$$d Ar_{ijk} = \alpha_{ki} \wedge \alpha_{ij} + \alpha_{ik} \wedge \alpha_{kj} + \alpha_{ij}\wedge \alpha_{jk}.$$
With this, we can immediately construct a Harrison cocycle
$$\linep{\alpha_{ki}}{\alpha_{ij}} + \linep{\alpha_{ik}}{\alpha_{kj}} + \linep{\alpha_{ij}}{\alpha_{jk}} - Ar_{ijk}.$$
By bracket-cobracket compatibility, this cocycle pairs non-trivially on the commutator
$[b_{ki}, b_{ij}] = [b_{ik}, b_{kj} ] = [b_{ij}, b_{jk}]$ where $b_{ij}$ is the (logarithm of) the standard pure braid generators braiding the $i$ and $j$ strands.\footnote{The equality of these commutators is the defining relation of a Drinfeld-Kohno Lie algebra -- the Malcev Lie algebra of the pure braid group.}

The corresponding Hopf invariant consists of four counts, which we interpret in terms of standard braid diagrams in a planar projection.   
The $Ar_{ijk}$ term counts $i$-over-$j$ crossings, but only those for which strand $k$ lies {to the right of the crossing}.
The count associated with $\linep{\alpha_{ki}}{\alpha_{ij}}$ can be interpreted in several ways, e.g. using weight reduction it yields a count $k$-over-$i$ crossings ``between canceling pairs of $i$-over-$j$ crossings'', all counted with signs which account for the orientation of the crossing relative to the plane. Geometrically, such a configuration has the strands $i$ and $j$ not linked, but strand $k$ is 
``coming between them'' and so possibly linking.

\subsection{An infinite example}

Consider the group generated
by $y$ and $x_i, i \in \mathbb{N}$  with $[x_{2i-1}, x_{2i}] = y$
for all $i$.  Such a group is the fundamental group of a 
countable union of genus one surfaces with one boundary component,
 with all such boundary components identified with each other.
 We consider closed 
 dual curves as in Section~\ref{surfaces} (with
 the genus two surface being the $n=2$ cousin of this infinite
 complex), so that $X_{2i-1}$ intersects $X_{2i}$ in a point
 $p_i$ and otherwise these curves are disjoint.  

 Possibly infinite unions of such curves define proper
 submanifolds of the thickening of this complex and
 thus Thom cochains.  These realize the first cohomology group,
 which is an infinite product rather than sum.  

The Harrison cochains $\tp{X_i}{X_j}$ are coclosed whenever $i,j \neq 2k, 2k-1$ for any $k$, and they form a dual basis to the corresponding commutators 
 $[x_i, x_j]$.  The generator $y$ is detected by 
 $\tp{\bigcup X_{2i-1}}{\bigcup X_{2i}} + Y$, which also pairs 
 with any $[x_{2i-1}, x_{2i}]$.  This Harrison cocycle is not expressible using compactly supported cochains. 
 Key here is that the first cohomology is already an infinite direct product, so
 infinite sums can be ``embedded'' in the standard
 Harrison complex. 

\subsection{Beyond residual nilpotence using the product totalization}\label{beyondnilpotence}

First recall from Definition~\ref{D:E} 
that the Harrison complex is defined as the standard total complex
of a bicomplex $\mE(A,\ \mu_A,\ d_A)$.  There is an alternative
totalization of a complex ${\rm Tot}_\Pi$
defined using direct product in each total degree,
\[
{\rm Tot}_{\Pi} \maps ( E^{\bullet,\bullet}) \longmapsto \prod_{p+q = \bullet} E^{p,q}
\]
from which we get a variant of our theory.

\begin{definition}
    
Define the {\bf completed Harrison
  complex} as the direct product totalization 
 $$\hat{\mE}(A,\ \mu_A,\ d_A) =
    {\rm Tot}_{\Pi} \bigl(\vscofreeE(s^\inv\bar A),\ 
  d_{\vscofreeE s^\inv\!\bar A},\ d_\mu\bigr).$$
  \end{definition}

We will not fully develop a theory based on the completed
Harrison complex here, but do give an example to show
that it detects phenomena beyond the residually
nilpotent setting, which has been a standard limitation for 
group theory through computational algebraic topology.

Consider the   group generated by $x_i$, $i \in {\mathbb N}$
with relations $x_{2i-1} = [x_{2i}, x_{2i + 1}]$.  Thus any
$x_{2i - 1}$ is in the intersection of all commutator
subgroups.  We realize this as the fundamental group of the 
quotient of the union $\bigsqcup_{\mathbb N} \Sigma_i$, where each $\Sigma_i$ 
is  a genus one surface with one boundary component, by identifying the boundary of $\Sigma_{i+1}$ with the longitude of 
$\Sigma_i$, as in the theory of gropes. Explicitly, we build this two-complex
as a union of pentagons as pictured on the left of Figure~\ref{fig:pentagons}.  
By abuse we use $x_i$ to decorate edges, which under identifications are loops generating the fundamental group.

\begin{figure}
  \begin{subfigure}[b]{0.40\textwidth}
    \begin{tikzpicture}[scale=.7]

\coordinate (1) at (2.8,0);
\coordinate (2) at (.309*2.8,-.9511*2.8);
\coordinate (3) at (-.8090*2.8,-.5878*2.8);
\coordinate (4) at (-.8090*2.8,.5878*2.8);
\coordinate (5) at (.309*2.8,.9511*2.8);

\begin{scope}[very thick,decoration={
    markings,
    mark=at position 0.55 with {\arrow[scale=1.4]{ang 90}}}
    ] 
    \draw[postaction={decorate}] (2)--(1);
    \draw[postaction={decorate}] (3)--(2);
    \draw[postaction={decorate}] (3)--(4);
    \draw[postaction={decorate}] (4)--(5);
    \draw[postaction={decorate}] (5)--(1);
\end{scope}

\draw node at(-3.0,-0.5){$x_{2i-1}$};
\draw node at(-1,2.8){$x_{2i+1}$};
\draw node at(-1,-2.8){$x_{2i}$};
\draw node at(2.6,1.7){$x_{2i}$};
\draw node at(2.6,-1.7){$x_{2i+1}$};

\end{tikzpicture}
  \end{subfigure}
  \begin{subfigure}[b]{0.40\textwidth}
    \begin{tikzpicture}[scale=.7]

\coordinate (0) at (0,0); 
\coordinate (1) at (2.8,0); 
\coordinate (2) at (.309*2.8,-.9511*2.8);; 
\coordinate (3) at (-.8090*2.8,-.5878*2.8); 
\coordinate (4) at (-.8090*2.8,.5878*2.8); 
\coordinate (5) at (.309*2.8,.9511*2.8); 

\draw[very thick]plot coordinates{(1)(2)(3)(4)(5)(1)};

\draw[very thick] (1) -- coordinate (1x) (2)--coordinate (2x) (3)--coordinate (3x) (4)--coordinate (4x) (5) -- coordinate (5x) (1);
\draw node[circle,fill,scale=.5] at (1x) {};
\draw node[circle,fill,scale=.5] at (2x) {};
\draw node[circle,fill,scale=.5] at (3x) {};
\draw node[circle,fill,scale=.5] at (4x) {};
\draw node[circle,fill,scale=.5] at (5x) {};

\draw[thick] (1x)--(4x);
\draw[thick] (2x)--(5x);
\coordinate (i) at (intersection of  1x--4x  and 2x--5x);
\draw[thick] (3x)--(i);

\draw  [thick] (1)--(3.35,-.05);
\draw  [thick] (2)--(.96,-3.213);
\draw  [thick] (3)--(-2.7,-2.1);
\draw  [thick] (4)--(-2.9,1.9);

\begin{scope}[thick,decoration={
    markings,
    mark=at position 0.5 with {\arrow[scale=1]{ang 90}}}
    ] 
\draw[postaction={decorate}] (-2.9,-.2)--(3x);
\end{scope}

\draw  [thick,-ang 90] (1x)--(2.4,-1.75);

\draw node at(-1.5,2.4) {$x_{2i+1}$};
\draw node at(2.8,.9){$x_{2i}$};
\draw node at(.5,1.5){$Y_{2i+1}$};
\draw node at(.4,-1.5){$Y_{2i}$};
\draw node at(2.95,-1.7){$Z_{2i+1}$};
\draw node at(-1.0,-.4){$Z_{2i-1}$};

\end{tikzpicture}
  \end{subfigure}
  \vspace{-1cm}  
  \caption{}\label{fig:pentagons}
\end{figure}

We use apply graphical cochains.
Though we normally prefer to use capitalized versions of
generator names for dual cochains, for clarity
let $Y_{2i+1}$ be the dual curve to $x_{2i+1}$ in $\Sigma_i$,
 a segment between midpoints of the $x_{2i+1}$ edges in the pentagon
 preimage of $\Sigma_i$, and similarly for $Y_{2i}$,
 as pictured on the right in Figure~\ref{fig:pentagons}.
 Let $Z_{2i-1}$ be the segment between the midpoint of
 $x_{2i-1}$ and the intersection point of $Y_{2i}$
 and $Y_{2i+1}$.  It is crucial to observe that while $Y_{2i}$ is closed, $Y_{2i+1}$ is not and its boundary represents the midpoint of $x_{2i+1}$ in the next pentagon $\Sigma_{i+1}$. (Technically, $Y_{2i+1}$ includes a small tab extending from that midpoint into the interior of the pentagon -- see Subsection \ref{CurveForms}.) This boundary is the left endpoint of $Z_{2i+1}$, so the sum $Y_{2i+1}+Z_{2i+1}$ only has boundary coming from the right endpoint of $Z_{2i+1}$.

 With appropriate choices of Thom
 forms, the infinite sum
 $$
 (Y_1+Z_1) + \tp{Y_3+Z_3}Y_2 + \tp{\tp{Y_5+Z_5} Y_4 }Y_2 + \tp{\tp{ \tp{Y_7+Z_7} Y_6} Y_4 }Y_2 + \cdots
 $$
 will be a cocycle in the product Harrison complex which will
 evaluate non-trivially on $x_1$ through its $Y_1$ summand.  Algebraically, the
 corresponding letter linking count is ``sensitive to infinitely 
 many possible subwords,'' but is well defined since only
 finitely many of them could occur in any given word.

 This infinitary construction echoes other recent developments in non-simply connected homotopy theory.
 As discussed in Section~\ref{Introwrap-up},  Rivera--Zeinalian \cite{Rivera-chains_and_pi1} and Rivera--Medina \cite{Rivera-Medina_cobar} show that the cobar construction on singular chains of a space faithfully encodes the fundamental group, but in a way that does not lend itself
 to calculation because of lack of quasi-isomorphism invariance and divergent spectral sequences. 
 The completed Harrison
 complex might serve as a bridge to theories such as
 these, with some ability for
 calculations by hand while being able to see beyond 
 traditional nilpotence restrictions.

\appendix


\section{Bar and Harrison complexes}\label{BarHarrison}

\subsection{Basics of the bar construction connection with Sullivan models}\label{basicbar}

Bar constructions are standard in homological algebra
of modules, but are
less familiar in the context of algebras, for functors such
as indecomposibles.  We share this perspective, which is key to understanding the Harrison
complex and its relationship to the well-known Sullivan
models of rational homotopy theory (from which we have
borrowed our title).

\begin{definition}\label{StandardBar}
    
Let $(A,\ \mu_A,\ d_A)$ be an associative differential graded algebra
with differential of degree one and with augmentation ideal ${\bar A}$. 

The {\bf classical Bar complex} denoted $\B(A)$ is the second-quadrant bicomplex for which 
the $-n^\mathrm{th}$ column is ${\bar A}^{\otimes n}$, with 
basic tensors denoted $a_1 | a_2 | \cdots | a_n$.  
The vertical differential, on ${\bar A}^{\otimes n}$, is $d_A$,
extended by the Leibniz rule (with Koszul signs reflecting a desuspension of every tensor factor of $\bar A$).  
The horizontal differential 
is the signed sum of $\mu_A$ on consecutive factors, 
sending  
$$a_1 | a_2 | \cdots | a_n \;\; \longmapsto  \;\; \pm(a_1 a_2 | a_3 | \cdots | a_n) \; \pm \; (a_1 | a_2 a_3 | a_4 | \cdots | a_n) \; 
+   \cdots \pm  \; (a_1 | a_2 | \cdots | a_{n-1} a_n).$$
\end{definition}
The total complex is computing the derived indecomposables of $A$, which are (the derived version of) the quotient $A/\bar{A}^2$. Equivalently, this is the derived tensor product
$\Q \overset{\mathbb{L}}\otimes_{A} \Q$.

One can apply the spectral sequence of a bicomplex \cite{Elias} to get what is known as
the {\bf bar spectral sequence}, 
converging to the homology of the total
complex, 
a standard organizational tool.  
Taking homology with
respect to the vertical differential first, the spectral 
sequence has $E_1$-page given by $\B(H^*(A))$.  The $d_1$ differential reflects product structure on cohomology,
for example computing the indecomposibles of $H^*(A)$
in the zero column at the $E_2$-page. Higher
differentials reflect Massey products, and indeed this
can be taken as a definition of Massey products.  For
example, if $ab = d(x)$ and $bc = d(y)$ then 
$$\alpha = a | b | c \;  \mp \; a | y \; \pm \; x | c $$ will have total 
bar differential $d(\alpha) = \mp \; ay \; \pm \; xc$, which represents
the standard Massey product $\langle a, b, c\rangle$.  This 
will give rise to 
a non-zero $d_2$-differential in the bar spectral 
sequence if and only if the Massey product is non-zero
modulo standard indeterminacy.

The classical bar complex arises in topology in a few ways.  If $A$ is a group ring, its bar complex computes group homology with trivial
coefficients.  If $A$ is the cochain algebra of a simply connected space, the bar complex 
models cochains on the loop space.  
There are many perspectives for this latter fact, with two relevant for the present theory.  One is that if
$a_i$ are Thom forms as in Appendix~\ref{Thomforms} or geometric
cochains, then $\tau_{W_1} | \cdots | \tau_{W_n}$ essentially
counts intersections of a loop (within a family of loops defining a smooth chain on the loopspace) with the $W_i$, 
in time-order.  If $W_i$ has boundary, this count can change 
through a homotopy of a loop that crosses $\partial W_i$, or more generally a homology of a cycle of such loops.
The count can also change if $W_i$ and $W_{i+1}$ intersect, as we saw in
discussion of the Whitehead link in Section~\ref{geomtop}. 
These two possible ways this count of ordered intersections
can change correspond to the
vertical and horizontal differentials in the bar complex, respectively.

Another perspective, pertinent to our work, is that of defining invariants on homotopy
groups using cohomology.  
We let cohomology evaluate on
$f: S^n \to X$ by pulling back and integrating  cohomology
classes.  If a cohomology class is a non-trivial product,
its pull-back will be zero in the cohomology of the sphere.
Thus evaluation of cohomology on homotopy factors through the
indecomposibles of cohomology.  Hopf invariants can be viewed
as a ``derived version'' of this, in which case a key fact is
that the bar complex computes derived indecomposibles.  
Recall that by the Leibniz rule, the differential on an
associative differential graded algebra passes
to its indecomposibles.

\begin{theorem}\label{BarDerivedIndec}
      Let $F$ be an augmented associative differential graded algebra which is quasi-isomorphic to $A$ and whose underlying algebra is a free (tensor) algebra.  Then the complex $F/\bar{F}^2$ of indecomposibles of $F$ is quasi-isomorphic to $\B (A)$.
\end{theorem}

\begin{proof}
We first apply the basic fact that the bar complex is quasi-isomorphism invariant to replace $\B(A)$ by $\B(F)$.
Then we show that $\B(F)$ is quasi-isomorphic to the 
indecomposibles of $F$.  

To do so, identify $F$ as the free associative algebra on $x_1, x_2, \cdots$, in which case $F$ has a monomial basis which determines one on $\B (F)$.
Define a chain homotopy $P$ on $\B(F)$ by 
sending a monomial basis element $m_1 | m_2 | \cdots | m_n$ to
either $0$ if $m_1$ is a generator $x_i$ for some $i$ or otherwise,
when $m_1 = \prod x_{i_k}$, to $x_{i_1} | m_1'  | m_2 | \cdots | m_n$, where 
$m_1' = \prod_{k \geq 2} x_{i_k}$.  
Calculation then shows that
$P$ is a homotopy between the identity and the zero map except
on the indecomposibles of $F$, through the inclusion of $F$ as the zeroth
column, modulo the first differential (the product).
\end{proof}

For example, if $x \cdot y = 0$ in some associative algebra (with zero differential) then $x | y$ will be a
cycle in $\B(A)$.  In some quasi-isomorphic $F \overset{\simeq}{\to} A$ sending $\hat{x} \mapsto x$ and $\hat{y} \mapsto y$ we could have
$\hat{x} \cdot \hat{y} = dz$ for a new generator $z$.  Then $z$ would represent the indecomposible
of $F$ corresponding to the original cycle $x | y$.
Indeed in $\B(F)$ we have the lifted cycle $\hat{x} | \hat{y} \;\mp\; z$, with the original 
trivial product and its corresponding derived indecomposible both accounted for.

\subsection{The Harrison complex, and explicit treatment of signs}\label{signs}

If $A$ is commutative, then there are some canonical
cocycles in $\B (A)$, starting with $a | b \;  \mp \; b | a$.
Reflective of this, this bar complex does not represent
derived indecomposibles when restricted to the commutative setting (replacing the free associative algebra $F$ in the previous theorem 
by a symmetric algebra).   Harrison's original construction \cite{Harrison} 
is to quotient $\B(A)$
by shuffle relations to obtain a complex which  computes 
the indecomposibles of a free commutative
algebra weakly equivalent to $A$.

In \cite[Definition 2.14]{Sinha-Walter1}, we use
 directed, acyclic, connected graphs with formal span denoted $\Gr(n)$ to build  $\Eil(n)$, 
a model for the arity-$n$ entry 
of the Lie cooperad given by the 
quotient of $\Gr(n)$  by local relations similar to those of Definition~\ref{Eil}.
\begin{align*}
\text{(arrow reversal)}\qquad & 
0 \ = \!
\begin{xy}                           
  (0,-2)*+UR{\color{blue}\scriptstyle \bm{a}}="a",    
  (3,3)*+UR{\color{red}\scriptstyle \bm{b}}="b",     
  "a";"b"**\dir{-}?>*\dir{>},      
    "a"!<.3pt,0pt>;"b"!<.3pt,0pt> **\dir{-},  "a"!<-.3pt,0pt>;"b"!<-.3pt,0pt> **\dir{-},
  (1.5,-5),{\ar@{. }@(l,l)(1.5,6)},
  ?!{"a";"a"+/va(210)/}="a1",
  ?!{"a";"a"+/va(240)/}="a2",
  ?!{"a";"a"+/va(270)/}="a3",
  "a";"a1"**[blue]\dir{-},  "a";"a2"**[blue]\dir{-},  "a";"a3"**[blue]\dir{-},
  (1.5,6),{\ar@{. }@(r,r)(1.5,-5)},
  ?!{"b";"b"+/va(90)/}="b1",
  ?!{"b";"b"+/va(30)/}="b2",
  ?!{"b";"b"+/va(60)/}="b3",
  "b";"b1"**[red]\dir{-},  "b";"b2"**[red]\dir{-},  "b";"b3"**[red]\dir{-},
\end{xy}\!+ \! 
\begin{xy}                           
  (0,-2)*+UR{\color{blue}\scriptstyle \bm{a}}="a",    
  (3,3)*+UR{\color{red}\scriptstyle \bm{b}}="b",     
  "a";"b"**\dir{-}?<*\dir{<},         
    "a"!<.3pt,0pt>;"b"!<.3pt,0pt> **\dir{-},  "a"!<-.3pt,0pt>;"b"!<-.3pt,0pt> **\dir{-},
  (1.5,-5),{\ar@{. }@(l,l)(1.5,6)},
  ?!{"a";"a"+/va(210)/}="a1",
  ?!{"a";"a"+/va(240)/}="a2",
  ?!{"a";"a"+/va(270)/}="a3",
  "a";"a1"**[blue]\dir{-},  "a";"a2"**[blue]\dir{-},  "a";"a3"**[blue]\dir{-},
  (1.5,6),{\ar@{. }@(r,r)(1.5,-5)},
  ?!{"b";"b"+/va(90)/}="b1",
  ?!{"b";"b"+/va(30)/}="b2",
  ?!{"b";"b"+/va(60)/}="b3",
  "b";"b1"**[red]\dir{-},  "b";"b2"**[red]\dir{-},  "b";"b3"**[red]\dir{-},
\end{xy} \\
\text{(Arnold)}\qquad & 
0 \ = \!
\begin{xy}                           
  (0,-2)*+UR{\color{blue}\scriptstyle \bm{a}}="a",    
  (3,3)*+UR{\color{red}\scriptstyle \bm{b}}="b",   
  (6,-2)*+UR{\color{OliveGreen}\scriptstyle \bm{c}}="c",   
  "a";"b"**\dir{-}?>*\dir{>},         
    "a"!<.3pt,0pt>;"b"!<.3pt,0pt> **\dir{-},  "a"!<-.3pt,0pt>;"b"!<-.3pt,0pt> **\dir{-},
  "b";"c"**\dir{-}?>*\dir{>},         
    "b"!<.3pt,0pt>;"c"!<.3pt,0pt> **\dir{-},  "b"!<-.3pt,0pt>;"c"!<-.3pt,0pt> **\dir{-},
  (3,-5),{\ar@{. }@(l,l)(3,6)},
  ?!{"a";"a"+/va(210)/}="a1",
  ?!{"a";"a"+/va(240)/}="a2",
  ?!{"a";"a"+/va(270)/}="a3",
  ?!{"b";"b"+/va(120)/}="b1",
  "a";"a1"**[blue]\dir{-},  "a";"a2"**[blue]\dir{-},  "a";"a3"**[blue]\dir{-},
  "b";"b1"**[red]\dir{-}, "b";(3,6)**[red]\dir{-},
  (3,-5),{\ar@{. }@(r,r)(3,6)},
  ?!{"c";"c"+/va(-90)/}="c1",
  ?!{"c";"c"+/va(-60)/}="c2",
  ?!{"c";"c"+/va(-30)/}="c3",
  ?!{"b";"b"+/va(60)/}="b3",
  "c";"c1"**[OliveGreen]\dir{-},  "c";"c2"**[OliveGreen]\dir{-},  "c";"c3"**[OliveGreen]\dir{-},
  "b";"b3"**[red]\dir{-}, 
\end{xy}\!+ \!                            
\begin{xy}                           
  (0,-2)*+UR{\color{blue}\scriptstyle \bm{a}}="a",    
  (3,3)*+UR{\color{red}\scriptstyle \bm{b}}="b",   
  (6,-2)*+UR{\color{OliveGreen}\scriptstyle \bm{c}}="c",    
  "b";"c"**\dir{-}?>*\dir{>},         
    "b"!<.3pt,0pt>;"c"!<.3pt,0pt> **\dir{-},  "b"!<-.3pt,0pt>;"c"!<-.3pt,0pt> **\dir{-},
  "c";"a"**\dir{-}?>*\dir{>},          
    "c"!<0pt,.2pt>;"a"!<0pt,.2pt> **\dir{-},  "c"!<0pt,-.2pt>;"a"!<0pt,-.2pt> **\dir{-},
  (3,-5),{\ar@{. }@(l,l)(3,6)},
  ?!{"a";"a"+/va(210)/}="a1",
  ?!{"a";"a"+/va(240)/}="a2",
  ?!{"a";"a"+/va(270)/}="a3",
  ?!{"b";"b"+/va(120)/}="b1",
  "a";"a1"**[blue]\dir{-},  "a";"a2"**[blue]\dir{-},  "a";"a3"**[blue]\dir{-},
  "b";"b1"**[red]\dir{-}, "b";(3,6)**[red]\dir{-},
  (3,-5),{\ar@{. }@(r,r)(3,6)},
  ?!{"c";"c"+/va(-90)/}="c1",
  ?!{"c";"c"+/va(-60)/}="c2",
  ?!{"c";"c"+/va(-30)/}="c3",
  ?!{"b";"b"+/va(60)/}="b3",
  "c";"c1"**[OliveGreen]\dir{-},  "c";"c2"**[OliveGreen]\dir{-},  "c";"c3"**[OliveGreen]\dir{-},
  "b";"b3"**[red]\dir{-}, 
\end{xy}\!+ \!                             
\begin{xy}                           
  (0,-2)*+UR{\color{blue}\scriptstyle \bm{a}}="a",    
  (3,3)*+UR{\color{red}\scriptstyle \bm{b}}="b",   
  (6,-2)*+UR{\color{OliveGreen}\scriptstyle \bm{c}}="c",    
  "a";"b"**\dir{-}?>*\dir{>},         
    "a"!<.3pt,0pt>;"b"!<.3pt,0pt> **\dir{-},  "a"!<-.3pt,0pt>;"b"!<-.3pt,0pt> **\dir{-},
  "c";"a"**\dir{-}?>*\dir{>},          
    "c"!<0pt,.2pt>;"a"!<0pt,.2pt> **\dir{-},  "c"!<0pt,-.2pt>;"a"!<0pt,-.2pt> **\dir{-},
  (3,-5),{\ar@{. }@(l,l)(3,6)},
  ?!{"a";"a"+/va(210)/}="a1",
  ?!{"a";"a"+/va(240)/}="a2",
  ?!{"a";"a"+/va(270)/}="a3",
  ?!{"b";"b"+/va(120)/}="b1",
  "a";"a1"**[blue]\dir{-},  "a";"a2"**[blue]\dir{-},  "a";"a3"**[blue]\dir{-},
  "b";"b1"**[red]\dir{-}, "b";(3,6)**[red]\dir{-},
  (3,-5),{\ar@{. }@(r,r)(3,6)},
  ?!{"c";"c"+/va(-90)/}="c1",
  ?!{"c";"c"+/va(-60)/}="c2",
  ?!{"c";"c"+/va(-30)/}="c3",
  ?!{"b";"b"+/va(60)/}="b3",
  "c";"c1"**[OliveGreen]\dir{-},  "c";"c2"**[OliveGreen]\dir{-},  "c";"c3"**[OliveGreen]\dir{-},
  "b";"b3"**[red]\dir{-}, 
\end{xy}                             
\end{align*}
  
\begin{proposition}
  Rooted trees modulo root change and Arnold is isomorphic to 
  directed, acyclic, connected graphs modulo arrow reversal and Arnold -- that is, 
  $\Eil\Tr_* \cong \Eil$ -- through the map $\Tr_*\to\Gr$
converting a rooted tree to a directed graph by orienting all edges
towards the root.    
\end{proposition}
\begin{proof}
The map converting 
a rooted tree to a directed acyclic graph with all edges oriented 
towards the root is  well-defined modulo the tree relations, since root
change and root Arnold for trees are
sent to graph arrow reversal and graph Arnold 
(after reversing one arrow in each term) relations 
respectively.
This map is a surjection because 
given a graph
representative we can pick an arbitrary vertex to be the ``root'' 
and, modulo arrow reversal, orient all edges appropriately to
obtain a directed graph in the image of the map from rooted trees.  
Injectivity holds true because an
arbitrary Arnold relation for graphs is the 
image of a root
Arnold relation for rooted trees after orienting edges appropriately
by repeated arrow reversals. 
\end{proof}


The most popular  approach to rational homotopy theory, by far, is given by Sullivan models which are free commutative differential graded algebras.  Our present work
advertises the usefulness of bar constructions.  
These are related by a version of Theorem~\ref{BarDerivedIndec}.

\begin{theorem}\label{HarrisonDerivedIndec}
    Let $A$ and $P$ be quasi-isomorphic commutative differential graded algebras, where the undelying algebra of $P$ is free as a graded-commutative algebra.  Then the subcomplex of indecomposibles of $P$ is quasi-isomorphic to $\mE(A)$
\end{theorem}

For Sullivan models,  minimality essentially says
that the indecomposibes of $P$
have no induced differential.  Thus one can obtain a 
proof through Sullivan models that
the Harrison homology of $X$ is linearly dual to its rational
homotopy groups in the simply connected setting; though
the use of Harrison homology is closer to Quillen's
original approach.

The proof of Theorem~\ref{HarrisonDerivedIndec} 
parallels that of Theorem~\ref{BarDerivedIndec}.
The quasi-isomorphism invariance of the Harrison complex
was proven in Proposition~2.10 of \cite{LieEncyclopedia}, without any
connectivity hypotheses.  A quasi-isomorphism between $\mE(P)$ and 
its indecomposibles is then given via the 
chain homotopy argument in \cite[Lemma~2.15]{sinha-walter2}.
This argument is more involved than that for 
Theorem~\ref{BarDerivedIndec}.  While the basic idea of 
``pulling out a generator'' is the same, it requires
an inductive filtration approach.
To our knowledge, this gives the first  explicit
proof of this standard fact.               

\medskip

We now address signs.  Recall the coproduct defined on $\Tr_*$ by excising subsymbols.  The induced  
coproduct in $\cofreeE(W)$ can be written similarly, applying our convention of 
placing elements of $W$ as labels of respective vertices in tree symbols for 
$\cofreeE(W)$.
For example if $a,b,c,d,e\in W$ we have 
\begin{align*} \Delta \tpc[\big] c \tp[\big] {\tpc a \tp b d} e =& \
 c\otimes \tp[\big] {\tpc a \tp b d} e \;+\; 
 \tpc a \tp b d \otimes \tp c e \\ &+\; 
 a \otimes \tpc[\big] c \tp[\big] {\tp b d} e  \;+\; 
 b \otimes \tpc[\big] c \tp[\big] {\tp a d} e. 
\end{align*}
(Recall Example~\ref{ex:trees} for a picture of this tree.)
Writing elements as labels for $\cofreeE(W)$ automatically encodes the 
label-shifting of the coproduct in $\Tr_*(n)$, greatly simplifying 
logistics. 
This is the rooted tree analog of the convention in \cite[Definition 3.3]{Sinha-Walter1} of $W$-labeled directed graphs.

For calculations in the graded setting, it is necessary to incorporate 
Koszul signs.  However, our convention that tree symbols labeled by
$W$ elements correspond to left-to-right ordering in the corresponding 
tensor product $W^{\otimes n}$ 
(as remarked after Proposition~\ref{cofree conilpotent}), means that 
Koszul signs may be computed following the same rules as for tensor algebras.
This stands in great contrast to the Koszul sign challenges of \cite{Sinha-Walter1}.

In particular if
$W$ is graded, then the corolla commutativity relations  
should carry Koszul signs.  For example $\tp a b = (-1)^{|a|\,|b|}\,b\,\tp a$
and also $\tpc a \tp b c = (-1)^{|a|\,|b|}\,\tpc b\tp a c$.
Note that we carefully chose the form of root arrow reversing and root Arnold
written via symbols
in Section~\ref{Eil definition} so as to not require Koszul signs.  This is the
benefit of allowing the root vertex to commute inside corolla products in
our notation.

The Koszul signs for graded corolla commutativity induce similar signs on 
graded coproduct and cobracket maps, also computed identically as 
for tensor algebras.
For example, the graded version of the coproduct shown above has Koszul signs 
as follows.
\begin{alignat*}{3}
\Delta \tpc[\big] c \tp[\big] {\tpc a \tp b d} e &=&&&
 c &\otimes \tp[\big] {\tpc a \tp b d} e \\ 
 &+& (-1)^{|c|\bigl(|a|+|b|+|d|\bigr)} && \ 
      \tpc a \tp b d &\otimes \tp c e\\
 &+&(-1)^{|c||a|}&& \  
      a &\otimes \tpc[\big] c \tp[\big] {\tp b d} e\\
 &+&(-1)^{\bigl(|a|+|c|\bigr)|b|}&& \ 
      b &\otimes \tpc[\big] c \tp[\big] {\tp a d} e. 
\end{alignat*}
The signs above are computed similar to those of the graded unshuffle coproduct in the 
graded tensor Hopf algebra. 

The Harrison complex of a graded algebra also gains Koszul signs.  
In particular the ``contract and multiply'' differential $d_\mu$ involves  
a sign due to a degree $1$ operator, $d_\mu$, moving across a tensor (the same as
in the bar complex), plus additional Koszul signs due to required rearrangements
of a tree symbol.  For example the tree symbol shape used above, appearing in the
Harrison complex, would have $d_\mu$ as follows.
\begin{align*}
d_\mu\Bigl( \tpc[\big] {s^{\inv}c} \tp[\big] {\tpc {s^{\inv}a} \tp {s^{\inv}b} s^{\inv}d} s^{\inv}e\Bigr)
 = & \pm (-1)^{|c|}\, \tp[\big] {\tpc {s^{\inv}a} \tp {s^{\inv}b} s^{\inv}d} s^{\inv}(ce) \\ 
   & \pm (-1)^{|a|}\, \tpc[\big] {s^{\inv}c} \tp[\big] {\tp {s^{\inv}b} s^{\inv}(ad)} s^{\inv}e \\
   &  \pm (-1)^{|b|}\, \tpc[\big] {s^{\inv}c} \tp[\big] {\tp {s^{\inv}a} s^{\inv}(bd)} s^{\inv}e \\
   & \pm (-1)^{|d|}\, \tpc {s^{\inv}c} \tp {s^{\inv}a} \tp {s^{\inv}b}  s^{\inv}(de) 
\end{align*}
where the $\pm$ are Koszul signs from applying graded commutativity of corollas,
reordering the symbol so that the ``contract and multiply''
terms are adjacent; as well as signs from $d_\mu$ moving across 
terms in the symbol to reach the location where it is applied.
The Koszul sign for the first term above is
$(-1)^{\bigl(|c|-1\bigr)\bigl((|a|-1)+(|b|-1)+(|d|-1)\bigr) + 
 \bigl((|a|-1)+(|b|-1)+(|d|-1)\bigr)}$ 
from applying corolla commutativity to move $s^{\inv}c$ next to $s^{\inv}e$ 
and also from
$d_\mu$ crossing terms. The Koszul sign of the last term is
$(-1)^{(|c|-1)+(|a|-1)+(|b|-1)}$ from $d_\mu$ crossing the
$s^{\inv}c$, $s^{\inv}a$, and $s^{\inv}b$ terms.
To get all Koszul signs correct and complete calculations with greatest
confidence, we have found it generally worthwhile to 
explicitly include the desuspensions $s^\inv$ in the Harrison complex.


\section{Thom forms and thickened two-complexes}\label{GraphModels}

We provide necessary background
for Section~\ref{CurveForms}, giving
a graphical approach to calculations
of products and Massey products 
for two-complexes.  This approach
is used throughout Section~\ref{examples}
and was also helpful in formulating most
of our main results. 
We plan to give a general theory of graphical models for
groups in future work.

In Section~\ref{Thomforms} below
we review elementary constructions of de\,Rham cochains associated to submanifolds.  
These allow for geometric access to cochains.  We summarize basic definitions and
properties of Thom forms from \cite{BottTu} and prove that one can find a
Thom form for a manifold with boundary which is compatible with a chosen form on the boundary.

These submanifold tools are applied in Section~\ref{CurveForms} to (non-smooth) 
two-complexes by constructing
four-manifold thickenings $M_G$ which deformation retract onto their sub-two-complexes $X_G$ in Section~\ref{sec:MG}.
This construction is a standard, immediate application of
handlebody theory; though we prefer an approach through real blow-ups which is new.

\subsection{Thom forms}\label{Thomforms}

\begin{definition} \label{DefThom}
Let $W \subseteq M$ be a submanifold of codimension $d$ 
whose normal bundle is oriented and  whose boundary is contained in that of $M$.  A {\bf Thom form} $\tau_W$ is a closed $d$-form on $M$ which is supported on a  tubular neighborhood
of $W$, so that the integral of $\tau_W$ over each transverse slice of the tubular neighborhood is a fixed number called the {\bf mass} of the
Thom form.
\end{definition}


The following is a summary of results from Chapter 6 of \cite{BottTu}.

\begin{theorem}\label{Thom} Let $W\subseteq M$ be any submanifold of codimension $d$ whose normal bundle 
is oriented and whose boundary is contained in that of $M$.
\begin{enumerate}
    \item 
Thom forms for $W$ exist and they are all cohomologous via forms supported on a tubular neighborhood of $W$, so represent a unique de\,Rham cohomology class $[W]\in H^d_{dR}(M\setminus \partial M)$.\label{ThomClass}

\item The value of a Thom form on a smooth $d$-cycle that meets $W$ transversally, or on the fundamental class of some closed $d$-manifold 
mapping to $M$ that is transverse to $W$, is the signed count of the preimages of $W$ multiplied by the mass of the form. \label{eval}

\item The pullback of a Thom form by a smooth map which is transverse to $W$ is a Thom form for the preimage of $W$.

\item The wedge product of Thom forms for submanifolds which intersect transversally is a Thom form of their intersection,
and its mass is the product of their masses. \label{product}
\end{enumerate}
\end{theorem}

For example, let $P(V)$ be the projective space of lines
in a vector space $V$, so $P(\C^{n+1}) \cong
\C P^n$.  The Thom form of any linear $V \subset \C^{n+1}$ represents a generator of $H^{2i}_{dR}(\C P^n)$, where $i$ is the codimension of $V$.  
If $V$ and $W$ intersect transversally we see that $\tau_{P(V)} \wedge \tau_{P(W)} = \tau_{P(V\cap W)}$.
Thus the cohomology ring structure of projective spaces is reflecting basic intersection properties of linear subspaces.

For many applications it is necessary to extend the notion of Thom forms to submanifolds with corners whose boundary is
in the interior of $M$.  (See for example \cite{GeometricCochains} for basic definitions.)

\begin{definition}\label{def:Thom form}
    Let $W \subset M$ be a proper submanifold  of a manifold, possibly with corners modeled by $[0,\infty)^n$ with $\partial^2 W \subset \partial M$.
    Let $\partial^\circ W$ denote the subset  of $\partial W$ in the interior of $M$.

    A {\bf Thom form} for $W$ is a form $\tau_W$ supported in a tubular neighborhood $\nu_W$ of $W$, whose exterior derivative is a Thom form for 
    $\partial^\circ W$,  supported in $\nu_{\partial^\circ W}$.  The 
    integral of $\tau_W$ is the mass over any transverse slice of the tubular neighborhood of $W$
    disjoint from  the support of $d\tau_W$.
\end{definition}

\begin{theorem}\label{Thomboundary}
Let $W$ be a submanifold with corners whose collar neighborhood extends from some 
    $[0,\varepsilon) \times \partial^\circ W \hookrightarrow W$ to $(-\varepsilon, \varepsilon) \times \partial^\circ W \hookrightarrow M$.
Let  $\tau_{\partial^\circ  W}$ be a Thom form for $\partial^\circ W$.  There exists a Thom form $\tau_W$ for $W$ so that $d \tau_W = \tau_{\partial^\circ  W}$.
\end{theorem}

\begin{proof}
We assume a mass of one and assume first that there exists a Thom form 
    $\hat{\tau}_W$ with $d \hat{\tau}_W$ some Thom form for $\partial^\circ W$, likely different from the given $\tau_{\partial^\circ  W}$.  
    Then because $\partial^\circ W$ is a submanifold whose  boundary is in $\partial M$, 
    $d \hat{\tau}_W$ and $\tau_{\partial^\circ  W}$ 
    differ by some $d \nu$ supported near $W$ by Theorem~\ref{Thom}~(\ref{ThomClass}).  Set $\tau_W = \hat{\tau}_W - \nu$, which is again a Thom form for $W$ and has the desired exterior derivative.

   To find such a $\hat{\tau}_W$ make use of the  extension of the collar neighborhood of $\partial^\circ W$. 
     We construct a Thom form $\tilde{\tau}_W$ for a tubular neighborhood $\nu_e$ of the extended submanifold -- $W$ together with the collar $(-\varepsilon,\varepsilon)\times \partial^\circ W$ -- as in Chapter 6 of \cite{BottTu}.  
     Let $t$ be the collar coordinate and $\sigma(t)\maps (-\varepsilon,\varepsilon)\to [0,1]$ be a smooth monotone function with $\sigma\equiv 0$ near $-\varepsilon$ and $\sigma \equiv 1$ near $\varepsilon$.  Let $f : \nu_e \to \R$ denote the function with value  $1$  outside the tubular neighborhood of the extended collar and which on that neighborhood is the projection onto the $(-\varepsilon, \varepsilon)$ coordinate followed by $\sigma$. 
     We  let $\hat{\tau}_W$ equal 
     $f \wedge \tilde{\tau}_W$.  By construction,
     its exterior
     derivative  $d(f \wedge \tilde{\tau}_W) = df \wedge \tilde{\tau}_W$
     is a Thom form for $\partial^\circ W$ and which integrated to $1$ on tranvserve slices outside an $\varepsilon$-neighborhood of $\partial^\circ W$, as desired.
\end{proof}

Thus we can use submanifolds to define de\,Rham cochains, and use these to get concrete models for de\,Rham cochain algebras.   We claim that analogues of Theorem~\ref{Thom}
all hold in the case with boundary, where intersections 
considered in items \ref{eval} and \ref{product} occur
away from the support of the Thom forms of the boundary.
We can then use submanifolds whose boundaries match up to define cocycles, choosing their boundaries to have
 cancelling Thom forms by Theorem~\ref{Thomboundary}.

One can in fact use the above construction to build  models (quasi-isomorphic subalgebras) for the de Rham algebra of a manifold.  The general idea is to first construct cohomology ring generators. Then one can use part (4) of Theorem~\ref{Thom} to calculate products as intersections.  For linear combinations of products which
are trivial, one finds unions submanifolds which cobound them, applying Theorem~\ref{Thomboundary} again.  Further
intersections then give rise to simple Massey products.
The process of cobounding and then accounting for new
intersections can go on indefinitely since we allow
for, and in the present work focus on, codimension one submanifolds.

An important technical point is that since non-trivial submanifolds can never intersect themselves transversally, part (4) of Theorem~\ref{Thom} cannot be applied immediately to understand squaring (in even degrees).  The same can be said for products
involving a manifold and its boundary or two manifolds which
meet at their boundaries.
For calculations, the submanifolds in question need to be ``moved off,''  but this process in turn will produce more intersections to analyze.

Such calculations can be made with integer coefficients using the recently developed
machinery of geometric cochains \cite{GeometricCochains}.

\subsection{Thickening a two-complex into a four-manifold}\label{sec:MG}

The methods in this paper require a cochain model for a space with a given fundamental group, which we informally call a ``cochain model for the group.''  Flavors of models include manifolds of interest, the simplicial model for the classifying space of a group, and two-dimensional CW-complexes.  
We have found the last to be a rich source for examples, intuition and  
technical use, in particular with our Lifting Criterion
Theorem~\ref{LiftingCriterion}
and the universal property of Theorem~\ref{thm-intro:universal property}.  We explicitly construct
smooth manifold models for two-complexes here 
so that we can study their de Rham theory
in 
Section~\ref{CurveForms}. 

\begin{definition}
  Let $G = \langle S\;|\;R\rangle$  be a presentation of a group
  via generators $S$ and relations $R$.
  
  Define $X_G$ 
  to be the $2$-complex whose $1$-skeleton is a wedge
  of circles, one for each generator in $S$, with $2$-disks attached 
  according to each relation in $R$, uniformly parametrized.
\end{definition}

 We will
mimic this construction in surgery theory using handle attachments to 
build $M_G$, a smooth four-manifold with 
boundary, such that $M_G \simeq X_G$.  This is standard, though we use a version of
surgery theory defined via blowup and attaching handles onto 
resulting manifolds with corners since that 
makes it simpler to identify submanifolds.

\begin{proposition}
Let $W$ be an embedded submanifold of $\partial M$ with codimension $d$ in the boundary. Then there is a manifold with corners $M_{\widehat{W}}$, which we call the 
{\bf blow-up of $M$ at $W$}, equipped with a smooth map $\beta: M_{\widehat{W}} \to M$, with the following properties.  \begin{itemize}
    \item $M_{\widehat{W}}$  is homeomorphic to $M$, with interior diffeomorphic to that of $M$.
    \item There  is a closed stratum $\widehat{W} \subset \partial M_{\widehat{W}}$ diffeomorphic to the unit disk bundle of the normal bundle of $W$ in $\partial M$.  Under this identification, the restriction of $\beta$ to $\widehat{W}$ coincides
    with the bundle map projection to $W$.     Outside of $\widehat{W}$, $\beta$ is a diffeomorphism.
\end{itemize}
\end{proposition}

Now consider the blowup when $W\cong S^k$ is a framed, embedded sphere in $\partial M$. The stratum $\widehat{W}$ lies on the boundary $\partial M_{\widehat{W}}$ and as such has a collar neighborhood diffeomorphic to $\widehat{W} \times (-\varepsilon,0]$ extending into the interior of $M_{\widehat{W}}$. In this case the framing gives an identification $\widehat{W} \cong S^k \times D^d$. Observe as well that the boundary of $D^{k+1}\times D^d$ as
 a manifold with corners has a component $S^k \times D^d$ which has a collar neighborhood diffeomorphic to $(S^k \times [0,\varepsilon)) \times D^d \cong  (S^k\times D^d)\times [0,\varepsilon) = \widehat{S^k} \times [0,\varepsilon)$.  
 
 \begin{definition}[Handle attachment]\label{def:handle attachement}
 If $W$ is a framed, embedded sphere in $\partial M$, the {\bf handle attachment along $W$} is the identification of $M_{\widehat{W}}$ and $D^{k+1}\times D^d$ along $S^k \times D^d \cong \widehat{W}$, with smooth structure defined through 
 $$\left(\widehat{W} \times (-\varepsilon,0]\right) \ \cup\ \left(S^k \times D^d \times [0,\varepsilon)\right) /_\sim\ \cong\ \widehat{W} \times (-\varepsilon,\varepsilon).$$
 \end{definition} 


 Note that there is a similar construction for handle attachment to the interior, where the blow-up is a manifold with boundary and the identifications have tautological smooth structure once collar neighborhoods are chosen.

 We apply these ideas to construct a manifold $M_G$ realizing a  group presentation $G = \langle S \; | \; R \rangle$.  

\begin{definition}
Let $M_0$ be a four dimensional ball $M_0 \cong D^4$, i.e. a $0$-handle.  

For each generator $s\in S$, choose an embedding
of a $0$-sphere $S^0 = \{+,-\}$ in the boundary of $M_0$.  Blow-up at each 
of these points to obtain 
$({M_0})_{\widehat{S}}$, a manifold with corners, homeomorphic to $D^4$.
\end{definition}

The interior of $({M_0})_{\widehat{S}}$ is a
standard open four-ball. Its boundary as a
manifold with corners now has
a ``large''  component we call $\partial_0$ diffeomorphic to the complement in $S^3$ of a union
of open disks, two for each $s\in S$.
There are also ``small'' boundary components contributed by each blowup.
For each generator $s \in S$, write $\partial^{\pm}_s$ for the 
boundary components of $({M_0})_{\widehat{S}}$ produced by the blowup of
the corresponding $S^0$ points embedded in $M_0$.
Note that the $\partial^{\pm}_s$ are diffeomorphic to 3-disks, each
meeting the ``large'' boundary component $\partial_0$ at a copy of $S^2$. 
These 2-spheres constitute the codimension-two corners of $({M_0})_{\widehat{S}}$.




Each $\partial^{\pm}_s$ has a neighborhood in $({M_0})_{\widehat{S}}$ which
is diffeomorphic, as manifolds with corners, to 
$D^3\times(-\varepsilon,0]$.  Thus after choosing an explicit identification 
with $D^3$, we may perform the following.

\begin{figure}
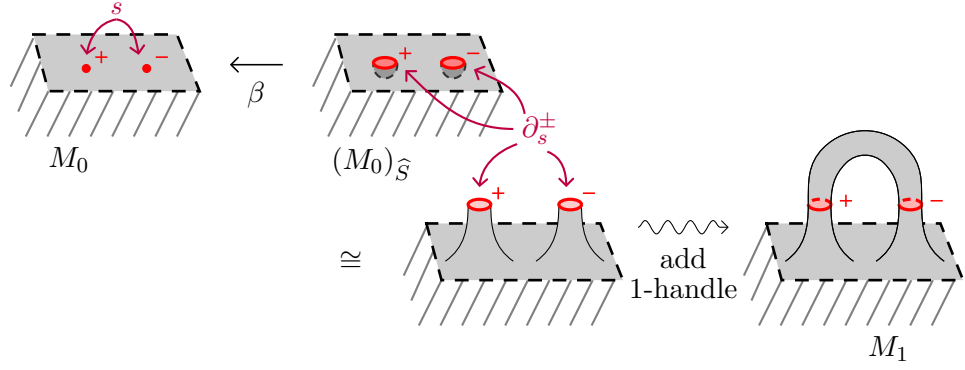

    \include{TikzFigures/surgery1}
    \vspace{-1.2cm}
    \caption{Attaching a 1-handle}
\end{figure}
 
\begin{definition}
    Let $M_S$ be  the quotient of
$({M_0})_{\widehat{S}} \;\coprod\;(\bigcup_S D^3\times D^1)$ smoothly attaching 
1-handles at each $\partial^{\pm}_s$.
\end{definition}

The result is a manifold with boundary but no corners, as the corners became
``internal'' once we glued the $D^3\times (-\varepsilon,0]$ sitting inside $(M_0)_{\hat S}$ to the subspace $D^3\times [0,\varepsilon)$ of the new handle, 
resulting in $D^3\times (-\varepsilon,\varepsilon)$ in the quotient.  We call the slice $0 \times D^1 \subseteq D^3\times D^1$ the {\bf core}
of the $1$-handle, and the images of its endpoints $0 \times 0$ and $0 \times 1$ in $(M_0)_{\hat S}$
are the outgoing and incoming {\bf handle attachment points} respectively.
The manifold with boundary $M_S$ retracts onto the wedge product $\bigvee_S S^1$ through the 
homotopy equivalence sending $M_0$ to a point and each 1-handle to its core.

Next, for each relation $r\in R$, we claim there exists a framed embedding of $S^1$ into the boundary 
$\partial M_S $ such that its image 
in $\pi_1(M_S) \cong \langle S\rangle$ is the word $r$, and these can be made disjoint.
Disjoint embeddings are generic because 
$\partial M_S$ is three-dimensional.  To have canonical framings we choose such an embedding so that its image in any of the 
one handles is a union of segments of the form $x \times D^1$ for some $x \in \partial D^3$, so that a choice of frame at either end determines a ``product framing.''   
Then we can use the fact that $\partial M_0 \cong S^3$ has a trivial tangent bundle to transport framings when the image is in the ``large'' boundary component $\partial_0$.  

\begin{definition}
Fix a collection of disjoint framed embeddings of $S^1$ into $M_S$, corresponding to relations $r\in R$ as in the last paragraph. The blowup of $M_S$ along these embedded circles will be called $(M_S)_{\widehat R}$. 
\end{definition}

$(M_S)_{\widehat R}$ is a manifold with corners homeomorphic to $M_S$. Its boundary consists of a ``large'' boundary component 
$\partial_1$ coming
from $\partial M_S$ without the blown-up circles, as well as ``small'' boundary components $\{\partial_r\}$,
diffeomorphic to $S^1\times D^2$ (shown in red in Figure \ref{fig:2-handle}), contributed by each individual blowup.
The $\partial_r$ and $\partial_1$ components 
connect at codimension 2 corners in $(M_1)_{\widehat R}$ (each diffeomorphic to the 
boundary of a 2-torus).
Each $\partial_r$ has a ``collar'' neighborhood in $(M_1)_{\widehat R}$ 
which is diffeomorphic,
as manifolds with corners, to $S^1\times D^2 \times (-\varepsilon,0]$.

Choosing such a diffeomorphism, we can smoothly ``fill in the relation $r$'', attaching a $2$-handle $D^2\times D^2$ by gluing along the identification 
$(\partial D^2)\times D^2 = S^1 \times D^2$.

\begin{figure}
    \begin{tikzpicture}

%
%

\draw[black!50,thick] (0/7,.12) -- (-.3+0/7,-.48); 
\draw[black!50,thick] (2/7,.12) -- (-.3+2/7,-.48);
\draw[black!50,thick] (4/7,.12) -- (-.3+4/7,-.48)node[below,black]{$M_1$};
\draw[black!50,thick] (6/7,.12) -- (-.3+6/7,-.48);
\draw[black!50,thick] (8/7,.12) -- (-.3+8/7,-.48);
\draw[black!50,thick] (10/7,.12) -- (-.3+10/7,-.48);
\draw[black!50,thick] (12/7,.12) -- (-.3+12/7,-.48);
\draw[black!50,thick] (14/7,.12) -- (-.3+14/7,-.48);

\draw[black!50,thick] (-.1,.52) -- (-.3-.1,-.48+.33);
\draw[black!50,thick] (-.2,.82) -- (-.3-.2,-.48+.66);

\draw[fill=black!20,dot diameter=1pt,dot spacing=5pt,dots] 
  (-.2,.87) -- (1.7,.87) 
           -- (2.0,.12) 
           -- (0,.12) 
           -- cycle;

\draw[red,thick] (1.30,.49)  arc[start angle=0, end angle=360, x radius=.4, y radius=.15];
\node[color=purple,scale=.8] at (1.45,.67) {$r$};

\draw[-angle 90,semithick] (3.1,.47) -- (2.4,.47);
\node at (2.75,.07) {$\beta$};

%
%

\draw[black!50,thick] (4.1+0/7,.12) -- (3.8+0/7,-.48); 
\draw[black!50,thick] (4.1+2/7,.12) -- (3.8+2/7,-.48);
\draw[black!50,thick] (4.1+4/7,.12) -- (3.8+4/7,-.48)node[below,black]{$(M_1)_{\widehat R}$};
\draw[black!50,thick] (4.1+6/7,.12) -- (3.8+6/7,-.48);
\draw[black!50,thick] (4.1+8/7,.12) -- (3.8+8/7,-.48);
\draw[black!50,thick] (4.1+10/7,.12) -- (3.8+10/7,-.48);
\draw[black!50,thick] (4.1+12/7,.12) -- (3.8+12/7,-.48);
\draw[black!50,thick] (4.1+14/7,.12) -- (3.8+14/7,-.48);

\draw[black!50,thick] (4.0,.52) -- (3.8-.1,-.48+.33);
\draw[black!50,thick] (3.9,.82) -- (3.8-.2,-.48+.66);

\draw[fill=black!20,dot diameter=1pt,dot spacing=5pt,dots] 
  (3.9,.87) -- (5.8,.87) 
           -- (6.1,.12) 
           -- (4.1,.12) 
           -- cycle;

\tikzset{decoration={
markings,
mark=at position 0.1 with {\coordinate[scale=.1] (A) [fill=blue] {};},
mark=at position 0.2 with {\coordinate[scale=.1] (B) [fill=blue] {};},
mark=at position 0.3 with {\coordinate[scale=.1] (C) [fill=red] {};},
mark=at position 0.4 with {\coordinate[scale=.1] (D) [fill=red] {};},
mark=at position 0.5 with {\coordinate[scale=.1] (E) [fill=red] {};},
mark=at position 0.6 with {\coordinate[scale=.1] (F) [fill=red] {};},
mark=at position 0.7 with {\coordinate[scale=.1] (G) [fill=red] {};},
mark=at position 0.8 with {\coordinate[scale=.1] (H) [fill=red] {};},
mark=at position 0.9 with {\coordinate[scale=.1] (I) [fill=red] {};},
mark=at position 1 with {\coordinate[scale=.1] (J) [fill=red] {};},
}}

\draw[red,thick,xshift=4.9,postaction={decorate},fill=red!20] (5.45,.5)  arc[start angle=0, end angle=360, x radius=.56, y radius=.245];
\draw[red,thick] (5.45,.5)  arc[start angle=0, end angle=360, x radius=.4, y radius=.15];

\draw[red,out=-90,in=90] (A) to (5.05,.5);
\draw[red,out=-90,in=45] (B) to (4.85,.5);
\draw[red,out=-90,in=135] (C) to (5.25,.5);
\draw[red,out=-90,in=90] (D) to (5.05,.5);
\draw[red,out=-45,in=200] (E) to (4.7,.5);
\draw[red,out=0,in=45] (F) to (4.8,.5);
\draw[red,out=45,in=135] (G) to (5.05,.5);
\draw[red,out=135,in=-45] (H) to (5.05,.5);
\draw[red,out=180,in=135] (I) to (5.3,.5);
\draw[red,out=225,in=-20] (J) to (5.4,.5);

\draw[fill=black!20,draw=none] (5.05,.5)  circle[start angle=0, end angle=360, x radius=.385, y radius=.135];

%
%

\draw[black!50,thick] (4.8+0/7,-2.38) -- (4.5+0/7,-2.98); 
\draw[black!50,thick] (4.8+2/7,-2.38) -- (4.5+2/7,-2.98);
\draw[black!50,thick] (4.8+4/7,-2.38) -- (4.5+4/7,-2.98);
\draw[black!50,thick] (4.8+6/7,-2.38) -- (4.5+6/7,-2.98);
\draw[black!50,thick] (4.8+8/7,-2.38) -- (4.5+8/7,-2.98);
\draw[black!50,thick] (4.8+10/7,-2.38) -- (4.5+10/7,-2.98);
\draw[black!50,thick] (4.8+12/7,-2.38) -- (4.5+12/7,-2.98);
\draw[black!50,thick] (4.8+14/7,-2.38) -- (4.5+14/7,-2.98);

\draw[black!50,thick] (4.7,-2.38+.40) -- (4.4,-2.98+.33);
\draw[black!50,thick] (4.6,-2.38+.70) -- (4.3,-2.98+.66);

\draw[fill=black!20,dot diameter=1pt,dot spacing=5pt,dots] 
  (4.6,.-1.63) -- (6.6,-1.63) 
           -- (7.0,-2.38) 
           -- (4.8,-2.38) 
           -- cycle;

\node at(3.75,-2.10) {$\cong$};

\draw[draw=none,fill=black!20] 
  (4.85,-2.00) to[out=35, in=-90] (5.15,-1.35) 
               to[out=25, in=155] (6.27,-1.35)
               to[out=-90, in=145] (6.57,-2.00);
               
\draw (4.85,-2.00) to[out=35, in=-90] (5.152,-1.35);
\draw (6.27,-1.35) to[out=-90,in=145] (6.57,-2.00);
\draw[out=50,in=-100](5.2,-2.10) to (5.4,-1.45);
\draw[out=65,in=-90](5.6,-2.20) to (5.7,-1.35);
\draw[out=145,in=-87](6.3,-2.20) to (6.0,-1.55);

\draw[red,thick,fill=red!20,xshift=4.9,postaction={decorate}] (6.1,-1.35)  arc[start angle=0, end angle=360, x radius=.56, y radius=.245];
\draw[red,thick] (6.1,-1.35)  arc[start angle=0, end angle=360, x radius=.4, y radius=.15];

\draw[red,out=-90,in=90] (A) to (5.05+.65,-1.35);
\draw[red,out=-90,in=45] (B) to (4.85+.65,-1.35);
\draw[red,out=-90,in=135] (C) to (5.25+.65,-1.35);
\draw[red,out=-90,in=90] (D) to (5.05+.65,-1.35);
\draw[red,out=-45,in=200] (E) to (4.7+.65,-1.35);
\draw[red,out=0,in=45] (F) to (4.8+.65,-1.35);
\draw[red,out=45,in=135] (G) to (5.05+.65,-1.35);
\draw[red,out=135,in=-45] (H) to (5.05+.65,-1.35);
\draw[red,out=180,in=135] (I) to (5.3+.65,-1.35);
\draw[red,out=225,in=-20] (J) to (5.4+.65,-1.35);

\draw[draw=none,fill=black!40] (5.7,-1.35)  circle[start angle=0, end angle=360, x radius=.385, y radius=.135];

\draw[black!80,line width=.08mm] (5.4,-1.27) to (5.45,-1.45);
\draw[black!80,line width=.08mm] (5.6,-1.21) to (5.62,-1.49);
\draw[black!80,line width=.08mm] (5.8,-1.21) to (5.78,-1.49);
\draw[black!80,line width=.08mm] (6.0,-1.27) to (5.95,-1.45);


\node[purple,] at(6.5,-.4) {$\partial_r$};
\draw[-ang 90,thick,purple] (6.35,-.2) to[out=120,in=-20] (5.55,.35);
\draw[-ang 90,thick,purple](6.4,-.6) to[out=-90,in=25] (6.1,-1.1);

\tikzset{snake it/.style={decorate, decoration=snake,segment length=3mm}}
\draw[snake it,-angle 90] (7.1,-1.9) -- (8.3,-1.9);
\node at (7.7,-2.3) {add};
\node at(7.7,-2.7) {2-handle};


%
%

\draw[black!50,thick] (8.95+0/7,-2.38) -- (8.65+0/7,-2.98); 
\draw[black!50,thick] (8.95+2/7,-2.38) -- (8.65+2/7,-2.98);
\draw[black!50,thick] (8.95+4/7,-2.38) -- (8.65+4/7,-2.98);
\draw[black!50,thick] (8.95+6/7,-2.38) -- (8.65+6/7,-2.98);
\draw[black!50,thick] (8.95+8/7,-2.38) -- (8.65+8/7,-2.98);
\draw[black!50,thick] (8.95+10/7,-2.38) -- (8.65+10/7,-2.98);
\draw[black!50,thick] (8.95+12/7,-2.38) -- (8.65+12/7,-2.98)node[below,black]{$M_G$};
\draw[black!50,thick] (8.95+14/7,-2.38) -- (8.65+14/7,-2.98);
\draw[black!50,thick] (8.95+16/7,-2.38) -- (8.65+16/7,-2.98);

\draw[black!50,thick] (8.85,-2.38+.40) -- (8.55,-2.98+.33);
\draw[black!50,thick] (8.75,-2.38+.70) -- (8.45,-2.98+.66);

\draw[fill=black!20,dot diameter=1pt,dot spacing=5pt,dots] 
  (8.75,.-1.63) -- (11.05,-1.63) 
           -- (11.45,-2.38) 
           -- (8.95,-2.38) 
           -- cycle;

\draw[fill=black!20] (9.0,-2.00) to[out=20,in=180] (10.05,-.65) to[out=0,in=160] (11.1,-2.00);

\draw[draw=none,fill=red!30]
    (10.63,-1.35) to[out=115,in=0]   (10.05,-1.10)
                  to[out=180,in=55]  (9.47,-1.35);
                  
\draw[densely dashed,thick] (10.05,-1.35) ellipse (.36 and .48);

\draw[draw=none,fill=red!20] 
     (9.47,-1.35) to[out=-55,in=180] (10.05,-1.60) 
                  to[out=0,in=235]   (10.63,-1.35);
                  
\draw[densely dashed,red,fill=black!20] (9.68,-1.35) arc (180:360:.37 and 0.1);
\draw[densely dashed,red,fill=black!20] (10.42,-1.35) arc (0:180:.37 and 0.1);

\draw[red,thick] (9.47,-1.35) to[out=-55,in=180] (10.05,-1.60) to[out=0,in=235] (10.63,-1.35);
\draw[red,thick,densely dashed] (10.63,-1.35) to[out=115,in=0] (10.05,-1.10) to[out=180,in=55] (9.47,-1.35);
  
\end{tikzpicture}

    \vspace{-1.0cm}
    \caption{Attaching a 2-handle}\label{fig:2-handle}
\end{figure}
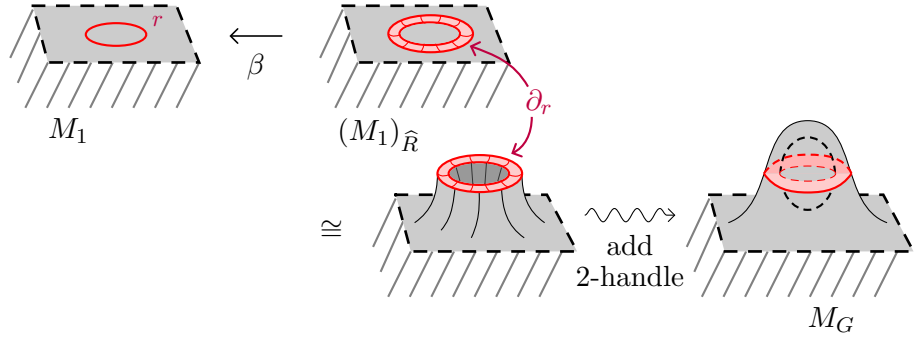

\begin{definition}\label{MG}
Define ${M}_G$ to be the manifold resulting from $(M_S)_{\widehat R}$ 
by attaching a $2$-handles associated to each relation $r\in R$, as described in Definition \ref{def:handle attachement}.  
\end{definition}

Note that $M_G$ has boundary but no
corners. Furthermore, $M_G$ deformation retracts onto the presentation 2-complex $X_G$ by having 
each handle retract onto its core cell.
With our construction through blow-ups, it is simple
to define smooth submanifolds, as needed
for our de\,Rham models, as unions
of cores within handles.  In particular, the extended indicator
submanifolds of Section~\ref{CurveForms}
with their ``tabs'' would be difficult
to specify in standard approaches to smooth
surgery.

\bibliographystyle{alpha}
\bibliography{bibliography}

\end{document}